\documentclass[11pt]{article}

\usepackage{amsmath, amsfonts, amssymb, amsthm,  graphicx, mathtools, enumerate,bbm}

\usepackage[blocks, affil-it]{authblk}

\usepackage{mathrsfs}
\usepackage{bbm}

\usepackage[round]{natbib}
\usepackage[CJKbookmarks=true,
            bookmarksnumbered=true,
			bookmarksopen=true,
			colorlinks=true,
			citecolor=red,
			linkcolor=blue,
			anchorcolor=red,
			urlcolor=blue]{hyperref}
\usepackage{hyperref}
\usepackage[usenames]{xcolor}

\usepackage[letterpaper, left=1.2truein, right=1.2truein, top = 1.2truein, bottom = 1.2truein]{geometry}
\usepackage[ruled, vlined, lined, commentsnumbered]{algorithm2e}

\usepackage{prettyref}
\usepackage{soul}
\usepackage[font=small,labelfont=it]{caption}

\newtheorem{lemma}{Lemma}[section]

\newtheorem{thm}{Theorem}[section]

\newtheorem{remark}{Remark}[section]

\newrefformat{eq}{(\ref{#1})}
\newrefformat{chap}{paper~\ref{#1}}
\newrefformat{sec}{Section~\ref{#1}}
\newrefformat{algo}{Algorithm~\ref{#1}}
\newrefformat{fig}{Fig.~\ref{#1}}
\newrefformat{tab}{Table~\ref{#1}}
\newrefformat{rmk}{Remark~\ref{#1}}
\newrefformat{clm}{Claim~\ref{#1}}
\newrefformat{def}{Definition~\ref{#1}}
\newrefformat{cor}{Corollary~\ref{#1}}
\newrefformat{lmm}{Lemma~\ref{#1}}
\newrefformat{lemma}{Lemma~\ref{#1}}
\newrefformat{prop}{Proposition~\ref{#1}}
\newrefformat{app}{Appendix~\ref{#1}}
\newrefformat{ex}{Example~\ref{#1}}
\newrefformat{exer}{Exercise~\ref{#1}}
\newrefformat{soln}{Solution~\ref{#1}}
\newrefformat{cond}{Condition~\ref{#1}}

\newcommand\e{\varepsilon}
\renewcommand\d{\delta}

\newcommand\R{\mathbb{R}}

\renewcommand\P{\mathbb{P}}
\newcommand\E{\mathbb{E}}

\newcommand\var{\mathrm{Var}}
\newcommand\cov{\mathrm{Cov}}

\newcommand{\TV}{\mathrm{TV}}

\newcommand{\PBJ}{\mathrm{PBJ}}
\newcommand{\LRT}{\mathrm{LRT}}

\newcommand{\T}{\mathcal{T}}

\renewcommand{\th}{\mathrm{th}}

\newcommand{\KL}{\mathrm{KL}}

\newcommand{\iid}{\mathrm{iid}}

\title{Sharp phase transitions in high-dimensional changepoint detection}

\author{Daniel Xiang}
\author{Chao Gao\thanks{The research of CG was supported in part by NSF Grants ECCS-2216912 and
DMS-2310769, NSF Career Award DMS-1847590, and an Alfred Sloan fellowship.}}
\affil{
Department of Statistics, University of Chicago
}

\begin{document}
\maketitle

\begin{abstract}
We study a hypothesis testing problem in the context of high-dimensional changepoint detection. Given a matrix $X \in \R^{p \times n}$ with independent Gaussian entries, the goal is to determine whether or not a sparse, non-null fraction of rows in $X$ exhibits a shift in mean at a common index between $1$ and $n$.  We focus on three aspects of this problem: the sparsity of non-null rows, the presence of a single, common changepoint in the non-null rows, and the signal strength associated with the changepoint.  Within an asymptotic regime relating the data dimensions $n$ and $p$ to the signal sparsity and strength, the information-theoretic limits of this testing problem are characterized by a formula that determines whether or not there exists a testing procedure whose sum of Type I and II errors tends to zero as $n,p \to \infty$. The formula, called the \emph{detection boundary}, partitions the parameter space into a two regions: one where it is possible to detect the presence of a single aligned changepoint (detectable region), and another where no test is able to consistently distinguish the mean matrix from one with constant rows (undetectable region). 

\smallskip

\end{abstract}

\section{Introduction}
The detection of a changepoint that occurs somewhere along a single sequence of numbers is a problem that has been studied since at least the middle of the $20^{\mathrm{th}}$ century (\cite{page1955test}), and interest in its extension to multiple sequences has grown over the past two decades. Some of the recent interest can be attributed to technological advances in devices that provide measurements in the form of multidimensional data. An event of scientific relevance may be coded within a data set as one or more changepoints along a particular dimension of measurement, as in the following applications:

\begin{enumerate}
\item \emph{Genetics.} 
Array comparative genomic hybridization (aCGH) is a tool used in human genetics to measure abnormalities with respect to a reference sample (\cite{shah2007modeling}). For each individual, a sequence of fluorescence intensities is measured; each element of the sequence corresponds to a distinct location on the individual's genome. The data can be represented by a $p \times n$ matrix of intensity measurements, where $n$ is the number of individuals, and $p$ is the number of locations on the genome where a measurement was made. Some regions along the genome may exhibit similar behavior among a sub-collection of the individuals, possibly related by common presence of a disease. It is then of interest to determine where, along the genome location dimension $p$, the average intensity of the signal becomes elevated or diminished. An approach using false discovery rates to analyze copy number variation (CNV) data was developed by \cite{efron2011false}, and Higher Criticism type test statistics for CNV analysis were proposed by \cite{jeng2013simultaneous}. \cite{wang2018high}  project each column of a CNV dataset from \cite{bleakley2011group} in the direction of the mean intensity change, and apply a univariate changepoint estimation procedure to the resulting single time series.

\item \emph{Neuroscience.} A ``resting state scan'' involves a subject whose brain at rest is imaged for a period of time, and these data are used to understand how processes of the brain behave in the absence of external stimuli (\cite{aston2012evaluating}). Certain analytical techniques require the assumption of stationarity of these measurements in time, but the extent to which this assumption is justified is not understood precisely (\cite{cole2010advances}); changepoint detection methods provide an opportunity to assess deviations from stationarity. Functional magnetic resonance image (fMRI) is a tool used to measure a signal at each voxel (volume pixel) at a discretized set of points in time. In this setting, the data can be represented as a $p \times n$ matrix, where $n$ is the number of discrete time points, and $p$ is the number of voxels at which a sequence of measurements was taken. \cite{aston2012detecting} develop changepoint detection methods using ideas from principal components analysis to detect departures from stationarity in fMRI data.

\end{enumerate}
We consider a high-dimensional changepoint testing problem where the $p \times n$ observation matrix $X$ can be regarded as a mean matrix $\theta = (\theta_{jt})_{j \in [p],t \in [n]} \in \mathbb{R}^{p \times n}$, corrupted by additive independent standard Gaussian noise in each entry.  The goal is to test the null hypothesis that each row of $\theta$ is constant, against the alternative that at a common time $t^* \in [n-1]$, a subset of the rows of cardinality $s$ undergo a change in mean.  The difficulty of this problem is characterized by the problem parameters $p$, $n$ and $s$, as well as a signal strength parameter $\rho$ that reflects the size of the change. Informally then, the testing problem can be written:
\begin{align*}
&H_0: \hspace{.1em} \theta_{j:} \in \R^{n} \text{ is a constant vector for each $j \in [p]$.} \\
&H_1: \hspace{.1em} \text{$s \in [p]$ of the rows in $\theta$ share a changepoint $t^* \in [n-1]$ with signal strength $\rho$},
\end{align*}
where under $H_1$, there are $s$ many rows of $\theta$ that are constant over columns $i=1,\dots,t^*$ and equal to a different constant after $t^*$. For a given level of sparsity $s\coloneqq |S| \in [p]$ of non-null rows, our ability to distinguish between the hypotheses $H_0$ and $H_1$ is monotone in the signal size $\rho$, which in this paper is defined as the minimum absolute coordinate-wise difference between the mean vector before and after the changepoint:
\begin{align*}
    \rho = \sqrt{\frac{t^* (n-t^*)}{n}} \min\left\{|\theta_{j,\text{before}}-\theta_{j,\text{after}}|: j\in [p] \text{ and } \theta_{j,\text{before}}\neq \theta_{j,\text{after}} \right\}.
\end{align*}

A critical signal strength $\rho^*$ determines when the problem is feasible in the following sense: when $\rho < \rho^*$, no hypothesis testing procedure can reliably distinguish between $H_0$ and $H_1$; when $\rho > \rho^*$, a powerful procedure exists with asymptotically negligible type 1 error. 
In this paper, the statistical difficulty of the changepoint detection problem is formulated via the minimax testing risk, which measures the smallest worst case testing error among all hypothesis testing procedures, defined formally in \eqref{minmaxerr}.

A {\em detection boundary} (\cite{donoho2004higher}) or \emph{critical radius} is a formula for $\rho^*$ that separates the parameter space $\Theta_1(\rho,s)$ into a detectable region (alternatives that can be reliably distinguished from the null) and an undetectable region (those that cannot). To derive this formula for the changepoint setting, we assume an asymptotic relationship between the dimensions of the data and the sparsity of signals. The relationship is defined by a pair $(a,\beta) \in (0,\infty) \times (0,1)$, where the parameter $a$ controls the number of potential changepoint locations and $\beta$ controls the signal sparsity:
\begin{equation}
\label{3log}
\begin{aligned}
\log\log\log n &\sim a\log p \\ 
s &\sim p^{1-\beta}, 
\end{aligned}
\end{equation}
where $a_p \sim b_p$ means $a_p/b_p \to 1$ as $p \to \infty$. The relationship \eqref{3log} is inspired by the testing rate derived by \cite{liu2021minimax}, in which the term $p\log\log(8n)$ appears as an inflated problem dimension. To obtain a tradeoff between $p$ and $n$, we consider an asymptotic regime where $\log\log n$ is polynomial in $p$, or equivalently, \eqref{3log} holds for some $a>0$. In this asymptotic setting, our goal is to derive an expression for the critical radius $\rho^*(a,\beta)$ and to construct an adaptive test 
for which the sum of the Type~I and Type~II error probabilities converges to zero whenever $\rho > \rho^*(a,\beta)$.

Our main result is that when the direction of the change is known, e.g. decreasing so that $\theta_{js} \geq \theta_{jt}$ for $s\leq t^* \leq t$, the detection boundary is
\begin{align}
\label{eq:one-sided-radius}
\rho^*(a,\beta)_{\mathrm{1-side}}^2 &:= \begin{cases}
p^{a-(1-2\beta)} \hspace{1em} &a \leq 1-2\beta \\
\bigl(a-(1-2\beta)\bigr) \log p &1-2\beta < a \leq 1-4\beta/3 \\
2(\sqrt{1-a}-\sqrt{1-a-\beta})^2 \log p &1-4\beta/3 < a \leq 1-\beta \\
p^{a-(1-\beta)} &a > 1-\beta.
\end{cases}
\end{align}
A slightly different formula for the detection boundary arises when the direction of the change is unknown (Theorem \ref{main2}), which differs from the above formula only in the first case $(a \leq 1-2\beta)$.

The key methodological ideas used in this paper for proving feasibility throughout the detectable region are inspired by those found in \cite{chan2015optimal} and \cite{liu2021minimax}. Our test involves the penalization of a Berk--Jones type statistic (\cite{berk1979goodness}) that aggregates a statistic computed from each candidate changepoint location in a geometrically growing grid of locations. 
This feasibility result is complemented by an impossibility result in Section \ref{calclow}, where a least favorable prior distribution is constructed for each $(\rho,a,\beta)$ in the undetectable region. 

\subsection{Related literature in changepoint detection}
\label{sec-lit-review-changepoint}

Dating back to the middle of the $20^{\mathrm{th}}$ century, \cite{page1954continuous} studied changepoint detection for a single sequence of data. Since then, changepoint detection in a single data source has been the subject of intensive research, as surveyed in, for example, \cite{horvath2014}. For the single sequence problem with multiple changepoints, \cite{killick2012optimal} give an approach (PELT) based on penalized likelihood and minimum description length. An extension was studied by \cite{horvath1999testing}, who tested for changes in the mean in multiple data sequences with a dependence structure. In a similar setting, \cite{bai2010common} proposed a least squares method for estimating a change in the mean, and a method based on quasi-maximum likelihood that can also detect a change in the variance.

More recently, methods were developed for changepoint estimation in the high-dimensional setting when the fraction of sequences exhibiting a change is sparse, which is the setting of the current work. Examples based on cumulative summation techniques include \cite{cho2015multiple} and \cite{wang2018high}, the latter of which was extended to handle missing observations by \cite{follain2022high}. The idea of \cite{kovacs2023seeded}, \cite{liu2021minimax}, \cite{kovacs2020optimistic} and \cite{li2023robust} is to check for a change on a geometrically growing grid of candidate locations. This idea yields optimal procedures in a minimax sense; an overview is given by \cite{yu2020review}, \cite{verzelen2023optimal} and \cite{pilliat2023optimal}. \cite{enikeeva2019high} study a high-dimensional changepoint problem very similar to the one considered here, but assume an asymptotic regime yielding a different formula for the critical radius.

The focus of this paper is on offline changepoint detection, where the practitioner sees the entire dataset prior to determining whether or not a change has occurred; the corresponding online problem, where one observes data sequentially and seeks to declare a change as soon as possible after it has occurred, was studied by \cite{xie2013sequential}, \cite{chen2022high} and \cite{chen2024inference}.

\subsubsection{Organization of the paper} 
Our main results are stated in Section \ref{mainresults} for the asymptotic regime \eqref{3log}, and the detection boundary for another asymptotic regime is given in Section \ref{liudiscuss}. We show various connections to existing results in the signal detection literature in Section \ref{comparisons}. An overview of the lower bound construction and the test statistic that achieves the upper bound are described in Section \ref{overview}. All technical arguments are deferred to the supplemental file \cite{xiang2025supplement}.


\subsubsection{Notation} For $p \in \mathbb{N}$, we write $[p] \coloneqq \{1,\dots,p\}$. Given $a,b\in\R$, we write $a\vee b \coloneqq \max(a,b)$ and $a\wedge b \coloneqq \min(a,b)$.  For two positive sequences $a_n$ and $b_n$, we write $a_n\lesssim b_n$ and $a_n \leq O(b_n)$ to mean that there exists a constant $C > 0$ independent of $n$ such that $a_n\leq Cb_n$ for all $n$; moreover, $a_n \asymp b_n$ means $a_n\lesssim b_n$ and $b_n\lesssim a_n$. We overload notation by using $a_n \sim b_n$ to denote $a_n/b_n \to 1$, and $X \sim F$ to mean the random variable $X$ is distributed according to $F$; the intended usage of ``$\sim$'' will be made clear from context. We write $a_n = o(b_n)$ when $\lim_{n\to\infty} a_n/b_n = 0$ and $a_n = \omega(b_n)$ when $\liminf_{n\to\infty} a_n/b_n = \infty$. For a set $S$, we use $\textbf{1}_{S}$ and $|S|$ to denote its indicator function and cardinality respectively. For a vector $v = (v_1,\ldots,v_d) \in\mathbb{R}^d$, the squared norm is denoted $\|v\|^2 \coloneqq \sum_{\ell=1}^d v_\ell^2$. For a matrix $A \in \R^{p \times n}$ we let $A_{j:}$ denote its $j^{th}$ row, and $A_{:i}$ denotes its $i^{th}$ column. The notation $A_{j,1:t} \in \R^t$ denotes the first $t$ elements of the $j^{th}$ row in $A$. We use $\phi$ and $\Phi$ to denote the standard normal density and CDF, respectively, and put $\bar{\Phi}\coloneqq 1-\Phi$. The notation $\mathbb{P}$ and $\mathbb{E}$ are generic probability and expectation operators. In the context of testing $H_0$ against $H_1$, the notation $\P_0,\E_0,\var_0$ and $\P_1,\E_1,\var_1$ refer to the probability, expectation, and variance operators under $H_0$ and $H_1$ respectively. The letter $Z$ denotes a generic standard Gaussian random variable, whose dimensions will be made clear from context.

\section{Main results}
\label{mainresults}

Recall that our observation matrix is of the form
\begin{align}
\label{eq:observation-matrix}
X = \theta + E \in \R^{p \times n},
\end{align}
where the entries $(E_{jt})_{j \in [p],t \in [n]}$ of $E$ are independent standard Gaussians, and the mean matrix $\theta$ either has constant rows, or contains a shift in mean for a subset of the rows at a common column index. Formally, the null scenario is that each row of $\theta$ is constant, i.e. it belongs to 
\begin{align*}
\Theta_0(p,n) &\coloneqq \bigl\{ \theta = (\theta_{jt}) \in \R^{p \times n} : \exists \mu_1,\ldots,\mu_p \in \R \text{ s.t. $\theta_{jt}=\mu_j$ for all $j \in [p]$, $t \in [n]$}\bigr\}.
\end{align*}
We first consider the one-sided alternative where the change at $t^*$ must be strictly decreasing.

\subsection{One-sided changepoint} For a signal strength $\rho > 0$ and sparsity level $s \in [p]$, the ``one-sided'' alternative hypothesis parameter space is defined
\begin{align}
\nonumber
\Theta_1^{\mathrm{1-side}}(p,n,\rho,s) &\coloneqq  \bigg\{\theta \in \R^{p \times n}: \exists S \subseteq [p], \mu_{j1},\mu_{j2} \in \R, t^* \in [n-1] \text{ such that } |S|\leq s \text{ and } \\
\nonumber
&\sqrt{\frac{t^*(n-t^*)}{n}}(\mu_{j1}-\mu_{j2}) \geq \rho \text{ for every $j\in S$, where $\theta_{ji}= \mu_{j1}$ for $i\leq t^*$,} \\
\label{1sidecp}
&\text{and $\theta_{ji}=\mu_{j2}$ for $i > t^*$} \bigg\}.
\end{align} 
The factor $\frac{t^*(n-t^*)}{n}$ in the normalization of the signal is the reciprocal of the null variance of the natural test statistic for testing for a mean shift in row $j$ when the changepoint location~$t^*$ is known:
\begin{align*}
\var_0\biggl(\frac{1}{t^*} \sum_{t=1}^{t^*} X_{jt} - \frac{1}{n-t^*} \sum_{t=t^*+1}^n X_{jt}\biggr) = \frac{n}{t^*(n-t^*)}.
\end{align*}
This rescaling can therefore be viewed as an effective sample size for a given changepoint location.

The coordinatewise minimum signal requirement in~\eqref{1sidecp} contrasts with the corresponding condition of~\cite{liu2021minimax} and~\cite{li2023robust} which is expressed in terms of the (squared) Euclidean norm of the vector of mean change.  In order to derive a formula for the critical radius that is sharp up to multiplicative constants, we require the coordinate-wise formulation, which is standard in the detection boundary literature (\cite{donoho2004higher,cai2014optimal}). The resulting squared critical radius is related to the non-asymptotic one derived by \cite{liu2021minimax} by a factor of $s$, as elaborated upon in Section~\ref{liudiscuss}.

Let $\Psi$ denote the set of Borel measurable test functions $\psi: \R^{p \times n}\to \{0,1\}$.  The \emph{worst-case testing error} of $\psi \in \Psi$ for our one-sided problem is defined as 
\[
\mathcal{R}_{\mathrm{1-side}}(\psi,p,n,\rho,s) \coloneqq \sup_{\theta \in \Theta_0(p,n)} \E_\theta \psi(X) + \sup_{\theta \in \Theta_1^{\mathrm{1-side}}(p,n,\rho,s)} \E_\theta\{1-\psi(X)\}.
\]
The \emph{minimax testing error} is then
\begin{align}
\label{minmaxerr}
\mathcal{R}_{\mathrm{1-side}}(p,n,\rho,s) \coloneqq \inf_{\psi \in \Psi} \mathcal{R}_{\mathrm{1-side}}(\psi,p,n,\rho,s).
\end{align}
Recall the formula for $\rho_{\mathrm{1-side}}^*(a,\beta)$ from~\eqref{eq:one-sided-radius}.  The following theorem is our main result.
\begin{thm}
\label{main1}
For $a > 0$ and $\beta_1 \in (0,1)$, let $\rho^* \coloneqq \rho_{\mathrm{1-side}}^*(a,\beta_1)$, and suppose that $s,n,p$ are related via~\eqref{3log}.
\begin{enumerate}
\item (Lower Bound) If $\beta > \beta_1$, then $\mathcal{R}_{\mathrm{1-side}}(p,n,\rho^*,s) \not\to 0$ as $p \to \infty$.
\item (Upper Bound) If $\beta < \beta_1$, then $\mathcal{R}_{\mathrm{1-side}}(p,n,\rho^*,s) \to 0$ as $p \to \infty$.
\end{enumerate}
\end{thm}
The alternative parameter spaces are nested with respect to the sparsity level in the sense that if $s_1 \leq s_2$, then $\Theta_1^{\mathrm{1-side}}(p,n,\rho,s_1) \subseteq \Theta_1^{\mathrm{1-side}}(p,n,\rho,s_2)$. 
Consequently, the first part of Theorem \ref{main1} implies that $\rho^*_{\mathrm{1-side}}(a,\beta_1)$ is a lower bound on the minimum signal required for the existence of a sequence of tests with asymptotically negligible error when $\beta > \beta_1$. We call a sequence of tests $(\psi_p)$ \textit{consistent} when the sum of type 1 and 2 errors of $\psi_p$ tends to zero as $p \to \infty$. The second part of Theorem \ref{main1} states that a consistent sequence of tests exists when $\beta < \beta_1$ and the signal strength is $\rho^*_{\mathrm{1-side}}(a,\beta_1)$. 

\subsection{Two-sided changepoint} In the ``two-sided'' version of the problem, the direction of the shift in mean at the changepoint location can be decreasing or increasing, so the parameter space is
\begin{align}
\nonumber
\Theta_1^{\mathrm{2-side}}(p,n,\rho,s) &\coloneqq   \bigg\{\theta \in \R^{p \times n}: \exists S \subseteq [p], \mu_{j1},\mu_{j2} \in \R, t^* \in [n-1] \text{ such that } |S|\leq s \text{ and } \\
\nonumber
&\sqrt{\frac{t^*(n-t^*)}{n}}|\mu_{j1}-\mu_{j2}| \geq \rho \text{ for every $j\in S$, where $\theta_{ji} = \mu_{j1}$ for $i\leq t^*$,} \\
\label{2sidecp}
&\text{and $\theta_{ji}=\mu_{j2}$ for $i > t^*$} \bigg\}.
\end{align} 
Since $\Theta_1^{\mathrm{1-side}}(p,n,\rho,s) \subseteq \Theta_1^{\mathrm{2-side}}(p,n,\rho,s)$, the minimum signal size required for the existence of consistent tests is at least as large for the two-sided version of the problem.

Define 
\begin{align}
\label{eq:two-sided-boundary}
\rho_{\mathrm{2-side}}^*(a,\beta)^2 &\coloneqq \begin{cases}
p^{\frac{a-(1-2\beta)}{2}} \hspace{1em} &a \leq 1-2\beta \\
\bigl(a-(1-2\beta)\bigr) \log p &1-2\beta < a \leq 1-4\beta/3 \\
2\bigl(\sqrt{1-a}-\sqrt{1-a-\beta}\bigr)^2 \log p &1-4\beta/3 < a \leq 1-\beta \\
p^{a-(1-\beta)} &a > 1-\beta.
\end{cases}
\end{align}
Observe that $\rho_{\mathrm{2-side}}^*(a,\beta)$ is strictly larger than $\rho_{\mathrm{1-side}}^*(a,\beta)$ in the densest case $a < 1-2\beta$, and identical everywhere else. 
Putting
\begin{align*}
\mathcal{R}_{\mathrm{2-side}}(p,n,\rho,s)\coloneqq \inf_{\psi  \in \Psi} \biggl[\sup_{\theta \in \Theta_0(p,n)} \E_\theta \psi(X) + \sup_{\theta \in \Theta_1^{\mathrm{2-side}}(p,n,\rho,s)} \E_\theta\{1-\psi(X)\} \biggr],
\end{align*}
we have the following two-sided analogue of Theorem \ref{main1}.
\begin{thm}
\label{main2}
For $a > 0$ and $\beta_1 \in (0,1)$, let $\rho^* \coloneqq \rho_{\mathrm{2-side}}^*(a,\beta_1)$, and suppose that $s,n,p$ are related via \eqref{3log}.
\begin{enumerate}
\item (Lower Bound) If $\beta > \beta_1$, then $\mathcal{R}_{\mathrm{2-side}}(p,n,\rho,s) \not\to 0$ as $p \to \infty$.
\item (Upper Bound) If $\beta < \beta_1$, then $\mathcal{R}_{\mathrm{2-side}}(p,n,\rho,s) \to 0$ as $p \to \infty$.
\end{enumerate}
\end{thm}
The proofs of Theorem \ref{main1} and \ref{main2} can be found in Sections 1.1 -- 1.4 of the supplementary file (\cite{xiang2025supplement}).
In the next section, we outline several connections between the changepoint detection boundaries stated above and several important results in the sparse signal detection literature. 

\section{Connections to related works}
\label{comparisons}

\subsection{Related literature in sparse signal detection}
\label{sec-lit-review-signal-detection}

Suppose $\theta \in \mathbb{R}^p$ satisfies $\|\theta\|_0 \leq s$.  Given $X \sim N_p(\theta,I_p)$, a canonical problem in sparse signal detection involves testing $H_0:\theta = 0$ against $H_1: \|\theta\|_2 \geq \rho^*(p,s)$. \cite{collier2017minimax} showed that the critical radius takes the form
\begin{align}
\label{collier}
\rho^*(p,s)^2 \asymp \begin{cases}
\sqrt{p} \hspace{1em} &\text{if } s\geq \sqrt{p} \\
s\log (ep/s^2) &\text{if } s < \sqrt{p}.
\end{cases}
\end{align}

The work of \cite{collier2017minimax} was extended by \cite{liu2021minimax} to a high-dimensional changepoint detection problem, in which each element in the collection of $p$ observations is a sequence $X_{j:} \in \R^n$. In other words, the observed data is a matrix $X \in \R^{p \times n}$, where each $n$-dimensional row is a unit of observation. This is in contrast with another view of the changepoint detection problem, where each column $X_{:t} \in \R^p$ constitutes a unit of observation, yielding a total of $n$ observations. However, we prefer to use $p \to \infty$ as the limiting index in order to simplify the comparison of our result to existing results on sparse signal detection. 

In the high-dimensional changepoint detection problem studied by \cite{liu2021minimax}, the signal size of the matrix $\theta \in \R^{p \times n}$ is defined as a normalized mean shift at the changepoint location. This formulation results in a critical radius
\begin{align}
\label{haoyang}
\rho^*(p,n,s)^2 \asymp \begin{cases}
\sqrt{p\log \log (8n)} \hspace{1em} &s > \sqrt{p\log\log (8n)} \\
s\log\left(\frac{ep\log\log (8n)}{s^2} \right) + \log\log (8n) &s \leq \sqrt{p\log\log (8n)},
\end{cases}
\end{align}
that can be compared to (\ref{collier}). The two critical radii are nearly identical when the problem dimension $p$ in (\ref{collier}) is replaced by $p\log \log (8n)$. The additive iterated logarithmic term in the sparse regime 
arises from the one dimensional changepoint detection problem $p=s=1$, which has critical radius:
\begin{align}
\label{gaocp}
\rho^*(1,n,1)^2 &\asymp \log\log (8n)
\end{align}
(\cite{arias2011detection,gao2020estimation}). In other words, the sparse high dimensional changepoint problem is at least as difficult as the one-dimensional changepoint problem, and at least as difficult as the sparse normal means problem with an inflated problem dimension $p\log\log(8n)$. We comment on this inflation factor and its role in comparing $\rho^*(a,\beta)_{\rm 1-side}^2$ to the formula of \cite{donoho2004higher} in the next section.


\subsection{Connection to Ingster--Donoho--Jin boundary}
\label{idjdiscuss}

In order to derive the multiplicative constant in~\eqref{collier}, one specifies an asymptotic relationship between $p$ and $s$. To capture the elbow effect at $s = \sqrt{p}$ in~\eqref{collier}, a natural asymptotic regime to consider is
\begin{align*}
s \sim p^{1-\beta} 
\end{align*}
for some $\beta \in (0,1)$.  The asymptotic constant when $\beta > 1/2$ was first derived by \cite{ingster1999some} and \cite{donoho2004higher} in the following version of the problem, where the number of non-null coordinates is random under the alternative:
\begin{align}
\label{ingster}
H_0: X_1,\dots,X_p \stackrel{\iid}{\sim} N(0,1) \hspace{1em} \mathrm{versus} \hspace{1em} H_1: X_1,\dots,X_p \stackrel{\iid}{\sim} (1-\e) N(0,1)  + \e N(\mu,1).
\end{align}
Here, $\e \coloneqq p^{-\beta}$ is the expected fraction of non-null coordinates. 
\cite{ingster1999some} and \cite{tony2011optimal} showed that the critical signal strength is 
\begin{align}
\label{idjcritconst}
\mu^*(\beta)^2 := \begin{cases}
p^{2\beta-1} &\text{ if } 0 < \beta \leq 1/2 \\
(2\beta- 1)\log p \hspace{1em} &\text{ if } 1/2 < \beta \leq 3/4 \\
2(1-\sqrt{1-\beta})^2\log p &\text{ if } 3/4 < \beta < 1.
\end{cases}
\end{align}

By the Neyman--Pearson lemma, the likelihood ratio test achieves the smallest sum of Type I and II errors in the testing problem (\ref{ingster}), and is therefore consistent throughout the detectable region $\{(r,\beta) : r > r^*(\beta)\}$. However, an obstacle in implementing the likelihood ratio test is that it requires knowledge of $r$ and $\beta$. \cite{donoho2004higher} resolved this issue by introducing an adaptive testing procedure based on Higher Criticism (\cite{tukey1989higher}) which does not require knowledge of $r$ and $\beta$, and is consistent throughout the detectable region.  The formula for $\mu^*(\beta)$ is recovered by taking the limit as $a \searrow 0$ in our one-sided changepoint detection boundary~\eqref{eq:one-sided-radius}.


The two-sided alternative in the sparse normal means problem (\ref{ingster}) is
\begin{align}
\label{symsparsemeans}
H_1: X_1,\dots,X_p \stackrel{\iid}{\sim} (1-\e) N(0,1)  + \e \biggl(\frac{1}{2} N(\mu,1)+\frac{1}{2} N(-\mu,1) \biggr).
\end{align}
Taking the limit as $a \searrow 0$ in the two-sided detection boundary~\eqref{eq:two-sided-boundary} recovers the two-sided sparse normal means critical testing radius, derived in \cite{cai2014optimal}:
\begin{align}
\label{eq:cai-wu-bdry}
\lim_{a \searrow 0} \rho^*_{\mathrm{2-side}}(a,\beta)^2 = \mu_{\mathrm{2-side}}^*(\beta)^2 \coloneqq 
\begin{cases}
p^{\beta-1/2} \hspace{1em} &\text{ if } 0<\beta\leq 1/2\\
(2\beta-1)\log p &\text{ if } 1/2 < \beta \leq 3/4 \\
2(1-\sqrt{1-\beta})^2 \log p &\text{ if } 3/4 < \beta < 1.
\end{cases}
\end{align}
The two detection boundaries are identical except in the dense case $\beta \in (0,1/2)$, where the factor of 2 difference (in the exponent) can be attributed to the symmetry of the problem (\cite{cai2014optimal}). Inspecting the formulae for $\rho^*_{\mathrm{1-side}}(a,\beta)$ and $\rho^*_{\mathrm{2-side}}(a,\beta)$ for general $a$, we see that the same phenomenon occurs in the dense regime $a \leq 1-2\beta$ of the high-dimensional changepoint problem.

In the dense regime, the detection boundary for the sparse means problem is related to that of the changepoint problem by a dimension inflation factor $p\log\log(8n) \sim p^{1+a}$, as mentioned in the previous section. Consider the total signal in both the changepoint problem and the sparse means problem, 
\begin{align*}
    s \rho^{*}_{\mathrm{2-side}}(a,\beta) = p^{1/2}, \hspace{1em} s \mu^*_{\mathrm{2-side}}(\beta) = p^{\frac{1+a}{2}}.
\end{align*}
Comparing the two, we find that replacing $p \mapsto p\log\log(8n)$ in the formula for the total signal in the sparse means problem yields the total signal in the changepoint problem. 
The formula of $\rho^*(a, \beta)^2$ in the other regimes does not appear to be easily derived from $\mu_{\mathrm{2-side}}^*(\beta)^2$ using the rule $p\mapsto p\log\log(8n)$.


\subsection{Connection to \cite{chan2015optimal}}
\label{chandiscuss}

\cite{chan2015optimal} studied a high-dimensional changepoint problem with multiple changepoints.  These authors test the global null that $\theta = 0$ in the model~\eqref{eq:observation-matrix}, against the alternative that there exist positive integers $q,\ell^{(1)},\ldots,\ell^{(q)}$, as well as $s^{(1)} < s^{(1)} + \ell^{(1)} < s^{(2)} < s^{(2)} + \ell^{(2)} < \cdots < s^{(q)} < s^{(q)} + \ell^{(q)} \leq~n$ and independent random variables $(I_j^{(k)})_{j \in [p],k \in [q]}$ and $(E_{jt})_{j \in [p],t \in [n]}$ with $I_j^{(k)} \sim \mathrm{Bern}(\varepsilon^{(k)})$ and $E_{jt} \sim N(0,1)$, such that 
\[
X_{jt} = \frac{\mu^{(k)}I_j^{(k)}}{\sqrt{\ell^{(k)}}} \textbf{1}_{\{t \in (s^{(k)},\; s^{(k)} + \ell^{(k)}]\}} + E_{jt} 
\]
for $j \in [p]$ and $t \in [n]$, where $\mu^{(1)},\ldots,\mu^{(q)} > 0$.  The authors consider the asymptotic regime 
\begin{align}
\label{eq:chan-walther-asymp}
\log\log \Bigl(\frac{n}{\ell^{(k)}}\Bigr) \sim a^{(k)}\log p
\end{align}
where $a^{(k)} > 0$ for $k \in [q]$.  The detection boundary $\rho^*_{\mathrm{1-side}}(a,\beta)$ for our one-sided changepoint problem appears as a critical radius for this problem in the following sense.
\begin{thm}[\cite{chan2015optimal}]
\label{chan}
Let $\mu^{(k)} \coloneqq \rho^*_{\mathrm{1-side}}(a^{(k)},\beta^{(k)})$ for $k \in [q]$. 
\begin{enumerate}
\item (Lower bound) For each $k \in [q]$, put $\e^{(k)} \coloneqq p^{-\beta^{(k)}-\d^{(k)}}$ for some $0 < \beta^{(k)} < 1$, and $\d^{(k)} > 0$. Then there does not exist a sequence of tests with Type I and II errors tending to zero as $p\to\infty$.
\item (Upper bound) For each $k \in [q]$, put $\e^{(k)} \coloneqq p^{-\beta^{(k)}+\d^{(k)}}$ for some $0 < \beta^{(k)} < 1$, and $0 < \d^{(k)} \leq \beta^{(k)}$. Then there exists a sequence of tests with Type I and II errors tending to zero as $p \to \infty$.
\end{enumerate}
\end{thm}
To lessen the notational burden, and to make a direct comparison to the problem analyzed in this paper, we only consider the case $q=1$. In \cite{chan2015optimal}, the ``Lower bound'' of Theorem \ref{chan} is proved using a construction with $q=1$, which implies that this restriction is without loss of generality. In this case, if the $j^{th}$ row of $\theta$ is non-null, there is one contiguous region $S_n \subseteq[n]$ over which the mean vector $\theta_{j:}$ is elevated, i.e. $\theta_{jt} > 0$ for $t \in S_n$, and $\theta_{jt} = 0$ for $t \notin S_n$.

According to the asymptotic regime \eqref{eq:chan-walther-asymp}, the length $\ell_n^{(1)} = |S_n|$ of the elevated region is determined by the relation,
\begin{align*}
\log\log\left(\frac{n}{\ell_n} \right) \sim a^{(1)} \log p.
\end{align*}
In other words, the mean vector is broken into three pieces, two of which consist only of zeros. The middle piece is positive and has length calibrated by $n$ and $p$ as above. For example, if $n/\ell_n \sim \log n$, then the length $\ell_n$ of the elevated region is a vanishing fraction of $n$.

In our formulation of the problem, the mean matrix is non-random, and each non-null row has two pieces. The length of the first piece is equal to the changepoint location $t^* \in [n-1]$, as defined in the parameter spaces $\Theta_1^{\mathrm{1-side}}$ and $\Theta_1^{\mathrm{2-side}}$. In terms of the critical signal size, the two piece problem with known direction of change coincides in difficulty with the three piece problem whose elevated region occurs between $\lfloor n- n/\log n\rfloor$ and $n$.

\subsubsection{Minor technical points}

In the lower bound construction for Theorem \ref{chan}, the changepoint location is selected uniformly at random from $n/\ell_n$ candidate positions in $[n-\ell_n]$. The likelihood ratio for this construction involves a sum of independent random variables, which the authors analyzed using the Lyapunov Central Limit Theorem. In the lower bound construction for Theorem \ref{main1}, a changepoint location is selected uniformly from a set of candidate positions $\T \subseteq[n]$. Since the elevated segment starts at index 1 and ends at the changepoint $t^* \in \T$, the resulting elevated segments (corresponding to the candidate changepoint locations $t^* \in \T$) overlap. As a result, the likelihood ratio for this construction involves a sum of dependent random variables, for which the CLT cannot be directly applied. Instead, we use a direct argument similar to the proof of Theorem 2(a) in \cite{hu2023likelihood} to handle this case. 

The upper bounds in Theorems \ref{main1}, \ref{main2}, and \ref{main3} are attained by a procedure which considers the maximum over a collection of Berk--Jones statistics, each computed with respect to a candidate elevated segment in the non-null rows of $\theta$. This maximum is then penalized by subtracting a quantity depending on the cardinality of the grid of candidate changepoint locations. The construction of an optimal test in the current paper uses a geometrically growing changepoint locations similar to those used in the optimal test for achieving the lower bound in Theorem 1 of \cite{liu2021minimax}. The ideas from the procedures in \cite{chan2015optimal} and \cite{liu2021minimax}, and how they are combined to yield optimal procedures for the testing problems considered in the current paper, are detailed in Section \ref{calcupp}.

\subsubsection{Multiple changepoints}

The comparison to \cite{chan2015optimal} suggests that our calibration $\log\log\log n \sim a\log p$ is a result of the alternative hypothesis having only one changepoint. Though not explicitly considered in \cite{chan2015optimal}, one could expect that the calibration $\log\log n \sim a\log p$ is appropriate
when the alternative hypothesis involves multiple changepoints. This difference between single and multiple changepoints has been previously studied by \cite{pilliat2023optimal}.

In order to illustrate this contrast, we extend our result in Theorem \ref{main2} to the case where there are two changepoints in the mean matrix. The null hypothesis is the same:
\begin{align*}
    H_0: \theta_{j1}=\dots=\theta_{jn} \hspace{1em} \text{ for all }j=1,\dots,p.
\end{align*}
In other words, $\theta\in\Theta_0(p,n)$ under null.
The alternative hypothesis $H_1:\theta\in \Theta_2(p,n,\rho,s)$ considers the parameter space
\begin{align*}
    \Theta_2(p,n,\rho,s)&:= \Big\{\theta\in\mathbb{R}^{p\times n}: \exists \mu^{(1)},\mu^{(2)},\mu^{(3)} \in \R^p, t_1 < t_2 \in [n-1], \text{ and } S \subseteq [p] \text{ with } |S| \leq s \\
    & \text{ such that for }j\in[p],  \theta_{jt}=\mu_{j}^{(1)} \text{ for all } t \leq t_1, \theta_{jt}=\mu_{j}^{(2)} \text{ for all } t \in (t_1,t_2] \text{ and } \theta_{jt} = \mu_{j}^{(3)}  \\
    & \text{ for all } t > t_2,  \text{while } \Delta^{(1)} \coloneqq \mu^{(2)}-\mu^{(1)} \text{ and } \Delta^{(2)} \coloneqq \mu^{(3)}-\mu^{(2)} \text{ satisfy } \Delta^{(1)}_j = \Delta^{(2)}_j = 0 \\
    & \text{ for all } j \in S^c, \text{ and for each $j \in S$, } \|\Sigma_{t_1,t_2}^{-1/2} (\Delta^{(1)}_j, \Delta^{(2)}_j)^\top \|^2 \geq \rho^2\Big\},
\end{align*}
where $\Sigma_{t_1,t_2} \in \R^{2 \times 2}$ is defined by
\begin{align*}
    \Sigma_{t_1,t_2} 
    &= \begin{bmatrix}
    \frac{1}{t_1}+\frac{1}{t_2-t_1} & \; - \frac{1}{t_2-t_1} \\ \\
    -\frac{1}{t_2-t_1} & \; \frac{1}{n-t_2} + \frac{1}{t_2-t_1}
    \end{bmatrix}.
\end{align*}
We briefly motivate the form of the signal requirement in the definition of the alternative hypothesis before stating our result. 
For simplicity, consider testing whether a single row $j$ is non-null, i.e. there exist integers $1\leq t_1<t_2 \leq n-1$ for which the mean vector $\theta_{j:}$ satisfies the conditions in $H_1$. A natural test statistic for checking whether there is a mean change at given $t_1$ and $t_2$ is given by
\begin{align*}
    Y_{t_1,t_2}\coloneqq \left( \frac{1}{t_2-t_1}\sum_{i=t_1+1}^{t_2}X_{ji}- \frac{1}{t_1}\sum_{i=1}^{t_1}X_{ji},\; \; \frac{1}{n-t_2}\sum_{t_2+1}^{n} X_{ji} - \frac{1}{t_2-t_1}\sum_{i=t_1+1}^{t_2}X_{ji}  \right)^\top \in \R^2.
\end{align*}
Suppose $t_1,t_2$ are the actual changepoint locations. Then the distribution of  $Y_{t_1,t_2}$ is given by
\begin{align*}
\Sigma_{t_1,t_2}^{-1/2} Y_{t_1,t_2} \sim N(\Sigma^{-1/2}(\Delta_j^{(1)},\Delta_j^{(2)})^\top,I_2),
\end{align*}
where $I_2$ is the $2 \times 2$ identity matrix. Therefore, the signal requirement is naturally inspired by using the chi-squared statistic $\|\Sigma_{t_1,t_2}^{-1/2} Y_{t_1,t_2}\|^2$ to detect changepoints at locations $t_1$ and $t_2$. Our next result states that the formula \eqref{eq:two-sided-boundary} for the two-sided detection boundary characterizes the critical signal size in this testing problem. We record a proof of the lower and upper bounds in Sections 1.7 and 1.8 of the supplemental file. 

\begin{thm}
\label{thm:two-changepoints}
For $a > 0$ and $\beta_1 \in (0,1)$, let $\rho^* \coloneqq \rho_{\mathrm{2-side}}^*(a,\beta_1)$, and suppose that $s,n,p$ are related via
\begin{equation}
\begin{aligned}
\log\log n &\sim a\log p\\
s &\sim p^{1-\beta},
\end{aligned}
\end{equation}
where $a>0$ and $\beta \in (0,1)$. Then the following statements hold.
\begin{enumerate}
\item (Lower Bound) If $\beta > \beta_1$, then 
\begin{align*}
    \inf_{\psi  \in \Psi} \biggl[\sup_{\theta \in \Theta_0(p,n)} \E_\theta \psi(X) + \sup_{\theta \in \Theta_2(p,n,\rho,s)} \E_\theta\{1-\psi(X)\} \biggr] \not\to 0
\end{align*}
as $p \to \infty$.
\item (Upper Bound) If $\beta < \beta_1$, then 
\begin{align*}
    \inf_{\psi  \in \Psi} \biggl[\sup_{\theta \in \Theta_0(p,n)} \E_\theta \psi(X) + \sup_{\theta \in \Theta_2(p,n,\rho,s)} \E_\theta\{1-\psi(X)\} \biggr] \to 0
\end{align*}
as $p \to \infty$.
\end{enumerate}
\end{thm}

\begin{remark}
Theorem \ref{thm:two-changepoints} can be extended to handle $K=O(1)$ changepoints. The proof would be the same, once we modify the covariance matrix $\Sigma_{t_1\cdots t_K}$ appropriately. Assuming the changes occur at $t_1,\dots,t_{K}$, the new covariance matrix in the definition of $H_1$ has non-zero entries along the diagonal and adjacent to the diagonal, and is zero otherwise:
\begin{align*}
\Sigma_{t_1\cdots t_K} \coloneqq \begin{bmatrix}
    \frac{1}{t_1} + \frac{1}{t_2-t_1} & -\frac{1}{t_2-t_1} & 0 & 0 & \cdots & 0 \\ \\
    -\frac{1}{t_2-t_1} & \frac{1}{t_2-t_1} + \frac{1}{t_3-t_2} & - \frac{1}{t_3-t_2} & 0 & \cdots & 0 \\ \\
    0 & -\frac{1}{t_3-t_2} & \ddots & \ddots & & \vdots \\ \\
    0 & 0 & \ddots &  &  & 0 \\ \\
    \vdots & \vdots &  &  & \ddots & -\frac{1}{t_K-t_{K-1}} \\ \\
    0 & 0 & \cdots  & 0 & -\frac{1}{t_K-t_{K-1}} & \frac{1}{t_K-t_{K-1}} + \frac{1}{n-t_{K}} 
\end{bmatrix}    
\end{align*}
For this version of the problem, the detection boundary is also the same as in the setting with a single changepoint. To avoid redundancy, we omit the details of this result.
\end{remark}





\subsection{Connection to the non-asymptotic rate of \cite{liu2021minimax}}
\label{liudiscuss}

\cite{liu2021minimax} define the critical radius (\ref{haoyang}) with respect to the Euclidean norm of the $p$ dimensional vector of mean differences at the changepoint location $t^* \in [n-1]$, 
\begin{align*}
\Theta^{(t^*)}(p,n,\rho,s) \coloneqq \bigg\{&\theta =(\theta_1,\dots,\theta_n) \in \R^{p \times n}: \theta_{t} = \mu_1 \text{ for some } \mu_1 \in \R^p \text{ for all }1 \leq t \leq t^*, \\ &\theta_t = \mu_2 \text{ for some } \mu_2 \in \R^p \text{ for all } t^* + 1 \leq t \leq n, \\ &\|\mu_1-\mu_2\|_0 \leq s, \sqrt{\frac{t^*(n-t^*)}{n}}\|\mu_1-\mu_2\| \geq \rho \bigg\}.
\end{align*}
Taking the union over all possible changepoint locations $t^* \in [n-1]$ gives the alternative hypothesis parameter space,
\begin{align*}
\Theta_1(p,n,\rho,s) \coloneqq \bigcup_{t^*=1}^{n-1} \Theta^{(t^*)}(p,n,\rho,s).
\end{align*}
The null parameter space is the same as the one considered in the current paper, $\Theta_0(p,n)$. The problem is to test,
\begin{align*}
H_0: \theta \in \Theta_0(p,n) \hspace{2em} \mathrm{versus} \hspace{2em} H_1: \theta \in \Theta_1(p,n,\rho,s),
\end{align*}
for which it is shown that (\ref{haoyang}) is the detection boundary. Note that the signal size $\rho$ in $\Theta_1(p,n,\rho,s)$ is defined with respect to the Euclidean norm of the entire vector (in $\R^p$) of changes, as opposed to the signal size required on each individual non-null row. To obtain a quantity comparable to the critical radius $\rho^*_{\mathrm{2-side}}$, we divide expression (\ref{haoyang}) by $s$ for the minimum required signal in each non-null row,
\begin{align}
\label{hy}
\frac{\rho^*(p,n,s)^2}{s} &\asymp \begin{cases}
\frac{\sqrt{p\log \log (8n)}}{s} \hspace{1em} &s > \sqrt{p\log\log (8n)} \\
\log\left(\frac{ep\log\log (8n)}{s^2} \right) + \frac{\log\log (8n)}{s} &s \leq \sqrt{p\log\log (8n)}.
\end{cases}
\end{align}
According to the calibration, 
\begin{align*}
\log\log\log n &\sim a \log p,\\
s &\sim p^{1-\beta},
\end{align*}
the ``sparse'' case, $s \leq \sqrt{p\log\log n} \iff 1+a-2(1-\beta)\geq 0$, in (\ref{hy}) can be further split into the following cases,
\begin{align*}
\log\left(\frac{ep\log\log n}{s^2} \right) + \frac{\log\log n}{s} \asymp \begin{cases}
\log\left( ep^{1+a-2(1-\beta)} \right) \hspace{1em} &a - (1-\beta) \leq 0\\
p^{a-(1-\beta)} &a-(1-\beta) > 0.
\end{cases}
\end{align*}
Then, rewriting (\ref{hy}) to include this further division, and ignoring constants, we have
\begin{align}
\label{hybdry}
\frac{\rho^*(p,n,s)^2}{s} \asymp \begin{cases}
p^{\frac{a-(1-2\beta)}{2}} \hspace{1em} &a \leq 1-2\beta \\
\log p &1-2\beta < a \leq 1-\beta\\
p^{a-(1-\beta)} &a>1-\beta,
\end{cases}
\end{align}
which coincides with the formula $\rho^*_{\mathrm{2-side}}(a,\beta)^2$ up to the coefficient of the $\log p$ rate in the case $1-2\beta < a \leq 1-\beta$. 

In the $s=1$ case, two competing terms in the non-asymptotic rate are $\log p$ and $\log\log n$, suggesting an asymptotic regime in which $\log\log n$ does not need to grow as fast as polynomial in $p$, in order for a nontrivial interaction to arise. This asymptotic regime yields another detection boundary, which we discuss next.

\subsubsection{Another asymptotic regime}
Another asymptotic calibration we consider is
\begin{equation}
\label{2log}
\begin{aligned}
\log\log n &\sim a\log p \hspace{2em} a > 0 \\
s &\sim p^{1-\beta} \hspace{2em} \frac{1}{2} < \beta \leq 1.
\end{aligned}
\end{equation}
The corresponding detection boundary is almost exactly the same as the IDJ-detection-boundary, with the changepoint aspect of the problem only playing a role in the most sparse regime $\beta=1$. The formula is $\rho_2^*(a,\beta) \coloneqq \sqrt{2r_2^*(a,\beta) \log p}$, where
\begin{align*}
r_2^*(a,\beta) &\coloneqq \begin{cases}
\beta-1/2 \hspace{1em} &1/2 < \beta \leq 3/4 \\
(1-\sqrt{1-\beta})^2 &3/4<\beta < 1\\
1+a &\beta = 1.
\end{cases}
\end{align*}
We only state the result for the one-sided version of the problem, since there is no difference between the one-sided and two-sided detection boundaries when $\beta > 1/2$. The squared critical radius for the two-sided problem in the case $\beta \leq 1/2$ is equal to $p^{\beta-1/2}$, which is directly implied by plugging $s=p^{1-\beta}$ and $\log\log n \asymp \log p$ into the non-asymptotic rate (\ref{hy}). 

In the formula $r_2^*(a,\beta)$, the case in which the boundary departs from the Ingster--Donoho--Jin boundary is when $\beta = 1$, corresponding to the case in which the number of non-null rows in $\theta$ is bounded as $p \to \infty$. Taking $a\searrow 0$ in the expression for $r^*_2(a,\beta)$ recovers the Ingster--Donoho--Jin boundary.
\begin{thm}
\label{main3}
For $a>0$ and $\beta \in (1/2,1]$, let $\rho \coloneqq \sqrt{2r\log p}$ for some $r > 0$, and suppose that $s,n,p$ are related via \eqref{2log}.
\begin{enumerate}
\item (Lower Bound) If $r<r_2^*(a,\beta)$, then $\mathcal{R}_{\mathrm{1-side}}(p,n,\rho,s) \not\to 0$, as $p \to \infty$.
\item (Upper Bound) If $r>r_2^*(a,\beta)$, then $\mathcal{R}_{\mathrm{1-side}}(p,n,\rho,s) \to 0$, as $p \to \infty$.
\end{enumerate}
\end{thm}
The detection boundary $\rho^*_2(a,\beta)$ does not depend continuously on $a$ as $\beta \to 1$, and this discontinuity arises in the non-asymptotic rate (\ref{hy}) as well. This can be understood by considering the case $\beta < 1$, where $\log\log n = o(s)$, so there is no contribution from the additive $\frac{\log\log n}{s} \sim \frac{a\log p}{p^{1-\beta}}$ term in the limit as $p\to \infty$, and thus $a$ has no contribution to the formula for the detection boundary. In contrast, when $\beta = 1$, this term becomes $\frac{\log\log n}{s} \sim a\log p$, which is non-negligible compared to the competing term $\log \left(\frac{ep\log\log n}{s^2}\right) \sim \log p$. Taking $\beta=1$ gives
\begin{align*}
\rho^*_2(a,1)^2 = 2(1+a)\log p \sim 2(\log p + \log\log n),
\end{align*}
which is the minimax testing radius for the $s=1$ high dimensional changepoint detection problem, with multiplicative constant 2. Taking $a \to \infty$ when $\beta=s=1$ gives
\begin{align*}
\rho^*_2(a,1)^2 = 2(1+a)\log p \sim 2(\log p + \log\log n) = 2(1+o(1))\log\log n.
\end{align*}
Hence the sharp constant of 2 in the $p=s=1$ changepoint detection problem (\cite{verzelen2023optimal}) is consistent with Theorem \ref{main3}. The proof of Theorem \ref{main3} can be found in Sections 1.5 and 1.6 of the supplemental file.

\subsection{Connection to submatrix detection in \cite{butucea2013detection}}
\label{butudiscuss}

\cite{butucea2013detection} studied a related problem in which the entire observation is a matrix $Y \in \R^{N \times M}$ of unit variance Gaussian entries. Under the null, these entries have zero mean. Under the alternative, a submatrix $(Y_{ij})_{i \in A, j \in B}$ has elevated mean $s_{ij} \geq \rho$ for $(i,j) \in A \times B$ and zero mean $s_{ij} = 0$ for $(i,j) \notin A \times B \subseteq[N] \times [M]$. $A$ and $B$ are constrained to be of sizes $n\leq N$ and $m\leq M$ respectively. In this notation, the testing problem can be written,
\begin{align*}
&H_0 : s_{ij} = 0 \text{ for all }i=1,\dots,N, \text{ and } j =1,\dots, M \\
&H_1: \exists A \times B\in \mathcal{C}_{nm} \text{ such that } s_{ij} = 0 \text{ if }(i,j) \notin A \times B, \text{ and }  s_{ij} \geq \rho \text{ if } (i,j) \in A \times B,
\end{align*}
where $\mathcal{C}_{nm}$ is the collection of rectangles $A \times B \subseteq[N] \times [M]$ satisfying $|A|=n$ and $|B|=m$. \cite{butucea2013detection} established both the rate and constant for the minimax separation rate $\rho$, in a specific asymptotic regime which requires,
\begin{align}
\label{ingsterasymp}
n \log (N/n) \asymp m \log(M/m).
\end{align}
The main result of \cite{butucea2013detection} states that
the minimax separation $\rho$ satisfies
\begin{align}
\label{butucea}
\rho^2 =  \frac{2(n\log(N/n)+m\log(M/m))}{nm}(1+o(1)).
\end{align}
The purpose of this section is to show how formula \eqref{butucea} coincides with the critical radius $\rho^*_{\mathrm{1-side}}$ in a specific case of the asymptotic regime \eqref{3log}. In what follows, we make a comparison between the changepoint problem and the submatrix detection problem with $m=1$, which can be solved using the method in this paper. Specifically, our method can be viewed as having aggregated a Berk-Jones statistic from each of the $M$ columns to test whether one of these columns has a sparse subset of entries with an elevated mean.

The proof of the lower bound in Theorems \ref{main1} and \ref{main2} relies on the hardness of distinguishing between the following simple hypotheses for small enough $\rho$:
\begin{align*}
&H_0 : (X_{j:})_{j\in[p]} \sim \prod_{j=1}^p N(0,I_n)\\
&H_1: (X_{j:})_{j \in [p]} \sim \frac{1}{|\T'|} \sum_{k=1}^{|\T'|} \prod_{j=1}^p \left[(1-\e)N(0,I_n)+\e N(\rho\theta^{(k)},I_n) \right],
\end{align*}
where $\varepsilon$ represents the sparsity rate $s/p\sim p^{-\beta}$, $\theta^{(k)}$ is a piece-wise constant vector with a single changepoint at $t_k \coloneqq \lfloor (\log n)^k\rfloor$ defined by the relation \eqref{def:theta-k} (for an explicit definition of $\theta^{(k)}$, see Lemma 2.1 in Section 2 of the supplemental file), 
and $\mathcal{T}' \subseteq [n]$ is a set of candidate changepoint locations with cardinality of order $(\log n)^{1+o(1)}$, defined in \eqref{lowerT}. The likelihood ratio for this testing problem is 
\begin{align*}
\frac{f_1}{f_0}(X) = \frac{1}{|\T'|} \sum_{k=1}^{|\T'|} \prod_{j=1}^p \left[ 1+ \e \left(e^{\rho Y_{jt_k}-\rho^2/2}-1 \right)\right],
\end{align*}
which only depends on $X$ via the contrasts 
\begin{align}
\label{def:theta-k}
Y_{jt_k} = \langle X_{j:},\theta^{(k)} \rangle = \sqrt{\frac{\lfloor b^k \rfloor (n-\lfloor b^k \rfloor)}{n}} \left(\frac{1}{\lfloor b^k \rfloor} \sum_{t=1}^{\lfloor b^k \rfloor} X_{jt} -\frac{1}{n-\lfloor b^k \rfloor} \sum_{t=n-\lfloor b^k \rfloor }^{n} X_{jt}\right)
\end{align}
where $j=1,\dots,p$ and $k = 1,\dots, |\T'|$, $b = \log n$, and the second equality holds by definition of $\theta^{(k)}$ (see Part 1 of Lemma 2.1 in Section 2 of the supplemental file). 
In other words, these contrasts are sufficient for computing the likelihood ratio test $\psi_{\LRT}$ (assuming $\e$ and $\rho$ are known). 
Since the true changepoint location is drawn uniformly at random from $\T'$, the contrasts can be written,
\begin{align*}
Y_{jt_k} &= \langle X_{j:},\theta^{(k)}\rangle = \langle \rho \theta^{(k^*)} + Z_{j:},\theta^{(k)} \rangle = \rho \langle \theta^{(k^*)},\theta^{(k)}\rangle + \langle Z_{j:},\theta^{(k)}\rangle .
\end{align*}
Letting $k^* \sim \mathrm{Unif}\{1,\dots,|\T'|\}$ denote the index of the true changepoint location, we have
\begin{align*}
Y_{jt_k} &\stackrel{(d)}{=} \begin{cases}
\rho + Z \hspace{1em} &\text{if } k^*=k \\
o(1) + Z &\text{if } k^*\neq k,
\end{cases}
\end{align*}
since the grid $\T'$ has base $b = \log n = e^{p^{a(1+o(1))}}$ in the current regime (\ref{3log}) (see Parts 2 and 3 of Lemma 2.1 in the supplemental file), where $Z \sim N(0,1)$. Further, the covariance between two contrasts $Y_{jt_k}$ and $Y_{jt_\ell}$ corresponding to row $j$ is
\begin{align*}
\cov(Y_{jt_k},Y_{jt_\ell}) &= \langle \theta^{(k)},\theta^{(\ell)} \rangle = \begin{cases}
1 \hspace{1em} &\text{if }k=\ell \\
O(b^{-\frac{|k-\ell|}{2}}) &\text{if } k\neq \ell,
\end{cases}
\end{align*}
by Part 3 of Lemma 2.1 of the supplemental file. 
In other words, the $(Y_{jt_k})$ are nearly independent Gaussians, and under $H_1$, there exists some $k^* \leq |\T'|$ for which $(Y_{jt_{k^*}})_{j\in [p]}$ have elevated mean $\rho$, and all other $(Y_{jt_k})_{j\in [p], k\neq k^*}$ are Gaussian variables that have unit variance and nearly zero mean.

Consider the matrix $Y \in \R^{p \times |\T'|}$ whose $(j,k)$ entry is $Y_{jt_k}$. By the previous observation, the $(Y_{jt_k})$ are approximately independent, and the testing problem (\ref{lbconstruction0}) versus (\ref{lbconstruction1}) is essentially testing between the hypotheses
\begin{align*}
&H_0 : (Y_{j:})_{j \in [p]} \sim \prod_{j=1}^{p} N(0,I_{|\T'|}) \\
&H_1: (Y_{j:})_{j \in [p]} \sim \frac{1}{|\T'|} \sum_{k=1}^{|\T'|} \prod_{j=1}^p \left[(1-\e)N(0,I_{|\T'|})+\e N(\rho e_{k},I_{|\T'|}) \right],
\end{align*}
where $e_k \in \R^{|\T'|}$ is the $k^{th}$ standard basis vector. Let $I_j \in \{0,1\}$ indicate whether or not row $j$ in $Y$ has non-zero mean $\rho e_k$, i.e. $I_j = 1 \iff$ row $j$ is non-null.  Under $H_1$, there is some rectangle 
\begin{align}
\label{skinnyrect}
\{j \leq p : I_j = 1\} \times \{k\} \subseteq[p] \times [|\T'|],
\end{align}
over which the entries $Y_{j t_k}$ have elevated mean $\rho > 0$. This problem is a special case of the submatrix detection problem studied by \cite{butucea2013detection}.

In our setting, the rectangles (\ref{skinnyrect}) and matrix of contrasts $(Y_{jk}) \coloneqq (Y_{jt_k})$ have dimensions 
\begin{align*}
    (n,m,N,M) \coloneqq (p^{1-\beta},1,p,e^{p^{a(1+o(1))}}).
\end{align*}
Then condition (\ref{ingsterasymp}) corresponds to $1-\beta = a$, but in general does not cover the range of parameters $(a,\beta)$ considered in our paper. The result (\ref{butucea}) then reduces to
\begin{align*}
\rho^2 = \frac{2 (p^{1-\beta} \beta \log p + p^{a})}{p^{1-\beta}} = 2\beta (1+o(1))\log p,
\end{align*}
which coincides with the detection boundary in the case $a = 1-\beta$,
\begin{align*}
\rho^*_{\mathrm{1-side}}(1-\beta,\beta) = 2\beta \log p.
\end{align*}

\section{Overview of calculations}
\label{overview}

In this section, we provide a high-level overview of the proof for the one-sided changepoint problem in the asymptotic calibration \eqref{3log}. All rigorous proofs are deferred to the supplementary material. 

\subsection{Upper bound}
\label{calcupp}
To derive the upper bound result of Theorems \ref{main1} and \ref{main3}, it suffices to provide a testing procedure with Type I and II error tending to zero as $p\to \infty$. If the location $t^*\in [n-1]$ of the changepoint were known, one could compute a contrast for each row $j \in [p]$,
\begin{align}
\label{ycontrast}
Y_{jt^*} \coloneqq \sqrt{\frac{t^*(n-t^*)}{n}} \left(\frac{1}{t^*} \sum_{t=1}^{t^*}X_{jt} - \frac{1}{n-t^*} \sum_{t=t^*+1}^{n} X_{jt} \right).
\end{align}
Since $Y_{jt} \sim N(0,1)$ under $H_0$ for any $t \in [n-1]$, we can count the number of contrasts exceeding a null quantile $\bar{\Phi}^{-1}(q)$ for some $q \in (0,1)$, and compare the fraction of large signals to the expected fraction $q$ under the null,
\begin{align*}
\bar{S}_{p,t^*}(q) &\coloneqq \frac{1}{p} \sum_{j=1}^p \textbf{1}_{\{Y_{jt^*} > \bar{\Phi}^{-1}(q)\}} \\
\E_0 \bar{S}_{p,t}(q) &= q \text{ for any } t \in [n-1].
\end{align*}
A measure of probabilistic dissimilarity of these two fractions is $K(\bar{S}_{p,t^*}(q),q)$, where the function $K : (0,1)^2 \to \R$ is defined,
\begin{align}
\label{KL-function}
K(x,t) \coloneqq x\log \frac{x}{t} + (1-x)\log\frac{1-x}{1-t},
\end{align}
which is the KL divergence between two Bernoulli distributions with success parameters $x$ and $t$, i.e. $K(x,t) = \KL(\mathrm{Bern}(x),\mathrm{Bernoulli}(t))$. Hence, the statistic $K(\bar{S}_{p,t^*}(q),q)$ is large when the observed proportion differs greatly the expected proportion under the null. 
The supremum of the quantity $K(\bar{S}_{p,t^*}(q),q)$ over all $q \in (0,1)$ is achieved at some maximizing $p$-value, reducing this supremum to a finite maximum \cite{berk1979goodness},
\begin{align}
\label{BJ1}
\sup_{q \in (0,1)} K(\bar{S}_{p,t^*}(q),q) = \max_{j \leq p} K(j/p,p_{(j)}),
\end{align}
where $p_{(1)} < p_{(2)} < \dots < p_{(p)}$ are the ordered $p$-values $p_j \coloneqq \bar{\Phi}(Y_{jt^*})$. Since the true changepoint location $t^*$ is typically not known a priori, we might consider using as a test-statistic the maximum of (\ref{BJ1}) over all possible locations $t^* \in [n-1]$. However, the elevated regions corresponding to the changepoint locations $t^*$ and $t^*+1$ are almost completely overlapping. For each $j$, the $(Y_{jt^*})_{t^* \in [n-1]}$ are highly dependent over $t^* \in [n-1]$, which may lead to a violation of type I error control. A similar issue arises in the non-asymptotic analysis of the problem, where \cite{liu2021minimax} overcome this obstacle by taking a restricted maximum over a grid of geometrically growing candidate changepoint locations,
\begin{align*}
\{1,2,4,\dots,2^{\lfloor \log_2 (n/2) \rfloor}\} \subseteq[n-1].
\end{align*}
In order to derive the multiplicative constant appearing in the formula $\rho^*_{\mathrm{1-side}}(a,\beta)$, we use a dense and symmetric version of the above grid which ensures that for any true changepoint location $t^* \in [n-1]$, there exists an element of the grid close to $t^*$. The grid is defined as
\begin{align*}
\T \coloneqq \left\{\lfloor (1+\delta)^0\rfloor,\lfloor (1+\delta)^1\rfloor,\dots,\lfloor (1+\delta)^{\log_{1+\delta}\frac{n}{2}}\rfloor,\lfloor n-(1+\delta)^{\log_{1+\delta} \frac{n}{2}}\rfloor,\dots,\lfloor n-(1+\delta)^0\rfloor \right\},
\end{align*}
where $\delta = \frac{1}{\log\log n} \to 0$. The final test statistic is a penalized maximum over this grid,
\begin{align}
\label{pbj}
\PBJ_p \coloneqq \left[\max_{t \in \T} \sup_{q \in (0,1)} p K(\bar{S}_{p,t}(q),q)\right] - 2\log |\T|,
\end{align}
and the testing procedure rejects when this test statistic exceeds a threshold,
\begin{align*}
    \psi_{\PBJ}(X) &= \begin{cases}
        1 \hspace{1em} &\text{if }\PBJ_p > 2(2+\gamma) \log p \\
        0 &\mathrm{else},
    \end{cases}
\end{align*}
where $\gamma>0$ is any positive constant. The two-sided test is exactly the same, except the count variable $\bar{S}_{p,t}(q)$ is defined to count the number of exceedances in either the positive or negative direction $\sum_{j=1}^p \textbf{1}_{\{|Y_{jt}|>\bar{\Phi}^{-1}(q/2)\}}$.

A similar testing procedure was developed by \cite{chan2015optimal} in a related changepoint detection problem discussed in Section \ref{chandiscuss}. The penalty term $2\log |\T|$ appearing in \eqref{pbj} is needed for type 1 error control, shown in Section 1.3 of the supplemental file.

\subsection{Lower bound}
\label{calclow}
A geometrically growing grid (similar to $\T$ from the upper bound) plays a key role in the lower bound construction for each of Theorems \ref{main1} and \ref{main2}. The true changepoint is first selected uniformly at random from this grid, and conditional on this selection, the rows of $X$ are distributed iid from a mixture distribution. Note that the cardinality of the grid $\T$ from the proof of the upper bound behaves as,
\begin{align*}
|\T| \asymp \log_{1+\delta} n = \frac{\log n}{\log(1+\delta)} \asymp \frac{\log n}{(\log\log n)^{-1}},
\end{align*}
since $\log (1+x) \asymp x$ as $x \to 0$. Thus, the dominant behavior is determined by the factor $\log n \asymp e^{p^{a(1+o(1))}}$. This exponentially large quantity is 
the correct order of magnitude for the number of candidate changepoint locations in the grid needed for the lower bound on $\rho^*_{\mathrm{1-side}}$ to match the upper bound. However, computing a contrast as in (\ref{ycontrast}) for $t^* \in \T$ results in a collection of dependent variables $(Y_{jt^*})$. In order to ensure a normal limiting distribution for a sum of dependent terms involving the variables $(Y_{jt^*})$, it suffices to dampen the dependence between these variables by further spacing out the candidate changepoint locations. In order to retain tightness of the lower bound construction,
this dampening needs to be done while maintaining the overall exponential order of $e^{p^{a(1+o(1))}}$ many candidate changepoint locations. This is accomplished by using the grid,
\begin{align}
\label{lowerT}
\T' \coloneqq \left\{\lfloor (\log n)^0 \rfloor, \lfloor (\log n)^1 \rfloor,\dots,\lfloor (\log n)^{\log_{\log n}(n-1)} \rfloor\right\}.
\end{align}
Ignoring the symmetry of the grid $\T$, both of the grids $\T$ and $\T'$ are (roughly) of the form $\{b^0,b^1,\dots,b^{\log_b n}\}$ with bases $b = 1+\delta$ and $b=\log n$ corresponding to $\T$ and $\T'$ respectively. The cardinality of $\T'$ behaves as
\begin{align*}
|\T'| \asymp \log_{\log n} n = \frac{\log n}{\log\log n},
\end{align*}
meaning that the dominant behavior of $|\T'|$ is still determined by the factor $\log n \asymp e^{p^{a(1+o(1))}}$. 

The worst case testing error in (\ref{minmaxerr}) is lower bounded by the Bayes testing error of the following simple vs simple testing problem,
\begin{align}
\label{lbconstruction0}
&H_0 : (X_{j:})_{j\in[p]} \sim \prod_{j=1}^p N(0,I_n)\\
\label{lbconstruction1}
&H_1: (X_{j:})_{j \in [p]} \sim \frac{1}{|\T'|} \sum_{k=1}^{|\T'|} \prod_{j=1}^p \left[(1-\e)N(0,I_n)+\e N(\rho\theta^{(k)},I_n) \right],
\end{align}
where $\theta^{(k)} \in \R^n$ is the unique vector associated to the contrast (\ref{ycontrast}). That is, $\langle X_{j:},\theta^{(k)}\rangle = Y_{j t_k}$, where $t_k \coloneqq \lfloor (\log n)^k\rfloor \in \T'$ is the $k^{th}$ grid element, and $\e=p^{-\beta}$ is the average fraction of non-null rows\footnote{In order for the realization of $\theta$ drawn in $H_1$ to be supported on $\Theta_1^{\mathrm{1-side}}(p,n,\rho,s)$ with high probability, $\e$ should technically be replaced with $\bar{\e}=p^{-\bar{\beta}}$ for some slightly smaller $\bar{\beta} < \beta$, where `slightly smaller' is described more precisely in Section 1.1 of the supplemental file.}.
\begin{remark}
\label{caiwuremark}
A general framework for deriving the detection boundary in an iid testing problem of the form,
\begin{align*}
H_0: Y_i \stackrel{\iid}{\sim} Q_n \hspace{1em} \mathrm{versus} \hspace{1em} H_1: Y_i \stackrel{\iid}{\sim} (1-\e_n) Q_n + \e_n G_n, \tag{$i=1,\dots,n$}
\end{align*}
where $Q_n$ and $G_n$ are distributions on $\R$, can be found in the paper by \cite{cai2014optimal}. We note that the high dimensional changepoint problem considered here falls outside of this setting. Indeed, in the Bayesian two-point formulation (\ref{lbconstruction0}) versus (\ref{lbconstruction1}), each row $X_{j:}$ in the observed matrix $X \in \R^{p \times n}$ is not a scalar random variable, and the rows $(X_{j:})_{j\leq p}$ are not independent. In this fomulation, the true changepoint location $k^* \in \T'$ is drawn uniformly from a set of candidate locations. Marginalizing over the uniformly drawn changepoint location yields a distribution on $X$ for which the rows $(X_{j:})_{j\leq p}$ are dependent.
\end{remark}
An optimal procedure for the simple-versus-simple testing problem in (\ref{lbconstruction0}) and (\ref{lbconstruction1}) is one that thresholds the likelihood ratio
\begin{align*}
\psi_{\LRT}(X) = \textbf{1}\left\{\frac{f_1}{f_0}(X) > 1\right\}.
\end{align*}
Thus, the detection boundary can be derived by studying the asymptotic behavior of $\frac{f_1}{f_0}(X)$, where $f_0$ and $f_1$ are the densities associated with the distributions in $H_0$ and $H_1$ respectively. For most cases in the argument, this involves computing the second moment of the likelihood ratio
and choosing the signal size $\rho$ small enough for this moment to be no larger than the typical size of $\frac{f_1}{f_0}(X)$ under the $H_0$. In this setting, the test $\psi_{\LRT}$ would be powerless in detecting a deviation from the null, implying the failure of an optimal test and thus the hardness of the testing problem. In the case $1-4\beta/3 < a  \leq 1-\beta$, we used a different argument based on truncation; the details of this calculation can be found in Case 3 of Section 1.1 of the supplemental file.



\paragraph{Acknowledgements.} The authors would like to thank Richard Samworth for thoughtful discussions and for his help in improving the organization of the paper.

\paragraph{Funding.}
The research of CG was supported in part by NSF Grants ECCS-2216912 and DMS-2310769, NSF Career Award DMS-1847590, and an Alfred Sloan fellowship.


\bibliographystyle{apalike}
\bibliography{reference}

\appendix 

\section{Proofs}
\label{proofs}

\noindent \textbf{Organization of the proofs.} Recall that $\Theta^{\mathrm{1-side}}_1(p,n,\rho,s) \subseteq\Theta_1^{\mathrm{2-side}}(p,n,\rho,s)$ implies the two sided changepoint detection problem is at least as difficult as the one sided version. Note that the detection boundaries only differ when $a \leq 1-2\beta$. Thus for the lower bounds, it suffices to prove the lower bound in the one-sided problem for all cases (Appendix \ref{LBmain1}), and prove the case $a \leq 1-2\beta$ in the two-sided problem (Appendix \ref{LBmain2}). Similarly, for the upper bounds, it suffices to prove the upper bound in the two-sided problem in all cases (Appendix \ref{UBmain2}), and prove the upper bound holds in the case $a \leq 1-2\beta$ for the one-sided problem (Appendix \ref{UBmain1}). Throughout, we let $Z$ denote a standard normal variable, whose dimensions should be clear from the context.

\subsection{Lower Bound for Theorem \ref{main1}}
\label{LBmain1}
The detection boundary for the one-sided changepoint is,
\begin{align*}
\rho_{\mathrm{1-side}}^*(a,\beta)^2 &\coloneqq \begin{cases}
p^{a-(1-2\beta)} \hspace{1em} &a \leq 1-2\beta \\
(a-(1-2\beta)) \log p &1-2\beta < a \leq 1-4\beta/3 \\
2(\sqrt{1-a}-\sqrt{1-a-\beta})^2 \log p &1-4\beta/3 < a \leq 1-\beta \\
p^{a-(1-\beta)} &a > 1-\beta.
\end{cases}
\end{align*}
Put $\rho \coloneqq \rho_{\mathrm{1-side}}^*(a,\beta_1)$ for $a > 0$ and $\beta_1 \in (0,1)$. We will show that if $\beta > \beta_1$, then $\mathcal{R}_{\mathrm{1-side}}(p,\rho,a,\beta) \not\to 0$, as $p \to \infty$. Consider the testing problem,
\begin{align*}
&H_0: X_{j:} \stackrel{\iid}{\sim} N(0,I_n) \\
&H_1: k^* \sim \mathrm{unif}\{1,\dots,|\T|\}, \; X_{j:} \mid k^* \stackrel{\iid}{\sim} (1-\e)N(0,I_n)+ \e N(\rho \theta^{(k^*)}, I_n),
\end{align*}
where $\e \coloneqq p^{-\beta}$, $\T$ is the base $b \coloneqq \log n \asymp e^{p^{a(1+o(1))}}$ grid, defined,
\begin{align*}
\mathcal{T} \coloneqq \{\lfloor b^1\rfloor,\lfloor b^2 \rfloor, \dots, \lfloor b^{\log_{\log n}n}\rfloor \},
\end{align*}
and $\theta^{(k)} \in\R^n$ is the unique vector for which
\begin{align*}
\langle Z,\theta^{(k)} \rangle = \sqrt{\frac{\lfloor b^k\rfloor(n-\lfloor b^k\rfloor)}{n}} (\bar{Z}_{1:\lfloor b^k\rfloor}-\bar{Z}_{\lfloor b^k\rfloor+1:n}).
\end{align*}
More explicitly,
\begin{align*}
\theta_t^{(k)} \coloneqq \begin{cases}
\left(\frac{n-\lfloor b^k\rfloor}{n\lfloor b^k\rfloor} \right)^{1/2} \hspace{1em} &\text{if } t \leq b^k \\
-\left(\frac{\lfloor b^k\rfloor}{n(n-\lfloor b^k\rfloor)} \right)^{1/2} &\text{if } t > b^k .
\end{cases}
\end{align*}
By Part 2 of Lemma \ref{thetas}, $\langle \theta^{(k)},\rho \theta^{(k)}\rangle = \rho$. Consequently, by Part 1 of Lemma \ref{thetas} and the definition of the signal in $\Theta^{\mathrm{1-side}}_1(p,n,\rho,s)$, if the draw of $\theta$ described in $H_1$ has no more than $s\sim p^{1-\beta}$ non-null rows, then $\theta \in \Theta^{\mathrm{1-side}}_1(p,n,\rho,s)$. Note however that the number of non-null rows under $H_1$ is distributed Binomial$(p,\e)$, which implies that the resulting (random) instance of $\theta \in \R^{p \times n}$ is not necessarily an element of $\Theta_1^{\mathrm{1-side}}(p,n,\rho,s)$. To remedy this issue, note that since $\beta > \beta_1$, any fixed $\bar{\beta} \in (\beta_1,\beta)$ satisfies $\bar{\e}\coloneqq p^{-\bar{\beta}} \gg p^{-\beta}$ as $p \to \infty$. Consider the modified testing problem,
\begin{align}
\nonumber
&H_0: X_{j:} \stackrel{\iid}{\sim} N(0,I_n) \\
\label{h1bar}
&\bar{H}_1: k^* \sim \mathrm{unif}\{1,\dots,|\T|\}, X_{j:} \mid k^* \stackrel{\iid}{\sim} (1-\bar{\e})N(0,I_n)+ \bar{\e} N(\rho \theta^{(k^*)}, I_n),
\end{align}
and define the event,
\begin{align*}
G \coloneqq \left\{\sum_{j=1}^{p}\textbf{1}_{\{\text{row $j$ is non-null}\}} \geq \frac{p \bar{\e} }{\log p}\right\}. 
\end{align*}
Each row $X_{j:}$ is non-null with probability $\bar{\e}=p^{-\bar{\beta}} =\omega (p^{-\beta})$ in (\ref{h1bar}). Thus $\frac{p\bar{\e}}{\log p} = \omega(p^{1-\beta})$, and on the set $G$, the realizations of $\theta$ belong to the parameter space $\Theta_1^{\mathrm{1-side}}(p,n,\rho,s)$. It follows that the minimax testing error is lower bounded, 
\begin{align}
\nonumber
\mathcal{R}_{\mathrm{1-side}}(p,n,\rho,s) &\coloneqq \inf_\psi \left[\sup_{\theta \in \Theta_0(p,n)} \P_\theta \psi + \sup_{\theta \in \Theta_1^{\mathrm{1-side}}(p,n,\rho,s)} \P_\theta(1-\psi) \right] \\
\nonumber
&\geq \inf_\psi \left[ \P_0 \psi + \bar{\P}_1(1-\psi)\textbf{1}_G \right]  \\
\label{pgc}
&\geq \inf_\psi \left[ \P_0 \psi + \bar{\P}_1(1-\psi) \right] - \bar{\P}_1(G^c) \\
\nonumber
&= 1-\TV(\P_0,\bar{\P}_1) - \bar{\P}_1(G^c) \tag{Neyman--Pearson},
\end{align}
where $\bar{\P}_1$ denotes the data distribution under $\bar{H}_1$. By Chebyshev's inequality, $\bar{\P}_1(G^c) \to 0$, so it suffices to show that $\liminf_{p\to\infty}(1-\TV(\P_0,\bar{\P}_1)) \geq c$, for any $c \in (0,1)$. To this end, the testing problem can equivalently be written,
\begin{align}
\label{modified-testing-problem}
&H_0 : (X_{j:})_{j\in[p]} \sim \prod_{j=1}^p N(0,I_n)\\
&\bar{H}_1: (X_{j:})_{j \in [p]} \sim \frac{1}{|\T|} \sum_{k=1}^{|\T|} \prod_{j=1}^p \left[(1-\bar{\e})N(0,I_n)+\bar{\e} N(\rho\theta^{(k)},I_n) \right].
\end{align}
Let $f_0$ and $f_1$ be the densities of $X$ corresponding to $H_0$ and $\bar{H}_1$. Note that,
\begin{align*}
1-\TV(\P_0,\bar{\P}_1) = \int f_0 \wedge f_1 \geq \int_{\{f_1 \geq cf_0\}} f_0 \wedge f_1 \geq c\P_0\left(\frac{f_1}{f_0} \geq c \right). \tag{for any $c \in (0,1)$}
\end{align*}
For the right hand side to tend to $c$, it suffices to show that
\begin{align}
\label{eq:lr-tends-to-1}
\frac{f_1}{f_0}(X) \stackrel{\P_0}{\to} 1.
\end{align}
The above convergence is implied by the condition, $\limsup_{p\to\infty} \E_0(\frac{f_1}{f_0}(X))^2 \leq 1$, since by Markov's inequality,
\begin{align*}
\P_0\left(\left|\frac{f_1}{f_0}(X) - 1 \right| > \d \right) &\leq \frac{\E_0(\frac{f_1}{f_0}(X))^2 - 2\E_0(\frac{f_1}{f_0}(X))+1 }{\d^2}   = \frac{\E_0(\frac{f_1}{f_0}(X))^2 - 1}{\d^2}.
\end{align*}
The form of the likelihood ratio for the testing problem \eqref{modified-testing-problem} is 
\begin{align*}
    \frac{f_1}{f_0}(X) &= \frac{1}{|\T|} \sum_{k=1}^{|\T|} \prod_{j=1}^p \frac{(1-\bar{\e})\phi(X_{j:})+\bar{\e}\phi(X_{j:}-\rho\theta^{(k)})}{\phi(X_{j:})} \\
    &= \frac{1}{|\T|} \sum_{k=1}^{|\T|} \prod_{j=1}^p \left[1-\bar{\e}+\bar{\e}\exp\left(\rho \langle X_{j:},\theta^{(k)}\rangle-\rho^2/2 \right) \right]\tag{$\|\theta^{(k)}\|=1$}.
\end{align*}
We proceed to analyze the likelihood ratio separately in several cases.

\subsubsection*{Cases 1 and 2: $0<a \leq 1-4\beta_1/3$} 

Letting $\bar{\E}_1$ denote the expectation with respect to $\bar{H}_1$, the second moment of the likelihood ratio is
\begin{align*}
\E_0 \left(\frac{f_1}{f_0}(X) \right)^2 &= \bar{\E}_1 \frac{f_1}{f_0}(X) \\
&= \frac{1}{|\T|} \sum_{k=1}^{|\T|} \bar{\E}_1 \prod_{j=1}^p \left[1-\bar{\e}+\bar{\e}\exp\left(\rho \langle X_{j:},\theta^{(k)}\rangle-\rho^2/2 \right) \right] \\
&= \frac{1}{|\T|} \sum_{k=1}^{|\T|} \E_{l \sim \mathrm{Unif}(\{1,\dots,|\T|\})} \prod_{j=1}^p \left[1-\bar{\e}+\bar{\e} \cdot \bar{\E}_1(\exp(\rho \langle X_{j:},\theta^{(k)}\rangle-\rho^2/2 \mid k^*=l) \right] \\
&= \frac{1}{|\T|^2} \sum_{k,l \leq |\T|} \prod_{j=1}^p \left[1-\bar{\e}+\bar{\e}\cdot  \E(\exp(\rho\langle \rho\theta^{(l)}I_j+Z_{j:},\theta^{(k)}\rangle - \rho^2/2)) \right],
\end{align*}
where $I_j = 1$ indicates that row $j$ has a changepoint, and $I_j=0$ otherwise ($I_j \stackrel{\iid}{\sim} \mathrm{Bern}(\bar{\e})$), and $Z_{j:}\in \R^{n}$ denotes a standard Gaussian vector, i.e. $Z_{ji}\stackrel{\iid}{\sim} N(0,1)$ for $(j,i)\in [p] \times [n]$. Now,
\begin{align*}
\E(\exp(\rho\langle \rho \theta^{(l)}I_j + Z_{j:},\theta^{(k)}\rangle - \rho^2/2) &= (1-\bar{\e}) \E e^{\rho U-\rho^2/2} + \bar{\e}\cdot  \E e^{\rho (\rho \langle \theta^{(l)},\theta^{(k)}\rangle + U)-\rho^2/2} \\
&= 1+\bar{\e} (\exp(\rho^2\langle \theta^{(l)},\theta^{(k)}\rangle)-1) \\
&\leq 1 + \bar{\e} (\exp(\rho^2 e^{-\frac{|k-l|}{2}p^{a(1+o(1))}})-1),
\end{align*}
where the last inequality follows by part 3 of Lemma \ref{thetas}, and $U \sim N(0,1)$ denotes a standard Gaussian variable. Noting that $\frac{|k-l|}{2}=0$ when $k=l$, we apply this bound to the expression for the second moment, obtaining
\begin{align*}
\E_0\left(\frac{f_1}{f_0}(X) \right)^2 &\leq \frac{1}{|\T|^2} \sum_{k=l} \prod_{j=1}^p \left[1+\bar{\e}^2 (e^{\rho^2}-1) \right] \\
&+ \frac{1}{|\T|^2} \sum_{k\neq l} \prod_{j=1}^p \left[1+\bar{\e}^2 (\exp(\rho^2 e^{-\frac{|k-l|}{2}p^{a(1+o(1))}})-1) \right].
\end{align*}
Since $k\neq l$ and $a>0$ imply that $\rho^2 e^{-\frac{|k-l|}{2}p^{a(1+o(1))}} \leq \rho^2 e^{-\frac{1}{2}p^{a(1+o(1))}} \to 0$ as $p\to\infty$, we have that when $k\neq l$,
\begin{align*}
\prod_{j=1}^p \left[1+\bar{\e}^2 (\exp(\rho^2 e^{-\frac{|k-l|}{2}p^{a(1+o(1))}})-1) \right] &\leq \left[1+2\bar{\e}^2 \rho^2 e^{-\frac{1}{2}p^{a(1+o(1))}} \right]^p \\
&\leq \exp\left( 2p\bar{\e}^2\rho^2 e^{-\frac{1}{2}p^{a(1+o(1))}}\right),
\end{align*}
where the first inequality follows from $e^{x}-1<2x$ for $0<x<1$, and the second inequality follows from $\log(1+x)\leq x$ for $x \in \R$. Note that $ 2p\bar{\e}^2\rho^2 e^{-\frac{1}{2}p^a}\to 0$ as $p\to \infty$, so the right hand side of the above tends to 1. Plugging the above back into the bound for the second moment, we obtain
\begin{align*}
\E_0\left(\frac{f_1}{f_0}(X) \right)^2 &\leq \frac{1}{|\T|^2} \sum_{k=l} \prod_{j=1}^p \left[1+\bar{\e}^2 (e^{\rho^2}-1) \right] + \frac{1}{|\T|^2} \sum_{k\neq l}\exp\left( 2p\bar{\e}^2\rho^2 e^{-\frac{1}{2}p^{a(1+o(1))}}\right) \\
&= \frac{1}{|\T|} \exp(p \bar{\e}^2 (e^{\rho^2}-1)) + 1+o(1).
\end{align*}
Recall $|\T| \asymp e^{p^{a(1+o(1))}}$. In order for the right hand side to tend to 1, it suffices for the first term to tend to zero, i.e. $\exp(p \bar{\e}^2 (e^{\rho^2}-1)-p^{a(1+o(1))}) \to 0$. Since $\rho^2 = \rho^*_{\mathrm{1-side}}(a,\beta_1)^2 \sim \log(1+p^{a-(1-2\beta_1)})$ when $a \neq 1-2\beta_1$, we have
\begin{align*}
\exp\left(p \bar{\e}^2 (e^{\rho^2}-1)-p^{a(1+o(1))}\right) &= \exp\left(p \bar{\e}^2 p^{(a-(1-2\beta_1))(1+o(1))}-p^{a(1+o(1))}\right) \\
&= \exp\left(p^{(a+2(\beta_1-\bar{\beta}))(1+o(1))} - p^{a(1+o(1))} \right) \\
&\to 0. \tag{$\bar{\beta}>\beta_1$}
\end{align*}
When $a = 1-2\beta_1$, we have $\rho^2 = 1$, so that
\begin{align*}
\exp\left(p \bar{\e}^2 (e^{\rho^2}-1)-p^{a(1+o(1))}\right) &= \exp\left(p \bar{\e}^2 (e-1)-p^{a(1+o(1))}\right) \\
&= \exp\left(p^{(1-2\bar{\beta})(1+o(1))} - p^{(1-2\beta_1)(1+o(1))} \right) \\
&\to 0. \tag{$\bar{\beta}>\beta_1$}
\end{align*}
\subsubsection*{Case 3: $1-4\beta_1/3<a\leq 1-\beta_1$}

Put $x \coloneqq \sqrt{1-a}$ and $y\coloneqq \sqrt{1-a-\beta_1}$. Then the signal $\rho$ is of the form, $\rho = (x-y)\sqrt{2\log p}$.
The likelihood ratio test for (\ref{h1bar}) is defined by
\begin{align*}
\psi_{\LRT}(X) \coloneqq \textbf{1}\left\{\frac{f_1}{f_0}(X) > 1\right\},
\end{align*}
where, recall that $f_1$ denotes the density of $X$ under $\bar{H}_1$. By (\ref{pgc}), the minimax testing risk for the original testing problem is asymptotically lower bounded by the testing risk of $\psi_{\LRT}$,
\begin{align*}
\mathcal{R}_{\mathrm{1-side}}(p,n,\rho,s) \geq \inf_{\psi}[\P_{0} \psi + \bar{\P}_{1}(1-\psi)] - o(1) = \P_0 \psi_\LRT + \bar{\P}_1(1-\psi_{\LRT}) - o(1).
\end{align*}
The likelihood ratio is of the form,
\begin{align*}
\frac{f_1}{f_0}(X) &= \frac{1}{|\T|} \sum_{k=1}^{|\T|}L_k \\
L_k &\coloneqq \prod_{j=1}^p \big[1-\bar{\e}+\bar{\e} e^{\rho \langle X_{j:},\theta^{(k)}\rangle - \rho^2/2} \big].
\end{align*}
By Lemma \ref{lem-LRT-error-lower-bd}, the error of the likelihood ratio test is bounded from below by
\begin{align*}
\P_0(\psi_{\LRT}=1) + \bar{\P}_1(\psi_{\LRT}=0) &\geq \frac{1}{3}\bar{\P}_1(A_3),
\end{align*}
where $A_3 \coloneqq \left\{ \frac{f_1}{f_0}(X) \leq 3 \right\}$. Now note
\begin{align*}
\bar{\P}_1 \left( A_3 \right) \geq \bar{\P}_1 \left( \frac{1}{|\T|} \sum_{k \neq k^*} L_k \leq 2, L_{k^*} \leq |\T| \right),
\end{align*}
where, recall that $k^* \in \{1,\dots,|\T|\}$ is the realized changepoint location defined in \eqref{h1bar}, and the notation $k\neq k^*$ is shorthand for the set of indices $\{k\leq |\T| : k\neq k^*\}$. Applying the tower property and the conclusion of Lemma \ref{lem-Lk}, we have
\begin{align*}
\bar{\E}_1 \left( \frac{1}{|\T|} \sum_{k\neq k^*} L_k \right) &= \bar{\E}_1 \left(\frac{1}{|\T|} \sum_{k\neq k^*} \bar{\E}_1 (L_k \mid k^*) \right)\\
&= \frac{1}{|\T|} \cdot (|\T|-1)\left(1+ O(p \e^2 \rho^2e^{-\frac{1}{2}p^{a(1+o(1))}})\right) \sim 1,
\end{align*}
which is $\leq 1 + o(1)$ for $a>0$ as $p \to \infty$. It now follows from Markov's inequality that 
\begin{align*}
\P_1 \left( \frac{1}{|\T|} \sum_{k \neq k^*} L_k \leq 2 \right) = 1-\P_1 \left( \frac{1}{|\T|} \sum_{k \neq k^*} L_k > 2 \right) \geq \frac{1}{2}-o(1).
\end{align*}
Thus, to show that $\bar{\P}_1(A_3)$ is asymptotically no smaller than $1/2$, it suffices to show
\begin{align}
\label{Lk-star-prob}
\P_1(L_{k^*} > |\T|) \to 0.
\end{align}
To this end, we split the indices $j \in [p]$ into four sets as follows. Letting $Q_j = \textbf{1}_{\{\text{row $j$ is non-null}\}}$, put
\begin{align*}
\Gamma_0 &\coloneqq \left\{j : Q_j = 0\right\} \\
\Gamma_1 &\coloneqq \left\{j : Q_j = 1, \langle X_{j:},\theta^{(k^*)} \rangle \leq \sqrt{2(1-a)\log p}\right\} \\
\Gamma_2 &\coloneqq \left\{j : Q_j = 1, \sqrt{2(1-a)\log p} < \langle X_{j:},\theta^{(k^*)} \rangle \leq 2\sqrt{2\log p} \right\} \\
\Gamma_3 &\coloneqq \left\{j : Q_j = 1, \langle X_{j:},\theta^{(k^*)} \rangle > 2\sqrt{2\log p}\right\} .
\end{align*}
Then $\log L_{k^*} = \sum_{i=0}^3 R_i$, where 
\begin{align}
\label{sum-Ri}
R_i \coloneqq \sum_{j \in \Gamma_i} \log \left( 1 + \bar{\e} \left( \exp\left( \rho \langle X_{j:},\theta^{(k^*)}\rangle -\rho^2/2\right)-1\right) \right), 
\end{align}
for $i=0,1,2,3$, and the probability in condition (\ref{Lk-star-prob}) is bounded,
\begin{align*}
\bar{\P}_1(L_{k^*} > |\T|) &= \bar{\P}_1\left( \sum_{i=0}^3 R_i > \log |\T| \right) \leq \sum_{i=0}^3 \bar{\P}_1\left(R_i > \frac{1}{4}\log |\T|\right).
\end{align*}
The result now follows upon showing that
\begin{align}
\label{i012}
\bar{\P}_1\left(R_i > \frac{1}{4}\log |\T|\right) \to 0, \hspace{1em} \text{for } i=0,1,2,3. 
\end{align}

For $i=0$, note that since the rows $X_{j:}$ given $Q_j=0$ are generated conditionally independently from the $N(0,I_n)$ distribution, we have $\langle X_{j:},\theta^{(k^*)}\rangle \sim N(0,1)$ independently across $j=1,\dots,p$, so that
\begin{align*}
\bar{\E}_1 e^{R_0} &= \bar{\E}_1 (\bar{\E}_1 (e^{R_0} \mid \Gamma_0)) = \bar{\E}_1 \prod_{j \in \Gamma_0}\E (1+\bar{\e} (e^{\rho U_j-\rho^2/2}-1)) = 1,
\end{align*}
where $U_j \stackrel{\iid}{\sim} N(0,1)$ independently of $|\Gamma_0| \sim \mathrm{Binomial}(p,1-\bar{\e})$. Thus by Markov's inequality,
\begin{align}
\label{i0}
\bar{\P}_1\left(R_0 > \frac{1}{4}\log |\T|\right) = \bar{\P}_1\left(e^{R_0} > |\T|^{1/4} \right) \leq \frac{1}{|\T|^{1/4}} \to 0.
\end{align}

For $i=1$, note that each summand in (\ref{sum-Ri}) is no smaller than $\log\left( 1+ \bar{\e} \left(0-1 \right)\right)=\log(1-~\bar{\e})$. Since there are $|\Gamma_1|$ many summands in $R_1$, the difference $R_1 - |\Gamma_1|\log(1-\bar{\e}) > 0$ is positive. Thus we may apply Markov's inequality together with $\log(1-\bar{\e})<0$ to obtain
\begin{align*}
\P\left(R_1 > \frac{1}{4} \log |\T| \right) \leq \P\left(R_1- |\Gamma_1|\log(1-\bar{\e}) > \frac{1}{4} \log |\T| \right) \leq \frac{4\bar{\E}_1(R_1-|\Gamma_1|\log(1-\bar{\e}))}{\log |\T|},
\end{align*}
Next, recall that $\log(1-\bar{\e}) \asymp -\bar{\e}$ as $\bar{\e} \to 0$ and $\log|\T| \asymp p^{a(1+o(1))}$. Thus, in order for the right hand side of the above to go to zero, it suffices to show
\begin{align}
\label{r1-cond1}
p^{-a(1+o(1))} \cdot \bar{\E}_1 R_1 &\to 0 \\
\label{r1-cond2}
p^{-a(1+o(1))-\bar{\beta}} \cdot \bar{\E}_1 |\Gamma_1| &\to 0.
\end{align}
To show \eqref{r1-cond1}, note that
\begin{align*}
\bar{\E}_1 R_1 &= \bar{\E}_1 \sum_{j \in \Gamma_1} \log \left(1+\bar{\e}\left(e^{\rho \langle X_{j:},\theta^{(k^*)}\rangle-\rho^2/2}-1 \right) \right) \\
&= \sum_{j=1}^p \bar{\E}_1\left[\textbf{1}_{\{Q_j=1,\langle X_{j:},\theta^{(k^*)}\rangle \leq \sqrt{2(1-a)\log p}\}} \log \left(1+\bar{\e}\left(e^{\rho \langle X_{j:},\theta^{(k^*)}\rangle-\rho^2/2}-1 \right) \right) \right] \\
&\leq \sum_{j=1}^p \bar{\E}_1\left[\textbf{1}_{\{Q_j=1,\langle X_{j:},\theta^{(k^*)}\rangle \leq \sqrt{2(1-a)\log p}\}} \bar{\e}\cdot e^{\rho \langle X_{j:},\theta^{(k^*)}\rangle-\rho^2/2} \right] \\
&\leq \sum_{j=1}^p \bar{\e}^2 \cdot\E \left[\textbf{1}_{\{\rho+Z \leq \sqrt{2(1-a)\log p}\}} \cdot e^{\rho (\rho+Z)-\rho^2/2} \right] \tag{$Z\sim N(0,1)$},
\end{align*}
since conditional on $Q_j=1$, we have that $\langle X_{j:},\theta^{(k^*)}\rangle \stackrel{(d)}{=}\rho+Z$ where $Z\sim N(0,1)$ independently of $Q_j$. Directly integrating, the above is equal to
\begin{align}
\nonumber
&= p \bar{\e}^2 e^{\rho^2/2} \int_{-\infty}^{\sqrt{2(1-a)\log p} - \rho} e^{\rho z} \phi(z) dz \\
\nonumber
&= p \bar{\e}^2 e^{\rho^2} \int_{-\infty}^{\sqrt{2(1-a)\log p} - \rho} \phi(z-\rho) dz \\
\label{2y-x}
&= p \bar{\e}^2 e^{\rho^2} \Phi(\sqrt{2(1-a)\log p} - 2\rho) .
\end{align}
Recall that $\rho = \sqrt{2\log p}\cdot (\sqrt{1-a}-\sqrt{1-a-\beta_1}) =: \sqrt{2\log p}\cdot (x-y)$. The condition $a > 1-4\beta_1/3$ implies 
\begin{align*}
\sqrt{2(1-a)\log p} - 2\rho = \sqrt{2\log p}\cdot (2y-x)=\sqrt{2\log p} \cdot (2\sqrt{1-a-\beta_1}-\sqrt{1-a}) < 0,
\end{align*}
and so expression (\ref{2y-x}) becomes (up to log factors in $p$)
\begin{align*}
&= p \bar{\e}^2 e^{\rho^2} \Phi(\sqrt{2\log p}\cdot (x-2(x-y))) = p^{1-2\bar{\beta}+2(x-y)^2-(2y-x)^2} =p^{1-2\bar{\beta}-2y^2+x^2},
\end{align*}
from which (\ref{r1-cond1}) follows, since $\bar{\beta}>\beta_1$ implies
\begin{align*}
p^{-a(1+o(1))} \cdot \bar{\E}_1 R_1 = p^{-a(1+o(1)) +1-2\bar{\beta}-2(1-a-\beta_1)+1-a} = p^{-2(\bar{\beta}-\beta_1)} \to 0.
\end{align*}
For \eqref{r1-cond2}, note that
\begin{align*}
\bar{\E}_1 |\Gamma_1| &= \sum_{j=1}^p \bar{\P}_1(Q_j=1,\langle X_{j:},\theta^{(k^*)}\rangle \leq \sqrt{2(1-a)\log p}) \\
&= p^{1-\bar{\beta}} \P(\rho+Z\leq \sqrt{2(1-a)\log p}) \tag{$Z\sim N(0,1)$}  \\
&\leq p^{1-\bar{\beta}}.
\end{align*}
(\ref{r1-cond2}) now follows since $\bar{\beta} > \beta_1$ implies
\begin{align*}
p^{-a(1+o(1))-\bar{\beta}} \cdot \bar{\E}_1 |\Gamma_1| \leq p^{-a(1+o(1))+1-2\bar{\beta}} \leq p^{-a(1+o(1))+1-2\beta_1} \to 0,
\end{align*}
because $a>1-4\beta_1/3 > 1-2\beta_1$. We have now shown (\ref{r1-cond1}) and (\ref{r1-cond2}), which imply that condition (\ref{i012}) holds with $i=1$. 

Next we check condition (\ref{i012}) for $i=2$. To this end, first note that each summand in $R_2$ is positive, because $\langle X_{j:},\theta^{(k^*)} \rangle > \sqrt{2(1-a) \log p}$ implies
\begin{align*}
\rho \langle X_{j:},\theta^{(k^*)} \rangle - \rho^2/2 &> \rho \sqrt{2(1-a)\log p} - \rho^2/2 \\
&= \rho \sqrt{2\log p} \cdot \left(\sqrt{1-a}-\frac{1}{2}(\sqrt{1-a}-\sqrt{1-a-\beta_1})\right) > 0,
\end{align*}
which implies $\log \left( 1 + \bar{\e} \left( \exp\left( \rho \langle X_{j:},\theta^{(k^*)}\rangle -\rho^2/2\right)-1\right) \right) > 0$. Thus Markov's inequality can be applied to obtain
\begin{align}
\label{markov-r2}
\P\left(R_2 > \frac{1}{4}\log |\T| \right) \leq \frac{4\bar{\E}_1 R_2}{\log |\T|} 
\end{align}
By construction, if $j \in \Gamma_2$, then $\langle X_{j:},\theta^{(k^*)}\rangle \leq 2\sqrt{2\log p}$, which implies
\begin{align*}
\bar{\E}_1 R_2 &= \sum_{j=1}^p \bar{\E}_1 \left[ \textbf{1}_{\{j \in \Gamma_2\} } \log \left(1+\bar{\e}\left(e^{\rho \langle X_{j:},\theta^{(k^*)}\rangle-\rho^2/2}-1 \right) \right) \right] \\
&\leq \sum_{j=1}^p \bar{\E}_1 \left[ \textbf{1}_{\{j \in \Gamma_2\} } \log \left(1+\bar{\e}\left(e^{\rho \cdot 2\sqrt{2\log p}-\rho^2/2}-1 \right) \right) \right] \\
&\leq \bar{\E}_1 |\Gamma_2| \cdot \log\left(1+p^{C} \right)
\end{align*}
for some fixed constant $C>0$ as $p\to \infty$. 
To bound $\bar{\E}_1 |\Gamma_2|$, note that
\begin{align*}
\bar{\E}_1  |\Gamma_2| &= \bar{\E}_1 \sum_{j=1}^p \textbf{1}_{\left\{Q_j=1,\sqrt{2(1-a)\log p} < \langle X_{j:},\theta^{(k^*)} \rangle \leq 2\sqrt{2\log p} \right\}} \\
&= \sum_{j=1}^p p^{-\bar{\beta}} \P\left(\sqrt{2(1-a)\log p} < \rho + Z \leq 2\sqrt{2\log p} \right) \tag{$Z\sim N(0,1)$}\\ 
&\leq p^{1-\bar{\beta}} \bar{\Phi}\left(\sqrt{2(1-a)\log p}-\rho\right) \\
&\asymp p^{1-\bar{\beta} - (\sqrt{1-a}-(x-y))^2} \tag{up to log factors in $p$}.
\end{align*}
Plugging this into (\ref{markov-r2}), we have (up to log factors in $p$) that
\begin{align*}
\bar{\P}_1 \left(R_2 > \frac{1}{4}\log |\T| \right) \lesssim \frac{p^{1-\bar{\beta}-(1-a-\beta_1)}}{p^{a(1+o(1))}} \to 0
\end{align*}
since $\beta_1 < \bar{\beta}$.

Finally, we check condition (\ref{i012}) for $i=3$. We have
\begin{align*}
\bar{\P}_1 \left(R_3 > \frac{1}{4}\log |\T| \right) &\leq \bar{\P}_1 \left(R_3 > 0 \right) \\
&= \bar{\P}_1(\Gamma_3\neq \varnothing) \\
&= \bar{\P}_1\left( \bigcup_{j=1}^p \left\{Q_j=1, \langle X_{j:},\theta^{(k^*)}\rangle > 2\sqrt{2\log p} \right\} \right) \\
&\leq \sum_{j=1}^p \bar{\P}_1\left(Q_j=1, \langle X_{j:},\theta^{(k^*)}\rangle > 2\sqrt{2\log p}  \right) \\
&= p\bar{\e}\cdot  \P(\rho+Z>2\sqrt{2\log p}) \\
&= p^{1-\bar{\beta}} \bar{\Phi}(\sqrt{2\log p}\cdot (2-(x-y))) \\
&= p^{1-\bar{\beta}-(2-(x-y))^2}, \tag{up to log factors in $p$}
\end{align*}
which tends to zero polynomially in $p$ if the exponent $1-\bar{\beta}-(2-(x-y))^2 < 0$. This condition is equivalent to
\begin{align*}
\sqrt{1-\bar{\beta}} + \sqrt{1-a}-\sqrt{1-a-\beta_1}< 2,
\end{align*}
since $x=\sqrt{1-a}$ and $y=\sqrt{1-a-\beta_1}$. This condition is satisfied because the left hand side is
\begin{align*}
\sqrt{1-\bar{\beta}} + \sqrt{1-a-\beta_1+\beta_1}-\sqrt{1-a-\beta_1}<\sqrt{1-\bar{\beta}} +\sqrt{\beta_1} < 2,
\end{align*}
because $\bar{\beta},\beta_1 \in (0,1)$. 

We have now shown (\ref{i012}) for each of $i=0,1,2,3$, completing the proof of this case.

\subsubsection*{Case 4: $a > 1-\beta_1$}

For this case, a slightly different construction is used, in which the 
signal is distributed evenly across the first $s\sim p^{1-\beta}$ rows (recall that $\rho^2 = p^{a-(1-\beta_1)}$). The hypothesis testing problem is
\begin{align*}
&H_0: (X_{j:})_{j\leq p} \sim \prod_{j=1}^p N(0,I_n) \\
&H_1: k^* \sim \mathrm{Unif}\{1,\dots,|\T|\}, (X_{j:})_{j\leq s} \mid k^* \sim \prod_{j=1}^s N(\rho \theta^{(k^*)},I_n),\hspace{1em} (X_{j:})_{j > s} \sim \prod_{j=s+1}^p N(0,I_n).
\end{align*}
Letting $\E_0$ and $\E_1$ (resp. $f_0$ and $f_1$) denote expectation (resp. density) with respect to the above two data generating mechanisms, the second moment of the likelihood ratio is
\begin{align*}
\E_0 \left(\frac{f_1}{f_0}(X)\right)^2 &= \E_1 \frac{f_1}{f_0}(X) \\
&= \frac{1}{|\T|} \sum_{k=1}^{|\T|}\E_1 \prod_{j=1}^s \frac{\phi(X_{j:}-\rho\theta^{(k)})}{\phi(X_{j:})} \\
&= \frac{1}{|\T|^2} \sum_{k,l} \prod_{j=1}^s  \E_1(\exp(\rho \langle X_{j:},\theta^{(k)}\rangle - \rho^2/2) \mid k^*=l) \\
&= \frac{1}{|\T|^2} \sum_{k,l} \prod_{j=1}^s  \E(\exp(\rho \langle \rho \theta^{(l)}+Z_{j:},\theta^{(k)}\rangle - \rho^2/2)),
\end{align*}
where $Z_{j:}\stackrel{\iid}{\sim}N(0,I_n)$, and $\E$ denotes expectation taken with respect to these standard Gaussian vectors. Now since $\langle Z_{j:},\theta^{(k)}\rangle \sim N(0,1)$ for any $j,k$, we have the above is equal to
\begin{align*}
&= \frac{1}{|\T|^2} \sum_{k,l} e^{s\rho^2 \langle \theta^{(k)},\theta^{(l)}\rangle} \\
&\leq \frac{1}{|\T|^2}\sum_{k=l} e^{p^{1-\beta} \rho^2} + \frac{1}{|\T|^2}\sum_{k\neq l} \exp\left(p^{1-\beta} e^{-\frac{|k-l|}{2}p^a}\right) \tag{Lemma \ref{thetas}}\\
&\leq \frac{1}{|\T|} e^{p^{1-\beta+a-(1-\beta_1)}} + 1+o(1),
\end{align*}
since $\langle \theta^{(k)},\theta^{(l)}\rangle \leq e^{-\frac{|k-l|}{2}p^a}$ according to Lemma \ref{thetas} with $b \sim \log n$ (recall the asymptotic relationship \eqref{3log}). Now since $|\T| = e^{p^{a(1+o(1))}}$, the first term is
\begin{align*}
\frac{1}{|\T|} e^{p^{1-\beta+a-(1-\beta_1)}} =\exp\left(p^{a+(\beta_1-\beta)} - p^{a(1+o(1))} \right) \to 0, \tag{$\beta_1 < \beta$}
\end{align*}
which gives $\limsup_{p\to\infty}\E_0 \left(\frac{f_1}{f_0}(X)\right)^2 \leq 1$, as desired.

\subsection{Lower Bound for Theorem \ref{main2}}
\label{LBmain2}
As discussed in the beginning of Appendix \ref{proofs}, since the detection boundaries for the one-sided and two-sided changepoint problems are the same except for case $a \leq 1-2\beta_1$, it suffices to show \eqref{eq:lr-tends-to-1} holds for the two-sided version of the testing problem problem. To this end, consider
\begin{align*}
&H_0 : (X_{j:})_{j\in[p]} \sim \prod_{j=1}^p N(0,I_n)\\
&\bar{H}_1: (X_{j:})_{j \in [p]} \sim \frac{1}{|\T|} \sum_{k=1}^{|\T|} \prod_{j=1}^p \left[(1-\bar{\e})N(0,I_n)+\bar{\e}\cdot \frac{1}{2} \left(N(\rho\theta^{(k)},I_n)+N(-\rho\theta^{(k)},I_n)\right) \right],
\end{align*}
where $\rho \coloneqq \rho^*_{\mathrm{2-side}}(a,\beta_1)$ and $a \geq 0$ and $\beta_1 < \beta$ and $\bar{\e} = p^{-\bar{\beta}}$, where $\bar{\beta} \in (\beta_1,\beta)$, and $\theta^{(k)}$ is defined as in the proof for Theorem \ref{main1}. In words, under the alternative hypothesis $\bar{H}_1$, the mean vector for a non-null row is either $\rho \theta^{(k)}$ or $-\rho \theta^{(k)}$ for some $k \in \{1,\dots, |\T|\}$. 

By construction of $\theta^{(k)}$ (Part 1 of Lemma \ref{thetas}), for any vector $v_\mu \in \R^n$ whose first $\lfloor b^k\rfloor$ components are equal to $\mu_1$, and remaining $n-\lfloor b^k\rfloor$ components are equal to $\mu_2$, we have
\begin{align*}
|\langle \theta^{(k)},v_\mu \rangle| = \sqrt{\frac{\lfloor b^k\rfloor(n-\lfloor b^k\rfloor)}{n}} |\mu_1-\mu_2|.
\end{align*}
Taking $v_\mu \in \{\rho \theta^{(k)},-\rho \theta^{(k)}\}$, this implies $|\langle \theta^{(k)},v_\mu \rangle| = \rho$. Thus, each non-null mean vector $\theta_{j:}$ generated in $\bar{H}_1$ satisfies the two-sided signal requirement of $\Theta_1^{\mathrm{2-side}}(p,n,\rho,s)$. Together with the discussion on the choice of $\bar{\beta}$ in Appendix \ref{LBmain1}, this shows that the random instance of $\theta$ generated under $\bar{H}_1$ is an element of $\Theta^{\mathrm{2-side}}_1(p,n,\rho,s)$ with high probability.

Since $\bar{\beta} < \beta_1$, the calculation (\ref{pgc}) discussed in Appendix \ref{LBmain1} implies that it is enough to show 
\begin{align}
\label{eq:sec-mom-condition}
\limsup_{p\to\infty} \E_0 \left(\frac{f_1}{f_0}(X)\right)^2 \leq 1,
\end{align}
where $f_0$ (resp. $f_1$) is the density of the data generated by $H_0$ (resp. $\bar{H}_1$). The likelihood ratio can be expressed
\begin{align*}
    \frac{f_1}{f_0}(X) &= \frac{1}{|\T|} \sum_{k=1}^{|\T|} \prod_{j=1}^p \frac{(1-\bar{\e})\phi(X_{j:}) + \bar{\e} \left(\frac{1}{2} \phi(X_{j:}-\rho \theta^{(k)}) + \frac{1}{2} \phi(X_{j:}+\rho \theta^{(k)}) \right)}{\phi(X_{j:})} \\
    &= \frac{1}{|\T|} \sum_{k=1}^{|\T|} \prod_{j=1}^p \left[ 1-\bar{\e} + \bar{\e} e^{-\rho^2/2} \cdot \frac{e^{ \rho \langle X_{j:},\theta^{(k)}\rangle}+e^{-\rho\langle X_{j:},\theta^{(k)}\rangle}}{2}  \right] \\
\end{align*}
By linearity of expectation,
\begin{align*}
\nonumber
\E_0 \left(\frac{f_1}{f_0}(X)\right)^2 &= \bar{\E}_1 \frac{f_1}{f_0}(X) \\
\nonumber
&= \frac{1}{|\T|} \sum_{k=1}^{|\T|} \bar{\E}_1 \prod_{j=1}^p \left[1-\bar{\e}+\bar{\e} \frac{e^{-\rho^2/2}}{2}\left( e^{\rho \langle X_{j:},\theta^{(k)}\rangle} + e^{-\rho \langle X_{j:},\theta^{(k)}\rangle} \right) \right].
\end{align*}
By definition of $\bar{H}_1$, the rows of $X$ are conditionally independent given the changepoint location, so the above is equal to
\begin{align}
\label{2side2m}
&= \frac{1}{|\T|^2} \sum_{k,l \leq |\T|} \prod_{j=1}^p \left[1-\bar{\e}+\bar{\e} \frac{e^{-\rho^2/2}}{2}\bar{\E}_1 \left( e^{\rho \langle X_{j:},\theta^{(k)}\rangle} + e^{-\rho \langle X_{j:},\theta^{(k)}\rangle} \mid k^* = l\right) \right].
\end{align}
Letting $Z_{j:}\stackrel{\iid}{\sim}N(0,I_n)$ for $j=1,\dots,p$, the conditional expectation can be bounded as follows,
\begin{align*}
\bar{\E}_1 (e^{\rho \langle X_{j:},\theta^{(k)}\rangle} \mid k^*=l) &= (1-\bar{\e}) e^{\rho^2/2} + \bar{\e} \cdot \frac{1}{2}\left(\E e^{\rho \langle \rho \theta^{(l)} + Z_{j:},\theta^{(k)}\rangle} + \E e^{\rho \langle -\rho \theta^{(l)}+Z_{j:},\theta^{(k)}\rangle} \right) \\
&= (1-\bar{\e}) e^{\rho^2/2} + \bar{\e} \cdot \frac{e^{\rho^2/2}}{2}\left(e^{\rho^2 \langle \theta^{(l)},\theta^{(k)}\rangle} + e^{-\rho^2 \langle \theta^{(l)},\theta^{(k)}\rangle} \right) \\
&\leq (1-\bar{\e}) e^{\rho^2/2} + \bar{\e} \cdot e^{\rho^2/2}e^{\rho^{4}\langle \theta^{(l)},\theta^{(k)}\rangle^2/2},
\end{align*}
where we have used the inequality $\frac{1}{2}(e^{x}+e^{-x}) \leq e^{x^2/2}$ in the last line. By symmetry we also have
\begin{align*}
\bar{\E}_1 (e^{-\rho \langle X_{j:},\theta^{(k)}\rangle} \mid k^*=l) \leq (1-\bar{\e}) e^{\rho^2/2} + \bar{\e} \cdot e^{\rho^2/2}e^{\rho^{4}\langle \theta^{(l)},\theta^{(k)}\rangle^2/2}.
\end{align*}
Plugging these two estimates back into (\ref{2side2m}), we obtain
\begin{align*}
\E_0 \left( \frac{f_1}{f_0}(X) \right)^2 &\leq \frac{1}{|\T|^2} \sum_{k,l\leq |\T|} \prod_{j=1}^p \left[1-\bar{\e}+\bar{\e} \cdot e^{-\rho^2/2} \left((1-\bar{\e})e^{\rho^2/2}+\bar{\e}\cdot e^{\rho^2/2}e^{\rho^4\langle \theta^{(l)},\theta^{(k)}\rangle^2/2} \right) \right] \\
&\leq \frac{1}{|\T|^2} \sum_{k,l\leq |\T|} \prod_{j=1}^p \left[1-\bar{\e}+\bar{\e} \left(1-\bar{\e}+\bar{\e} e^{\rho^4\langle \theta^{(l)},\theta^{(k)}\rangle^2/2} \right) \right] \\
&= \frac{1}{|\T|^2} \sum_{k,l\leq |\T|} \prod_{j=1}^p \left[1+\bar{\e}^2 \left(e^{\rho^4\langle \theta^{(l)},\theta^{(k)}\rangle^2/2}-1 \right) \right] \\
&\leq \frac{1}{|\T|^2} \sum_{k,l\leq |\T|} \exp\left(p^{1-2\bar{\beta}}(e^{\rho^4\langle \theta^{(l)},\theta^{(k)}\rangle^2/2}-1) \right)\tag{$\log(1+x)\leq x$}\\
&\leq \frac{1}{|\T|^2} \sum_{k=l} \exp\left(p^{1-2\bar{\beta}}(e^{\rho^4/2}-1) \right) + \frac{1}{|\T|^2} \sum_{k\neq l} \exp\left(p\bar{\e}^2(e^{\rho^4 e^{-\frac{1}{2}p^{a(1+o(1))}}/2}-1) \right) ,
\end{align*}
where in the last inequality, we have used $|\langle \theta^{(k)},\theta^{(l)}\rangle | \leq e^{-\frac{1}{2}p^{a(1+o(1))}}$, which follows from Part 3 of Lemma \ref{thetas} with $b=\log n \sim e^{p^a}$. Now since $\rho^4 e^{-\frac{1}{2}p^{a(1+o(1))}} \to 0$ as $p\to \infty$, the inequality $e^{x}-1 \leq 2x$ for $0<x<1$ implies that the above is bounded by
\begin{align*}
&\leq \frac{1}{|\T|} \exp\left(p^{1-2\bar{\beta}}(e^{\rho^4/2}-1) \right) + \frac{1}{|\T|^2} \sum_{k\neq l} \underbrace{\exp\left(p\bar{\e}^2\rho^4 e^{-\frac{1}{2}p^{a(1+o(1))}} \right)}_{1+o(1)}  \tag{$a>0$}\\
&= \exp\left(p^{1-2\bar{\beta}}(e^{\rho^4/2}-1)-p^{a(1+o(1))} \right) + 1+o(1).
\end{align*}
In order for the right hand side to tend to 1, it suffices for the first term to tend to zero, i.e. $\exp(p^{1-2\bar{\beta}}(e^{\rho^4/2}-1)-p^{a(1+o(1))}) \to 0$. Under  $a \leq 1-2\beta_1$, there are two cases; first suppose $a < 1-2\beta_1$. Then
\begin{align*}
\rho^2 =p^{\frac{a-(1-2\beta_1)}{2}}\sim \sqrt{\log(1+p^{a-(1-2\beta_1)})},    
\end{align*}
and the first term becomes
\begin{align*}
\exp(p^{1-2\bar{\beta}}(e^{\rho^4/2}-1)-p^{a(1+o(1))}) &\leq \exp(p^{1-2\bar{\beta}+a-(1-2\beta_1)}-p^{a(1+o(1))}) \to 0,
\end{align*}
since $\beta_1 < \bar{\beta}$. If $a = 1-2\beta_1$, then $\rho^2 = 1$, so that
\begin{align*}
\exp(p^{1-2\bar{\beta}}(e^{\rho^4/2}-1)-p^{a(1+o(1))}) &= \exp(p^{(1-2\bar{\beta})(1+o(1))}-p^{(1-2\beta_1)(1+o(1))}) \to 0,
\end{align*}
since $0<\beta_1<\bar{\beta}$.

\subsection{Upper Bound for Theorem \ref{main2}}
\label{UBmain2}

Put $\beta < \beta_1$ and $\rho \coloneqq \rho^*_{\mathrm{2-side}}(a,\beta_1)$ as in the statement of Theorem \ref{main2}. To show the second part of Theorem \ref{main2}, it suffices to show that for any $\theta \in \Theta_1^{\mathrm{2-side}}(p,n,\rho,s)$, the Type I and II errors of the penalized Berk--Jones statistic tend to zero as $p\to \infty$. The test statistic in this setting is defined in terms of the KL divergence function \eqref{KL-function} for two Bernoulli distributions,
\begin{align}
\label{pbj-stat}
\PBJ_p \coloneqq \left[\max_{k\leq | \T|} \sup_{q \in (0,1)} p K(\bar{S}_{p,k}(q),q)\right] - 2\log |\T|,
\end{align}
where $\mathcal{T}$ is the following grid,
\begin{align*}
\T &\coloneqq \left\{\lfloor (1+\d)^0\rfloor,\lfloor (1+\d)^1\rfloor,\dots,\lfloor (1+\d)^{\log_{1+\d}\frac{n}{2}}\rfloor,\lfloor n-(1+\d)^{\log_{1+\d}\frac{n}{2}}\rfloor,\dots,\lfloor n-(1+\d)^0\rfloor \right\},
\end{align*}
where $\d = \frac{1}{\log\log n}\to 0$, and $\bar{S}_{p,k}(q)$ is defined
\begin{align}
\label{sydef}
\bar{S}_{p,k} &\coloneqq \frac{1}{p} \sum_{j=1}^p \textbf{1}_{\{|Y_{jk}| > \bar{\Phi}^{-1}(q/2)\}}, \hspace{1em} k=1,\dots,|\T|
\end{align}
and $Y_{jk}$ is the contrast corresponding to the $k^{th}$ element in the grid $t_k \in \T$,
\begin{align*}
Y_{jk} \coloneqq \sqrt{\frac{t_k (n-t_k)}{n}}(\bar{X}_{j,1:t_k} - \bar{X}_{j,t_k+1:n}), \hspace{2em} 1\leq j \leq p, \hspace{1em} 1\leq k \leq |\T| .
\end{align*}
Note that in the current asymptotic setting (\ref{3log}), the cardinality of the grid is $|\T| = e^{p^{a(1+o(1))}}$. The test is performed by checking if the penalized Berk--Jones statistic \eqref{pbj-stat} exceeds the level $2(2+\gamma)\log p$, where $\gamma > 0$ is an arbitrary positive constant,
\begin{align}
\label{eq:pbj-test}
\psi_{\PBJ}(X) \coloneqq \textbf{1}_{\{\PBJ_p > 2(2+\gamma)\log p\}}.
\end{align}

\subsubsection*{Type I error}

For any $\theta \in \Theta_0(p,n)$, we have $Y_{jk} \sim N(0,1)$. The maximum over $q > 0$ is equivalent to taking a maximum over a finite set \citep{berk1979goodness},
\begin{align*}
\max_{q \in (0,1)} K(\bar{S}_{p,k}(q),q) = \max_{j\leq p} K(j/p,p_{k,(j)}),
\end{align*}
where $p_{k,j} \coloneqq 2\bar{\Phi}(|Y_{jk}|)$ is the two sided $p$-value corresponding to the $k^{th}$ element in $\T$, and $p_{k,(1)}<\dots <p_{k,(p)}$ are the ordered $p$-values. Under the null, they have the same distribution as the order statistics of $p$ iid Uniform$(0,1)$ variables, since for every $k \leq |\T|$ the $Y_{jk}$ are independent across $j=1,\dots,p$ and standard normal under the null. It now follows by the union bound that
\begin{align*}
\P_\theta(\PBJ_p > 2(2+\gamma) \log p) &= \P_\theta \left(\max_{k\leq |\T|,j\leq p} pK(j/p,p_{k,(j)}) > 2(2+\gamma)\log p + 2\log |\T| \right) \\
&\leq |\T| \sum_{j=1}^p \P_\theta\left(pK(j/p,p_{1,(j)}) > 2(2+\gamma)\log p + 2\log |\T| \right) ,
\end{align*}
since for every $j$, the quantities $K(j/p,p_{k,(j)})$ are identically distributed across $k=1,\dots,|\T|$ when the null is true. By Lemma \ref{chernoff}, the above is bounded by
\begin{align*}
&\leq |\T| \left[(1+9/e) e^{-2(2+\gamma)\log p - 2\log |\T|} +\sum_{j=2}^p e\sqrt{2} j e^{-(1-1/j)(2(2+\gamma)\log p + 2\log |\T|)}\right] \\
&\leq |\T| \left[(1+9/e) e^{-\gamma\log p - \log |\T|} +e\sqrt{2} p\sum_{j=2}^p  e^{-\frac{1}{2}(2(2+\gamma)\log p + 2\log |\T|)}\right] \\
&\leq |\T| \left[(1+9/e) e^{-\gamma\log p - \log |\T|} +e\sqrt{2} p^2 e^{-(2+\gamma)\log p -\log |\T|}\right] \to 0.
\end{align*}

\subsubsection*{Type II error}
For any $\theta \in \Theta_1^{\mathrm{2-side}}(p,n,\rho,s)$, the Type II error is,
\begin{align*}
\P_\theta(\PBJ_p \leq 2(2+\gamma)\log p) &= \P_\theta \left(\max_{k\leq |\T|,q>0} K(\bar{S}_{p,k}(q),q) \leq \frac{2(2+\gamma)\log p + 2p^{a(1+o(1))}}{p} \right),
\end{align*}
since $|\T| = e^{p^{a(1+o(1))}}$. To show that the right hand side tends to zero as $p\to \infty$, it suffices to pick $k \leq |\T|$ and $q > 0$ such that $K(\bar{S}_{p,k}(q),q) = \omega(p^{a-1+\zeta})$ in probability for some constant $\zeta > 0$, i.e.
\begin{align}
\label{type2crit}
\frac{K(\bar{S}_{p,k}(q),q)}{p^{a-1}} \stackrel{\P_\theta}{\longrightarrow} \infty \hspace{1em}\text{polynomially fast in $p$.}
\end{align}
Denote by $t^*$ the true changepoint location in $\theta \in \Theta_1^{\mathrm{2-side}}(p,n,\rho,s)$. The Berk--Jones test is equivalent to a threshold test (see Section 1.1 of \cite{arias2019detection}) whose rejection region is of the form,
\begin{align}
\label{threshtest}
\bigcup_{k \leq \T, t \in \mathcal{S}} \{\bar{S}_{p,k}(t) \geq c_t\},
\end{align}
for some subset $\mathcal{S} \subseteq\R$ and $(c_t)$ a set of critical values.
It then follows from the definition (\ref{sydef}) that if the means of the Gaussian statistics $Y_{jk}$ are increased, the power of the threshold test increases, i.e. the Type II error decreases. Thus it is sufficient to show the Type II error tends to zero when the lower bounds on $s$ and $\rho$ in the definition (\ref{2sidecp}) of the parameter space $\Theta_1^{2-\mathrm{side}}(p,n,\rho,s)$ are achieved, that is,
\begin{align*}
\sqrt{\frac{t^*(n-t^*)}{n}}|\mu_{j1}-\mu_{j2}| &= \rho \\
\sum_{j=1}^p \textbf{1}_{\{\text{row $j$ has a changepoint}\}} &= s.
\end{align*}
Without loss of generality, suppose that $t^* \leq n/2$; by symmetry of $\T$, an analogous argument can be made for $t^* > n/2$. Let $\rho \coloneqq \rho^*_{\mathrm{2-side}}(a,\beta_1)$. By Lemma \ref{closegrid}, there exists some $\tilde{t} \coloneqq \lfloor (1+\d)^{\tilde{k}} \rfloor \in \T$ for which
\begin{align}
\label{murho}
|Y_{j\tilde{k}}| = |\mu + Z| \text{ where } \mu = (1+o(1)) \rho,
\end{align}
when $j$ corresponds to a non-null row in $\theta$. When $j$ corresponds to a null row in $\theta$, 
\begin{align*}
Y_{jk} \sim N(0,1) \hspace{2em} \text{for all }k =1,\dots,|\T|.
\end{align*}
Throughout the following calculations, $\mu$ refers to the mean satisfying (\ref{murho}).

\subsubsection*{Case 1: $a \leq 1-2\beta_1$}

In the case $a < 1-2\beta_1$, we have $\rho = p^{\frac{a-(1-2\beta_1)}{4}} \to 0$. Put $q \coloneqq 2\bar{\Phi}(\mu) \to 1$, where $\mu = (1+o(1))\rho =(1+o(1)) p^{\frac{a-(1-2\beta_1)}{4}} \to 0$. Then by (\ref{murho}), we have
\begin{align*}
\E_\theta \bar{S}_{p,\tilde{k}}(q) &= \frac{1}{p} \sum_{j=1}^p \P_\theta \left(|Y_{j\tilde{k}} | > \bar{\Phi}^{-1}(q/2)\right) \\
&= (1-\e) \P(|Z| > \mu) + \e \P(|\mu+Z|>\mu) \\
&= (1-\e)2\bar{\Phi}(\mu) + \e (\bar{\Phi}(0) + \bar{\Phi}(2\mu)) \to 1, \\
\var_\theta(\bar{S}_{p,\tilde{k}}(q)) &\leq \frac{1}{p^2}\sum_{j=1}^p \P_\theta\left(|Y_{j\tilde{k}}|>\bar{\Phi}^{-1}(q/2) \right) \leq p^{-1},
\end{align*}
where $\e \coloneqq p^{-\beta}$ is the fraction of non-null rows in $\theta$. 
It now follows from Chebyshev's inequality that
\begin{align}
\label{chebyshev-k-tilde}
\P_\theta\left(|\bar{S}_{p,\tilde{k}}(q) -\E_\theta \bar{S}_{p,\tilde{k}}(q)|\leq p^{-1/2}\log p\right) \to 1,
\end{align}
as $p\to\infty$. The above calculation for $\E_\theta\bar{S}_{p,\tilde{k}}(q)$ and the Mean Value Theorem imply that
\begin{align}
\nonumber
\E_\theta\bar{S}_{p,\tilde{k}}(q)-q &= \e (\bar{\Phi}(0) + \bar{\Phi}(2\mu) - 2\bar{\Phi}(\mu)) \tag{$q=2\bar{\Phi}(\mu)$} \\
\nonumber
&= \e (\phi'(0) \mu^2 + o(\mu^2)) \tag{$\mu \to 0$}\\
\nonumber
&= \phi'(0) p^{-\beta+\frac{a-(1-2\beta_1)}{2}}(1+o(1)) \\
\nonumber
&= \phi'(0) p^{(\beta_1-\beta) + \frac{a}{2}-\frac{1}{2}} (1+o(1)) \\
\label{omega-variance}
&= \omega(p^{-1/2}\log p),
\end{align}
since $\beta_1 > \beta$. Hence, by Part 1 of Lemma \ref{kbehave}, we have 
\begin{align*}
K(\bar{S}_{p,\tilde{k}}(q),q) &\geq 2(\bar{S}_{p,\tilde{k}}(q)-q)^2 \\
&= 2(\bar{S}_{p,\tilde{k}}(q)-\E_\theta\bar{S}_{p,\tilde{k}}(q) + \E_\theta\bar{S}_{p,\tilde{k}}(q)-q)^2 \\
&= 2(\E_\theta\bar{S}_{p,\tilde{k}}(q)-q)^2(1+o_{\P_\theta}(1)) \tag{\eqref{chebyshev-k-tilde} and \eqref{omega-variance}}\\
&= 2\phi'(0)^2 p^{2(\beta_1-\beta) + a-1}(1+o_{\P_\theta}(1)) .
\end{align*}
Then (\ref{type2crit}) follows since $\beta_1 > \beta$.

In the case $a = 1-2\beta_1$, we have $\rho^2 = 1$. Put $q \coloneqq 2\bar{\Phi}(\mu)$. Then by the same calculation in the previous case,
\begin{align*}
\E_\theta \bar{S}_{p,\tilde{k}}(q) &= (1-\e) 2 \bar{\Phi}(\mu) + \e (\bar{\Phi}(0) + \bar{\Phi}(2\mu)) \\
\var_\theta(\bar{S}_{p,\tilde{k}}(q)) &\leq p^{-1}.
\end{align*}
We then have
\begin{align*}
\E_\theta \bar{S}_{p,\tilde{k}}(q) - q &= \e(\bar{\Phi}(0) + \bar{\Phi}(2\mu) - 2\bar{\Phi}(\mu)) \\
&\sim p^{-\beta} (\bar{\Phi}(0) + \bar{\Phi}(2) - 2\bar{\Phi}(1))  \\
&= \omega(p^{-1/2}\log p),
\end{align*}
since $\beta < \beta_1 = \frac{1-a}{2} < \frac{1}{2}$ and $\bar{\Phi}(0) + \bar{\Phi}(2) - 2\bar{\Phi}(1)>0$. By Chebyshev's inequality,
\begin{align*}
\P_\theta \left( |\bar{S}_{p,\tilde{k}}(q) -\E_\theta \bar{S}_{p,\tilde{k}}(q)|\leq p^{-1/2}\log p\right) \to 1,
\end{align*}
as $p \to \infty$. By Part 1 of Lemma \ref{kbehave}, and the reasoning from the previous paragraph, we have
\begin{align*}
K(\bar{S}_{p,\tilde{k}}(q),q) &\geq 2(\E_\theta \bar{S}_{p,\tilde{k}}(q) - q)^2 (1+o_{\P_\theta}(1)) \\
&\gtrsim p^{-2\beta}(1+o_{\P_\theta}(1)).
\end{align*}
Now since $-2\beta-(a-1) > 1-2\beta_1 - a = 0$, the condition (\ref{type2crit}) is satisfied.

\subsubsection*{Case 2: $1-2\beta_1 < a \leq 1-4\beta_1/3$}

In this case, $\rho = \sqrt{(a-(1-2\beta_1))\log p} \to \infty$. Put $q \coloneqq 2\bar{\Phi}(2\mu) \to 0$. We have
\begin{align}
\nonumber
\E_\theta \bar{S}_{p,\tilde{k}}(q) &= \frac{1}{p} \sum_{j=1}^p \P_\theta \left(|Y_{j\tilde{k}} | > \bar{\Phi}^{-1}(q/2)\right) \\
\nonumber
&= (1-\e) \P(|Z| > 2\mu) + \e \P(|\mu+Z|>2\mu) \\
\label{altexpbd}
&= (1-\e)2\bar{\Phi}(2\mu) + \e (\bar{\Phi}(\mu) + \bar{\Phi}(3\mu)) \\
\label{altvarbd}
\var_\theta(\bar{S}_{p,\tilde{k}}(q)) &\leq \frac{1}{p^2} \sum_{j=1}^{p} \P_{\theta}\left(|Y_{j\tilde{k}} | > \bar{\Phi}^{-1}(q/2)\right) = p^{-1}\E_{\theta}\bar{S}_{p,\tilde{k}}(q),
\end{align}
since $Y_{j\tilde{k}}$ are identically distributed across $j=1,\dots,p$. It follows that
\begin{align*}
\frac{\var_\theta(\bar{S}_{p,\tilde{k}}(q))}{(\E_\theta \bar{S}_{p,\tilde{k}}(q))^2} &\leq \frac{1}{p \E_\theta \bar{S}_{p,\tilde{k}}(q)} \\
&= \frac{1}{p\big((1-\e)2\bar{\Phi}(2\mu) + \e (\bar{\Phi}(\mu)+\bar{\Phi}(3\mu))\big)} \\
&\leq \frac{1}{p^{1-\beta-\frac{1}{2}(a-(1-2\beta_1))(1+o(1))}} \to 0,
\end{align*}
where the last inequality follows from Mill's ratio, $\bar{\Phi}(\mu) \geq p^{-\frac{1}{2}(a-(1-2\beta_1))(1+o(1))}$. The convergence to zero follows since the inequality $1-\beta-\frac{1}{2}(a-(1-2\beta_1)) > 0$ is implied by the inequalities $a \leq 1-4\beta_1/3$, $\beta < \beta_1$, and $\beta_1 \leq \frac{3}{4}$ (if $\beta_1 > \frac{3}{4}$, then $0\leq a \leq 1-4\beta_1/3$ could not be satisfied). 

There are two possibilities that determine the behavior of the term $\E_\theta \bar{S}_{p,\tilde{k}}(q)$,
\begin{align}
\label{1case3expS}
-2(a-(1-2\beta_1)) \geq -\beta - \frac{1}{2}(a-(1-2\beta_1)), \\
\label{2case3expS}
-2(a-(1-2\beta_1)) < -\beta - \frac{1}{2}(a-(1-2\beta_1)).
\end{align}

First suppose (\ref{1case3expS}) holds, so that $\mu \sim \rho= \sqrt{(a-(1-2\beta_1))\log p}$ and $\e \bar{\Phi}(\mu) \leq 2\bar{\Phi}(2\mu)$ together with \eqref{altexpbd} implies
\begin{align*}
    2\bar{\Phi}(2\mu)(1+o(1))\leq\E_{\theta}\bar{S}_{p,\tilde{k}}(q) \leq 4\bar{\Phi}(2\mu)(1+o(1)).
\end{align*}
It then follows from Chebyshev's inequality (since $\frac{\var_\theta(\bar{S}_{p,\tilde{k}}(q))}{(\E_\theta \bar{S}_{p,\tilde{k}}(q))^2} \to 0$) that 
\begin{align*}
    \frac{ \bar{S}_{p,\tilde{k}}(q)}{q} = \frac{ \bar{S}_{p,\tilde{k}}(q)}{2\bar{\Phi}(2\mu)} \leq \frac{4}{2} + o_{\P_\theta}(1),
\end{align*}
which implies $\frac{ \bar{S}_{p,\tilde{k}}(q)}{q} \leq 4$ on a set with probability tending to 1. Then Part 2 of Lemma \ref{kbehave} implies
\begin{align}
\label{asymp:high-prob-lb}
    \P_\theta \left( K(\bar{S}_{p,\tilde{k}}(q),q) \geq \frac{(\bar{S}_{p,\tilde{k}}(q)-q)^2}{9q} \right) \to 1.
\end{align}
Note by (\ref{altexpbd}), (\ref{altvarbd}), and (\ref{1case3expS}) that 
\begin{align*}
\sqrt{\var_\theta(\bar{S}_{p,\tilde{k}}(q))} &\leq \sqrt{\frac{(1-\e)2\bar{\Phi}(2\mu)+\e(\bar{\Phi}(\mu)+\bar{\Phi}(3\mu))}{p}} \\
&\sim \sqrt{ p^{-1-2(a-(1-2\beta_1))(1+o(1))}} = o(\E_\theta \bar{S}_{p,\tilde{k}}(q) - q),
\end{align*}
where the last equality follows because by Mill's ratio, 
\begin{align*}
    \E_\theta \bar{S}_{p,\tilde{k}}(q) - q = \e \bar{\Phi}(\mu)-\e \bar{\Phi}(2\mu) + \e \bar{\Phi}(3\mu) \sim \e \bar{\Phi}(\mu) = p^{-\beta-\frac{1}{2}(a-(1-2\beta_1))(1+o(1))}
\end{align*}
and since $\beta_1 > \beta$ implies
\begin{align*}
\frac{-1-2(a-(1-2\beta_1))}{2} < -\beta - \frac{1}{2}(a-(1-2\beta_1)).
\end{align*}
It now follows from another application of Chebyshev's inequality that
\begin{align*}
\frac{(\bar{S}_{p,\tilde{k}}(q) -q)^2}{9q} &= \frac{(\bar{S}_{p,\tilde{k}}(q) - \E_\theta \bar{S}_{p,\tilde{k}}(q)+\E_\theta \bar{S}_{p,\tilde{k}}(q)-q)^2}{9q} \\
&= \frac{(\E_\theta \bar{S}_{p,\tilde{k}}(q)-q)^2}{9q} (1+o_{\P_\theta}(1)).
\end{align*}
By Mill's ratio, $q = 2\bar{\Phi}(2\mu) \leq 2 p^{-2(a-(1-2\beta_1))(1+o(1))}$, so the above and \eqref{asymp:high-prob-lb} imply that on a set with probability tending to 1,
\begin{align*}
K(\bar{S}_{p,\tilde{k}}(q),q) &\geq  \frac{1}{18} \cdot p^{2(a-(1-2\beta_1))(1+o(1)) -2\beta-(a-(1-2\beta_1))(1+o(1))} = \omega(p^{a-1}),
\end{align*}
since $\beta < \beta_1$, and (\ref{type2crit}) now follows.

Now suppose that (\ref{2case3expS}) holds, so that $\E_\theta \bar{S}_{p,\tilde{k}}(q) = \e \bar{\Phi}(\mu) (1+o(1))$. Since $\frac{\var_\theta(\bar{S}_{p,\tilde{k}}(q))}{(\E_\theta \bar{S}_{p,\tilde{k}}(q))^2} \to 0$, it follows that,
\begin{align*}
\frac{\bar{S}_{p,\tilde{k}}(q)}{q} &= \frac{\E_\theta \bar{S}_{p,\tilde{k}}(q)(1+o_{\P_\theta}(1))}{2\bar{\Phi}(2\mu)} \tag{Chebyshev}\\
&= \frac{\e \bar{\Phi}(\mu)(1+o_{\P_\theta}(1))}{2\bar{\Phi}(2\mu)} \tag{by (\ref{2case3expS})}\\
&\geq p^{-\beta-\frac{1}{2}(a-(1-2\beta_1))(1+o(1))+2(a-(1-2\beta_1))(1+o(1))}(1+o_{\P_\theta}(1)) \tag{Mill's ratio}\\
&\to \infty,\tag{by \eqref{2case3expS}} \\
\bar{S}_{p,\tilde{k}}(q) &= (\E_{\theta}\bar{S}_{p,\tilde{k}}(q))(1+o_{\P_\theta}(1)) \to 0. \tag{Chebyshev}
\end{align*}
Thus, on a set with probability tending to 1, we have both $e^2\leq \frac{\bar{S}_{p,\tilde{k}}(q)}{q}$ and $\frac{1}{2}\E_\theta \bar{S}_{p,\tilde{k}}(q)\leq \bar{S}_{p,\tilde{k}}(q) \leq \frac{1}{2}$. On this set, by Part 3 of Lemma \ref{kbehave},
\begin{align*}
K(\bar{S}_{p,\tilde{k}}(q),q) &\geq \frac{1}{2} \bar{S}_{p,\tilde{k}}(q) \log \frac{\bar{S}_{p,\tilde{k}}(q)}{q} \\
&\geq \frac{1}{2} \E_\theta \bar{S}_{p,\tilde{k}}(q) \tag{$\log e^2=2$}\\
&= \frac{1}{2} \e \bar{\Phi}(\mu)(1+o(1)) \\
&\geq \frac{1}{2} p^{-\beta-\frac{1}{2}(a-(1-2\beta_1))(1+o(1))}  = \omega(p^{a-1}),
\end{align*}
since $-\beta - \frac{1}{2}(a-(1-2\beta_1)) > a-1$ is implied by $a \leq 1-4\beta_1/3 < 1$ and $\beta < \beta_1$. Thus, (\ref{type2crit}) also holds in this case.

\subsubsection*{Case 3: $1-4\beta_1/3 < a \leq 1-\beta_1$}

In this case, put $x \coloneqq \sqrt{1-a}$ and $y \coloneqq \sqrt{1-a-\beta_1}$. Then,
\begin{align*}
\rho &= (\sqrt{1-a}-\sqrt{1-a-\beta_1})\sqrt{2\log p} \to \infty, \\
q &\coloneqq 2\bar{\Phi}(x\sqrt{2\log p}) \to 0.
\end{align*}
Then since $\mu \sim \rho=(x-y)\sqrt{2\log p}$,
\begin{align*}
\E_\theta \bar{S}_{p,\tilde{k}}(q) &= \frac{1}{p} \sum_{j=1}^p \P_\theta\left( |Y_{j\tilde{k}}| > \bar{\Phi}^{-1}(q/2)\right) \\
&= (1-\e) 2\bar{\Phi}(x\sqrt{2\log p}) + \e \P(|\mu + Z| > x\sqrt{2\log p}) \\
&= (1-\e) 2 \bar{\Phi}(x\sqrt{2\log p}) + \e \left(\bar{\Phi}((y+o(1))\sqrt{2\log p}) + \Phi((y-2x)(1+o(1))\sqrt{2\log p})\right) \\
&\sim 2p^{-x^2(1+o(1))} + p^{-\beta -y^2+o(1)}(1+o(1)) \tag{$y<2x-y$}\\
&\asymp p^{-\beta-y^2+o(1)},
\end{align*}
since $\beta<\beta_1$. Next notice that
\begin{align*}
\frac{\var_\theta(\bar{S}_{p,\tilde{k}}(q))}{(\E_\theta \bar{S}_{p,\tilde{k}}(q))^2} \leq \frac{1}{p \E_\theta \bar{S}_{p,\tilde{k}}(q)} \sim \frac{1}{2p^{1-x^2(1+o(1))}+p^{1-\beta-y^2+o(1)}} \to 0,
\end{align*}
since $y^2 = 1-a-\beta_1$ and $\beta < \beta_1$. It now follows from Chebyshev's inequality and Mill's ratio that
\begin{align*}
\frac{\bar{S}_{p,\tilde{k}}(q)}{q} = \frac{\E_\theta\bar{S}_{p,\tilde{k}}(q) }{q} (1+o_{\P_\theta}(1))\asymp p^{-\beta-y^2+o(1) + (1-a)(1+o(1))}(1+o_{\P_\theta}(1))\to\infty,
\end{align*}
since $y^2=1-a-\beta_1$ and $\beta < \beta_1$. Then on a set with probability tending to 1, we have $e^2 \leq \frac{\bar{S}_{p,\tilde{k}}(q)}{q}$ and $\frac{1}{2}\E_\theta \bar{S}_{p,\tilde{k}}(q) \leq \bar{S}_{p,\tilde{k}}(q) \leq \frac{1}{2}$. On this set, by Part 3 of Lemma \ref{kbehave},
\begin{align*}
K(\bar{S}_{p,\tilde{k}}(q),q) &\geq \frac{1}{2} \bar{S}_{p,\tilde{k}}(q) \log \frac{\bar{S}_{p,\tilde{k}}(q)}{q} \\
&\geq \frac{1}{2} \E_\theta \bar{S}_{p,\tilde{k}}(q) \tag{$\log e^2=2$} \\
&\geq \frac{1}{2} p^{-\beta-y^2+o(1)}(1+o(1)) \\
&=\frac{1}{2} p^{-\beta-(1-a-\beta_1)+o(1)}(1+o(1))  = \omega(p^{a-1}),
\end{align*}
since $\beta < \beta_1$. (\ref{type2crit}) now follows.

\subsubsection*{Case 4: $a > 1-\beta_1$}

In this case, $\rho^2 = p^{a-(1-\beta_1)} \to \infty$. Put $q \coloneqq 2\bar{\Phi}(\mu/2) \leq e^{-\frac{1}{8}(1+o(1))p^{a-(1-\beta_1)}} \to 0$. Then,
\begin{align*}
\E_\theta \bar{S}_{p,\tilde{k}}(q) &= \frac{1}{p} \sum_{j=1}^p \P_\theta(|Y_{j\tilde{k}}| > \mu/2) \\
&= (1-\e) 2\bar{\Phi}(\mu/2) + \e \P(|\mu+Z| > \mu/2) \\
&= (1-\e) 2\bar{\Phi}(\mu/2) + \e (\bar{\Phi}(-\mu/2)+\bar{\Phi}(3\mu/2)) \\
&\sim \e \bar{\Phi}(-\mu/2) = \e (1+o(1)),
\end{align*}
since $\bar{\Phi}(-\mu/2)\to 1$, $\bar{\Phi}(3\mu/2) \to 0$ and $\bar{\Phi}(\mu/2) \to 0$. The variance satisfies
\begin{align*}
\frac{\var_\theta (\bar{S}_{p,\tilde{k}}(q))}{(\E_\theta \bar{S}_{p,\tilde{k}}(q))^2} &\leq \frac{1}{ p \E_\theta \bar{S}_{p,\tilde{k}}(q)} = p^{-1+\beta}(1+o(1)) \to 0.
\end{align*}
Then we have by Chebyshev's inequality and Mill's ratio that,
\begin{align*}
\frac{\bar{S}_{p,\tilde{k}}(q)}{q} &= \frac{\E_\theta \bar{S}_{p,\tilde{k}}(q)}{q}(1+o_{\P_\theta}(1)) \\
&\geq \frac{\e}{e^{-\frac{1}{8}(1+o(1))p^{a-(1-\beta_1)}}}(1+o_{\P_\theta}(1)) \\
&\to \infty,
\end{align*}
since $a > 1-\beta_1$. Then on a set with probability tending to 1, we have $e^2 \leq \frac{\bar{S}_{p,\tilde{k}}(q)}{q}$ and $\frac{1}{2} \E_\theta \bar{S}_{p,\tilde{k}}(q) \leq \bar{S}_{p,\tilde{k}}(q) \leq \frac{1}{2}$. On this set, by Part 3 of Lemma \ref{kbehave},
\begin{align*}
K(\bar{S}_{p,\tilde{k}}(q),q) &\geq \frac{1}{2} \bar{S}_{p,\tilde{k}}(q) \log \frac{\bar{S}_{p,\tilde{k}}(q)}{q} \\
&\geq \frac{1}{4} (\E_\theta \bar{S}_{p,\tilde{k}}(q))\log \frac{\frac{1}{2}\E_\theta \bar{S}_{p,\tilde{k}}(q)}{e^{-\frac{1}{8}(1+o(1))p^{a-(1-\beta_1)}}} \\
&\geq \frac{1}{4} (\E_\theta \bar{S}_{p,\tilde{k}}(q))\log \left(e^{\frac{1}{8}(1+o(1))p^{a-(1-\beta_1)}}\right) \\
&= \frac{1}{32}(1+o(1)) p^{-\beta+a-(1-\beta_1)} = \omega(p^{a-1}),
\end{align*}
since $\beta<\beta_1$. Now (\ref{type2crit}) follows.

\subsection{Upper Bound for Theorem \ref{main1}}
\label{UBmain1}

As discussed in the beginning of Appendix \ref{proofs}, it suffices to show the upper bound in the case $a \leq 1-2\beta_1$, as the upper bound for the other cases is implied by the calculation in the previous section. It suffices to show that for any $\theta \in \Theta_1^{\mathrm{1-side}}(p,n,\rho,s)$, the Type I and II errors of the penalized Berk--Jones statistic tend to zero as $p\to\infty$. The test statistic is the same as for the two-sided problem, except $\bar{S}_{p,k}(q)$ is defined to count the fraction of the contrasts that exceed a threshold,
\begin{align*}
\bar{S}_{p,k} \coloneqq \frac{1}{p} \sum_{j=1}^p \textbf{1}_{\{Y_{jk} > \bar{\Phi}^{-1}(q)\}}.
\end{align*}
Since the one-sided $p$-values $p_{k,j} \coloneqq \bar{\Phi}(Y_{jk})$ are distributed as iid Uniform$(0,1)$ variables under the null, the proof of the Type I error is exactly the same as in the two sided setting. We proceed to show that the Type II error also goes to zero when $a \leq 1-2\beta_1$.

\subsubsection*{Type II error}
Suppose $\theta \in \Theta_1^{\mathrm{1-side}}(p,n,\rho,s)$, with $\beta < \beta_1$ as in the statement of the Theorem. It suffices to pick $k\leq |\T|$ and $q > 0$ for which (\ref{type2crit}) is satisfied. Denote by $t^*$ the true changepoint location in $\theta$. By the discussion in Appendix \ref{UBmain2} about threshold tests (see (\ref{threshtest})), assume that the lower bounds $\rho$ and $s$ on the signal size and sparsity in the definition of $\Theta_1^{\mathrm{1-side}}(p,n,\rho,s)$ are achieved. Without loss of generality, assume that $t^* \leq n/2$. By Lemma \ref{closegrid}, there exists some $\tilde{t} \coloneqq \lfloor (1+\d)^{\tilde{k}} \rfloor \in \T$ for which
\begin{align}
\label{murho2}
Y_{j\tilde{k}} = \mu + Z \text{ where }\mu=(1+o(1))\rho,
\end{align}
when $j$ corresponds to a non-null row in $\theta$, and $\rho \coloneqq \rho^*_{\mathrm{1-side}}(a,\beta_1)$, and $Z \sim N(0,1)$. When $j$ corresponds to a null row in $\theta$,
\begin{align*}
Y_{jk} \sim N(0,1) \hspace{2em} \text{for all } k = 1,\dots,|\T|.
\end{align*}
Throughout the following calculation, $\mu$ refers to the mean satisfying (\ref{murho2}). 

\subsubsection*{Case 1: $a \leq 1-2\beta_1$}
First suppose $a < 1-2\beta_1$. In this case, $\mu \sim \rho = p^{\frac{a-(1-2\beta_1)}{2}} \to 0$. Put $q = \bar{\Phi}(2\mu) \to \frac{1}{2}$. Then
\begin{align*}
\E_\theta \bar{S}_{p,\tilde{k}}(q) &= \frac{1}{p} \sum_{j=1}^p \P_\theta(Y_{j\tilde{k}} > \bar{\Phi}^{-1}(q)) \\
&= (1-\e) \P(Z > 2\mu) + \e \P(\mu+Z>2\mu) \\
&= (1-\e) \P(Z > 2\mu) + \e \P(Z>\mu) \to \frac{1}{2} , \\
\var_\theta(\bar{S}_{p,\tilde{k}}) &\leq \frac{1}{p^2} \sum_{j=1}^p \P_\theta(Y_{j\tilde{k}}>\bar{\Phi}^{-1}(q)) \leq p^{-1}.
\end{align*}
Then since $\frac{\var_\theta(\bar{S}_{p,\tilde{k}}(q))}{(\E_\theta \bar{S}_{p,\tilde{k}}(q))^2} \to 0$, it follows by Chebyshev's inequality that 
\begin{align*}
\P_\theta\left(|\bar{S}_{p,\tilde{k}}(q)-\E_\theta \bar{S}_{p,\tilde{k}}(q)| \leq p^{-1/2}\log p\right) \to 1,
\end{align*}
By the Mean Value Theorem,
\begin{align*}
\E_\theta \bar{S}_{p,\tilde{k}}(q) - q &= (1-\e) \bar{\Phi}(2\mu) + \e \bar{\Phi}(\mu) - \bar{\Phi}(2\mu) \\
&= \e (\bar{\Phi}(\mu)-\bar{\Phi}(2\mu)) \\
&= \e (\phi(0)\mu + o(\mu)) \\
&= \phi(0) p^{-\beta+\frac{a-(1-2\beta_1)}{2}}(1+o(1)) \\
&= \omega(p^{-1/2}\log p)\tag{$\beta<\beta_1$},
\end{align*}
so that $\E_\theta \bar{S}_{p,\tilde{k}}(q) - q$ is of larger order than $\bar{S}_{p,\tilde{k}}(q)-\E_\theta \bar{S}_{p,\tilde{k}}(q)$ on a set with probability tending to 1. By Part 1 of Lemma \ref{kbehave}, 
\begin{align*}
K(\bar{S}_{p,\tilde{k}}(q),q) &\geq 2(\bar{S}_{p,\tilde{k}}(q)-q)^2 \\
&= 2(\bar{S}_{p,\tilde{k}}(q)-\E_\theta \bar{S}_{p,\tilde{k}}(q)+\E_\theta \bar{S}_{p,\tilde{k}}(q)-q)^2 \\
&= 2(\E_\theta \bar{S}_{p,\tilde{k}}(q)-q)^2 (1+o_{\P_\theta}(1)) \\
&= 2\phi(0)^2 p^{-2\beta+a-(1-2\beta_1)}(1+o_{\P_\theta}(1)) \\
&= \omega (p^{a-1}),
\end{align*}
since $\beta < \beta_1$. Hence, (\ref{type2crit}) holds.

In the case $a = 1-2\beta_1$, we have $\mu \sim \rho = 1$. Put $q \coloneqq \bar{\Phi}(\mu)$. Then by the same calculation as in the previous case,
\begin{align*}
\E_\theta \bar{S}_{p,\tilde{k}}(q) &= (1-\e)\P(Z>\mu) + \e \P(\mu+Z > \mu) \\
&= (1-\e) \bar{\Phi}(\mu) + \e \bar{\Phi}(0) \\
\var_\theta (\bar{S}_{p,\tilde{k}}(q)) &\leq p^{-1}.
\end{align*}
Then,
\begin{align*}
\E_\theta \bar{S}_{p,\tilde{k}}(q) - q &= \e (\bar{\Phi}(0)-\bar{\Phi}(\mu)) \gtrsim p^{-\beta} = \omega(p^{-1/2}\log p),
\end{align*}
since $\beta < \beta_1 = \frac{1-a}{2} < \frac{1}{2}$. By Chebyshev's inequality,
\begin{align*}
\P_\theta \left( |\bar{S}_{p,\tilde{k}}(q)-\E_\theta \bar{S}_{p,\tilde{k}}(q)| \leq p^{-1/2}\log p \right) \to 1,
\end{align*}
so that $\E_\theta \bar{S}_{p,\tilde{k}}(q) - q$ is of larger order than $\bar{S}_{p,\tilde{k}}(q)-\E_\theta \bar{S}_{p,\tilde{k}}(q)$ on a set with probability tending to 1. Thus by Part 1 of Lemma \ref{kbehave},
\begin{align*}
K(\bar{S}_{p,\tilde{k}}(q),q) &\geq 2(\E_\theta \bar{S}_{p,\tilde{k}}(q)-q)^2 (1+o_{\P_\theta}(1)) \\
&\gtrsim p^{-2\beta}(1+o_{\P_\theta}(1)).
\end{align*}
Now since $-2\beta-(a-1) > 1-2\beta_1-a=0$, we must have
\begin{align*}
    K(\bar{S}_{p,\tilde{k}}(q),q) \gtrsim p^{-2\beta-(a-1)+(a-1)} = \omega(p^{1-2\beta_1-a+(a-1)}) = \omega(p^{a-1})
\end{align*}
on a set with probability tending to 1, so condition (\ref{type2crit}) is satisfied.

\subsection{Lower Bound for Theorem \ref{main3}}
\label{LBmain3}

Recall the formula $\rho_2^*(a,\beta) = \sqrt{2r_2^*(a,\beta)\log p}$, where
\begin{align*}
r_2^*(a,\beta) &\coloneqq \begin{cases}
\beta-1/2 \hspace{1em} &1/2 < \beta \leq 3/4 \\
(1-\sqrt{1-\beta})^2 &3/4<\beta < 1\\
1+a &\beta = 1,
\end{cases}
\end{align*}
and the asymptotic setting,
\begin{align*}
\log\log n &\sim a\log p \hspace{2em} a > 0 \\
\nonumber
s &\sim p^{1-\beta} \hspace{2em} 1/2 < \beta \leq 1.
\end{align*}
Put $\rho \coloneqq \sqrt{2r\log p}$, where $0<r < r^*_2(a,\beta)$ as in the statement of the theorem. We show below that if $\beta \in (\frac{1}{2},1]$, then $\mathcal{R}_{\mathrm{1-side}}(p,n,\rho,s) \to 1$ as $p \to \infty$. A different prior on $\theta$ is used when $\beta \in (\frac{1}{2},1)$ versus when $\beta = 1$. 
\subsubsection*{Cases 1 and 2: $\beta \in (1/2,3/4]$ and $\beta \in (3/4,1)$}
Since $0 < r < r_2^*(a,\beta)$, which is increasing in $\beta$, there exists some $\beta_1 \in (1/2,\beta)$ for which $r_2^*(a,\beta_1) = r$. Consider the testing problem,
\begin{align*}
&H_0: X_{j:} \stackrel{\iid}{\sim} N(0,I_n) \\
&\bar{H}_1: X_{j:} \stackrel{\iid}{\sim} (1-\bar{\e}) N(0,I_n) + \bar{\e} N(\rho \theta^{(1)},I_n),
\end{align*}
where $\bar{\e} \coloneqq p^{-\bar{\beta}}$ with $\bar{\beta}\in(\beta_1,\beta)$, and $\theta^{(1)}$ is defined as in Appendix \ref{LBmain1}, and the reason for the choice $\bar{\beta}\in(\beta_1,\beta)$ is discussed in the beginning of Appendix \ref{LBmain1}. Since $\theta^{(1)}$ is a unit vector (see Lemma \ref{thetas}), there exist unit vectors $u_2,\dots,u_n \in \R^n$ such that $\{\theta^{(1)},u_2,\dots,u_n\}$ is an orthonormal basis for $\R^n$. Then each row $X_{j:}$ can be rewritten in this basis,
\begin{align}
\label{newcoord}
X_{j:} \equiv (X_{j:}^\top \theta^{(1)},X_{j:}^\top u_2,\dots,X_{j:}^\top u_n).
\end{align}
Since $u_i^\top \theta^{(1)} = 0$, we have $X_{j:}^\top u_i \stackrel{\iid}{\sim} N(0,1)$ under both $H_0$ and $H_1$ for each $i=2,\dots,n$, and
\begin{align*}
X_{j:}^{\top} \theta^{(1)} &\stackrel{\iid}{\sim} \begin{cases}
N(0,1) \hspace{1em} &\text{under } H_0 \\
(1-\bar{\e})N(0,1)+\bar{\e}N(\rho,1) &\text{under }H_1.
\end{cases} 
\end{align*}
Hence only the first coordinate $X_{j:}^\top \theta^{(1)}$ is informative for testing between $H_0$ and $H_1$, and the remaining coordinates in (\ref{newcoord}) can be ignored. Recall that $\bar{\beta} > \beta_1$, and that $\rho = \sqrt{2r\log p} = \sqrt{2r_2^*(a,\beta_1)\log p}$ is the minimal signal strength needed to detect a non-null fraction $p^{-\beta_1}$ in the classical setting of Ingster--Donoho--Jin. If the true sparsity level is $\bar{\e} = p^{-\bar{\beta}}$, then $\mathcal{R}_{\mathrm{1-side}}(p,n,\rho,s) \to 1$ as a consequence of the original Ingster--Donoho--Jin detection boundary.

\subsubsection*{Case 3: $\beta=1$}
In this case, $\rho = \sqrt{2r\log p} < \sqrt{2(1+a)\log p}$. We show that $\mathcal{R}_{\mathrm{1-side}}(p,n,\rho,s=1) \to 1$ as $p\to \infty$. Consider the testing problem,
\begin{align*}
H_0: X_{j:} \stackrel{\iid}{\sim} N(0,I_n) \hspace{1em} \mathrm{vs}\hspace{1em} H_1: k &\sim \mathrm{Unif} \{1,\dots,|\T|\} \\
i&\sim \mathrm{Unif}\{1,\dots,p\} \\
X_{j:} \mid k,i &\stackrel{\mathrm{indep.}}{\sim} \begin{cases}
N(\rho \theta^{(k)},I_n) \hspace{1em} &j = i\\
N(0,I_n) &j \neq i,
\end{cases} 
\end{align*}
for $j=1,\dots,p$, where the grid $\T$ is defined,
\begin{align*}
\mathcal{T} \coloneqq \{\lfloor b^1\rfloor,\lfloor b^2 \rfloor, \dots, \lfloor b^{\log_{\log n}n}\rfloor \},
\end{align*}
with base $b = \log n$, and the $\theta^{(k)}$ are defined as in previous sections (see Lemma \ref{thetas}). Note that the cardinality of this grid is $|\T| \asymp \frac{\log n}{\log\log n} = p^{a(1+o(1))}$, according to the calibration (\ref{2log}). It suffices to show the following two conditions,
\begin{align}
\label{trunc1}
\P_1(A_p^c) &\to 0 \\
\label{trunc2}
\limsup_{p\to\infty} \E_0\left(\frac{f_1}{f_0}(X) \right)^2\textbf{1}_{A_p} &\leq 1,
\end{align}
for a suitable truncation event $A_p$, where $f_0$ and $f_1$ are the densities of $X$ corresponding to $H_0$ and $\bar{H}_1$ respectively. Indeed, if (\ref{trunc1}) and (\ref{trunc2}) hold, then by monotonicity and Markov's inequality,
\begin{align*}
\P_0\left(\left|\frac{f_1}{f_0}(X)-1 \right| > \eta \right) &\leq \P_0 \left(\left|\frac{f_1}{f_0}(X)\textbf{1}_{A_p}-1 \right| > \eta \right) \tag{$\eta\in(0,1)$}\\
&\leq \frac{\E_0\left( \frac{f_1}{f_0}(X) \right)^2 \textbf{1}_{A_p} - 2\E_0 \frac{f_1}{f_0}(X) \textbf{1}_{A_p} +1}{\eta^2} \to 0,
\end{align*}
since $\P_1(A_p^c) \to 0$ is equivalent to $\E_0 \frac{f_1}{f_0}(X)\textbf{1}_{A_p} \to 1$. Thus (\ref{trunc1}) and (\ref{trunc2}) imply $\frac{f_1}{f_0}(X) \stackrel{\P_0}{\longrightarrow} 1$, which implies $H_0$ and $H_1$ are asymptotically hard to distinguish (see Appendix \ref{LBmain1}). 

To obtain (\ref{trunc1}) and (\ref{trunc2}), define the truncation event,
\begin{align*}
A_p \coloneqq \left\{\max_{k\leq |\T|, j\leq p} \langle X_{j:},\theta^{(k)}\rangle \leq \sqrt{2(1+a+\gamma)\log p} \right\},
\end{align*}
for some small constant $\gamma$ satisfying
\begin{align}
\label{gammareq}
0 < \gamma < 4r - (1+a),
\end{align}
the existence of which we may assume without loss of generality, since distinguishing between $H_0$ and $H_1$ is even more difficult if $4r \leq 1+a$. Next, we check that (\ref{trunc1}) holds,
\begin{align*}
\P_1(A_p^c) &= \frac{1}{p|\T|} \sum_{k=1}^{|\T|} \sum_{i=1}^p \P\left(\max_{m\leq |\T|,j\leq p} \langle \rho \theta^{(k)}\textbf{1}_{\{i=j\}} + Z_{j:},\theta^{(m)}\rangle > \sqrt{2(1+a+\gamma)\log p} \right) \\
&= \frac{1}{p|\T|} \sum_{k=1}^{|\T|} \sum_{i=1}^p \sum_{m=1}^{|\T|} \sum_{j=1}^p \P\big(\langle \rho \theta^{(k)}\textbf{1}_{\{i=j\}}+Z_{j:},\theta^{(m)}\rangle > \sqrt{2(1+a+\gamma)\log p}\big),
\end{align*}
where the probability measure $\P$ on the right hand side refers to $Z_{j:} \stackrel{\iid}{\sim}N(0,I_n)$. Applying Part 3 of Lemma \ref{thetas} with $b =\log n \asymp p^{a(1+o(1))}$, we have $\langle \theta^{(k)},\theta^{(m)}\rangle \leq p^{-\frac{a|k-m|}{2}}$, which is $o(1/\rho)$ for any $k\neq m$. Thus the sum can be split into three pieces,
\begin{align*}
&\leq \frac{1}{p|\T|} \sum_{(k,m,i,j):j=i,m=k} \bar{\Phi}\big(\sqrt{2(1+a+\gamma)\log p}-\rho \big) \\
&+ \frac{1}{p|\T|} \sum_{(k,m,i,j):j=i,m\neq k} \bar{\Phi}\big(\sqrt{2(1+o(1))(1+a+\gamma)\log p} \big) \\
&+ \frac{1}{p|\T|} \sum_{(k,m,i,j):j\neq i} \bar{\Phi}\big(\sqrt{2(1+a+\gamma)\log p} \big),
\end{align*}
where we have used that $\langle Z_{j:},\theta^{(m)} \rangle \sim N(0,1)$ for every $m \leq |\T|$ and $j\leq p$. Applying Mill's ratio to the above tail probabilities, and noting that $|\T| \asymp \frac{\log n}{\log\log n} = p^{a(1+o(1))}$ according to the calibration (\ref{2log}), we find that the above is bounded by
\begin{align*}
&\leq \frac{1}{p^{1+a(1+o(1))}} p^{1+a(1+o(1))} p^{-(\sqrt{1+a+\gamma}-\sqrt{r})^2} \\
&+ \frac{1}{p^{1+a(1+o(1))}} p^{1+2a(1+o(1))}p^{-(1+a+\gamma)(1+o(1))} \\
&+ \frac{1}{p^{1+a(1+o(1))}} p^{2(1+a(1+o(1)))}p^{-(1+a+\gamma)} \to 0,
\end{align*}
since $1+a+\gamma > r$ and $\gamma > 0$. To check (\ref{trunc2}), we compute the truncated second moment,
\begin{align*}
\E_0 \left( \frac{f_1}{f_0}(X) \right)^2\textbf{1}_{A_p} &= \E_1 \frac{f_1}{f_0}(X) \textbf{1}_{A_p} \\
&= \frac{1}{p|\T|}  \sum_{k=1}^{|\T|} \sum_{i=1}^p \underbrace{\E_1\big(e^{\rho \langle X_{i:},\theta^{(k)}\rangle - \rho^2/2} \textbf{1}_{A_p}\big)}_{(*)}.
\end{align*}
The term $(*)$ is,
\begin{align*}
(*) &= \frac{1}{p|\T|} \sum_{l=1}^{|\T|} \sum_{j=1}^p \E \left(e^{\rho\langle \rho \theta^{(l)}\textbf{1}_{\{i=j\}}+Z_{i,:},\theta^{(k)}\rangle - \rho^2/2}\textbf{1}_{A_p}\right) \tag{$Z_{i,:} \stackrel{\iid}{\sim}N(0,I_n)$}\\
&\leq \frac{1}{p|\T|} \sum_{l=1}^{|\T|} \left[ (p - 1) + e^{\rho^2 \langle \theta^{(l)},\theta^{(k)}\rangle }\E e^{ \rho U-\rho^2/2}\textbf{1}_{\{\rho \langle \theta^{(l)},\theta^{(k)}\rangle + U \leq \sqrt{2(1+a+\gamma)\log p}\}} \right] \tag{$U\sim N(0,1)$} \\
&= \frac{(p-1)|\T|}{p|\T|} + \frac{1}{p|\T|}\left( e^{\rho^2}\E e^{\rho U - \rho^2/2}\textbf{1}_{\{U \leq \sqrt{2(1+a+\gamma)\log p}-\rho\}} + \sum_{l\leq |\T|: l \neq k}^{|\T|} (1+c_p) \right),
\end{align*}
where we have used Part 3 of Lemma \ref{thetas} to deduce 
\begin{align*}
    \rho^2 \langle \theta^{(l)},\theta^{(k)}\rangle = \begin{cases}
        c_p \hspace{1em} &\text{when } l\neq k \\
        \rho^2 &\text{when } l=k,
    \end{cases}
\end{align*}
for some positive sequence $c_p \to 0$ This is bounded by
\begin{align*}
&\leq 1 + \frac{1}{p|\T|} \left[ e^{\rho^2} \Phi(\sqrt{2(1+a+\gamma)\log p} - 2\rho) + (|\T|-1) (1+o(1)) \right] \\
&\leq 1 + \frac{1}{p^{1+a(1+o(1))}}\cdot p^{2r} \cdot p^{-(\sqrt{1+a+\gamma}-2\sqrt{r})^2}+o(1),
\end{align*}
where we have used that $\sqrt{1+a+\gamma}-2\sqrt{r} < 0$ is implied by the condition (\ref{gammareq}). The right hand side of the expression for $(*)$ now becomes
\begin{align*}
&= 1 + c_p + p^{-(1+a)(1+o(1)) + 2(1+a)\frac{r}{1+a} - (1+a)\big(\sqrt{1+\frac{\gamma}{1+a}}-2\sqrt{\frac{r}{1+a}}\big)^2} \\
&= 1 + c_p + p^{-(1+a)(1+o(1))\left(1-2\frac{r}{1+a} + \big(\sqrt{1+\frac{\gamma}{1+a}}-2\sqrt{\frac{r}{1+a}}\big)^2\right)}.
\end{align*}
As $\gamma \to 0$, some manipulation will show that the above becomes
\begin{align*}
&= 1+c_p+p^{-2(1+a)(1+o(1))\big(1-\sqrt{\frac{r}{1+a}}\big)^2+O(\gamma)},
\end{align*}
It follows from the assumption $r < 1+a$ that $\gamma$ can be chosen small enough to make the exponent negative, yielding $(*) = 1 + o(1)$ uniformly over $(k,i)$ pairs. Plugging this bound back into the truncated second moment gives
\begin{align*}
\E_0 \left( \frac{f_1}{f_0}(X) \right)^2\textbf{1}_{A_p} &= \E_1 \frac{f_1}{f_0}(X) \textbf{1}_{A_p} \\
&\leq \frac{1}{p|\T|}  \sum_{k=1}^{|\T|} \sum_{i=1}^p \left( 1+c_p+p^{-2(1+a)(1+o(1))\big(1-\sqrt{\frac{r}{1+a}}\big)^2+O(\gamma)} \right) \\
&= 1 + o(1),
\end{align*}
as desired.

\subsection{Upper Bound for Theorem \ref{main3}}
\label{UBmain3}

We consider the cases $\beta \in (1/2,1)$ and $\beta=1$ separately; the penalized Berk--Jones test \eqref{eq:pbj-test} gives the upper bound for the former case, while a simple maximum statistic (defined later in expression \eqref{eq:max-test}) gives the upper bound for the latter case. The tests can be combined via $\psi =\psi_{\PBJ} \vee \psi_{\mathrm{max}}$ to give an overall test achieving the upper bound in the theorem.

First let $\beta \in (1/2,1)$. Suppose without loss of generality that $r \in (0,1)$; if $r \geq 1$, then
\begin{align*}
\Theta_1^{\mathrm{1-side}}(p,n,\rho=\sqrt{2r\log p},s) \subseteq \Theta_1^{\mathrm{1-side}}(p,n,\rho=\sqrt{2\log p},s),
\end{align*}
and thus the worst case testing error becomes smaller. Since $r > r^*(a,\beta)$, there exists some $\beta_1 > \beta$ for which $r^*_2(a,\beta_1)=r$. 
Put $\rho \coloneqq \sqrt{2r\log p}$ as in the statement of Theorem \ref{main3}. To show the upper bound, it suffices to show that for any $\theta \in \Theta_1^{\mathrm{1-side}}(p,n,\rho,s)$, the Type I and II errors of the penalized Berk--Jones statistic tend to zero as $p\to \infty$. Here the test statistic is,
\begin{align*}
\PBJ_p \coloneqq \left[\max_{k\leq |\T|} \sup_{q \in (0,1)} pK(\bar{S}_{p,k}(q),q) \right]-2\log |\T|
\end{align*}
where we recall from the beginning of Appendix \ref{UBmain2} that $\mathcal{T}$ and $\bar{S}_{p,k}(q)$ are defined
\begin{align*}
\T &\coloneqq \left\{\lfloor (1+\d)^0\rfloor,\lfloor (1+\d)^1\rfloor,\dots,\lfloor (1+\d)^{\log_{1+\d}\frac{n}{2}}\rfloor,\lfloor n-(1+\d)^{\log_{1+\d}\frac{n}{2}}\rfloor,\dots,\lfloor n-(1+\d)^0\rfloor \right\}, \\
\bar{S}_{p,k} &\coloneqq \frac{1}{p} \sum_{j=1}^p \textbf{1}_{\{Y_{jk} > \bar{\Phi}^{-1}(q)\}},
\end{align*}
where $\d = \frac{1}{\log\log n} \to 0$, and $Y_{jk}$ is the contrast corresponding to the corresponding to the $k^{th}$ element in the grid $t_k \in \T$,
\begin{align}
\label{xtoy}
Y_{jk} \coloneqq \sqrt{\frac{t_k (n-t_k)}{n}}(\bar{X}_{j,1:t_k} - \bar{X}_{j,t_k+1:n}), \hspace{2em} 1\leq j \leq p, \hspace{1em} 1\leq k \leq |\T| .
\end{align}
Note that $|\T| = p^{a(1+o(1))}$ so that $2\log |\T| = 2a(1+o(1))\log p$ in the asymptotic setting (\ref{2log}), which is the setting for this result. The test is performed by checking if the penalized Berk--Jones statistic exceeds the level $2(2+\gamma)\log p$, where $\gamma > 0$ is a small constant,
\begin{align*}
\psi_{\PBJ}(X) \coloneqq \textbf{1}_{\{\PBJ_p > 2(2+\gamma)\log p\}}.
\end{align*}
Since the one sided $p$-values $p_{k,j} \coloneqq \bar{\Phi}(Y_{jk})$ are distributed as iid Uniform$(0,1)$ variables under the null, the proof of the Type I error is exactly the same as in Appendix \ref{UBmain2}. We proceed to show that the Type II error also goes to zero.

For any $\theta \in \Theta_1^{\mathrm{1-side}}(p,n,\rho,s)$, the Type II error is,
\begin{align*}
\P_\theta(\PBJ_p \leq 2(2+\gamma)\log p) &= \P_\theta \left( \max_{k\leq |\T|, q> 0} p K(\bar{S}_{p,k}(q),q) \leq \big(2(2+\gamma)+2a(1+o(1))\big)\log p \right),
\end{align*}
since $|\T| = p^{a(1+o(1))}$. To prove that the Type II error goes to zero, it suffices to pick $k\leq |\T|$ and $q > 0$ such that $K(\bar{S}_{p,k}(q),q) = \omega(p^{-1+\zeta})$ in probability for some $\zeta > 0$, i.e.
\begin{align}
\label{type2crit2}
p^{1-\zeta} K(\bar{S}_{p,k}(q),q) \stackrel{\P_\theta}{\longrightarrow} \infty \hspace{1em} \text{as $p \to \infty$.}
\end{align}
Denote by $t^*$ the true changepoint location in $\theta$. By the discussion in Appendix \ref{UBmain2} about threshold tests (see, e.g. expression (\ref{threshtest})), we may assume that the lower bounds $\rho$ and $s$ on the signal size and sparsity in the definition of $\Theta_1^{\mathrm{1-side}}(p,n,\rho,s)$ are achieved. Without loss of generality, suppose $t^*\leq n/2$; by symmetry of $\T$, an analogous argument can be made for $t^* > n/2$. By Lemma \ref{closegrid}, there exists some $\tilde{t} \coloneqq \lfloor (1+\d)^{\tilde{k}} \rfloor \in \T$ for which
\begin{align}
\label{murho3}
Y_{j\tilde{k}} \sim N(\mu,1) \text{ where } \mu = (1+o(1)) \rho,
\end{align}
when $j$ corresponds to a non-null row in $\theta$. When $j$ corresponds to a null row in $\theta$,
\begin{align*}
Y_{jk} \sim N(0,1) \hspace{2em} \text{for all }  k=1,\dots,|\T|.
\end{align*}
Throughout the following calculations, $\mu$ refers to the mean satisfying (\ref{murho3}). Since $r \in (0,1)$, there are two cases: $r \in (0,\frac{1}{4}]$ and $r \in (\frac{1}{4},1)$. Case 3 is when $\beta = 1$, where it is shown that as long as $r > 1+a$, we have $\mathcal{R}_{\mathrm{1-side}}(p,n,\rho=\sqrt{2r\log p},1) \to 0$.

\subsubsection*{Case 1: $0<r\leq \frac{1}{4}$} 
In this case, $\rho = \sqrt{2r\log p} = \sqrt{(2\beta_1-1)\log p}$, and $\beta_1 \in (\frac{1}{2},\frac{3}{4}]$. Put $q \coloneqq \bar{\Phi}(2\mu) \to 0$. Then,
\begin{align}
\nonumber
\E_\theta \bar{S}_{p,\tilde{k}}(q) &= \frac{1}{p} \sum_{j=1}^p \P_\theta \left( Y_{j\tilde{k}} > \bar{\Phi}^{-1}(q) \right) \\
\nonumber
&= (1-\e) \P(Z > 2\mu) + \e \P(\mu + Z > 2\mu) \\
\label{expmain3}
&= (1-\e) \bar{\Phi}(2\mu) + \e \bar{\Phi}(\mu), \\
\label{varmain3}
\var_\theta(\bar{S}_{p,\tilde{k}}(q)) &\leq \frac{1}{p^2} \sum_{j=1}^p \P_\theta(Y_{j\tilde{k}}>\bar{\Phi}^{-1}(q)) = p^{-1} \E_\theta \bar{S}_{p,\tilde{k}}(q).
\end{align}
Applying Mill's ratio to bound the difference $\E_\theta \bar{S}_{p,\tilde{k}}(q) - q$, we obtain
\begin{align}
\nonumber
\E_\theta \bar{S}_{p,\tilde{k}}(q) - q &= \e(\bar{\Phi}(\mu) - \bar{\Phi}(2\mu)) \\
\label{expqdiffmain3}
&= p^{-\beta}\big(p^{-\frac{1}{2}(2\beta_1-1)(1+o(1))} - p^{-2(2\beta_1-1)(1+o(1))}\big).
\end{align}
It is straightforward to check that the assumptions $\beta < \beta_1$ and $\beta_1 \leq \frac{3}{4}$ imply the following inequalities,
\begin{align*}
\frac{1}{2}(-1-2(2\beta_1-1)) &< -\beta-\frac{1}{2}(2\beta_1-1) \\
\frac{1}{2}\big(-1-\beta-\frac{1}{2}(2\beta_1-1)\big) &< -\beta-\frac{1}{2}(2\beta_1-1).
\end{align*}
Together with (\ref{expmain3}), (\ref{varmain3}), and (\ref{expqdiffmain3}), the above two inequalities imply that,
\begin{align}
\nonumber
\sqrt{\var_\theta(\bar{S}_{p,\tilde{k}}(q))} &\leq \sqrt{p^{-1}\E_\theta \bar{S}_{p,\tilde{k}}(q)}\\
\nonumber
&= \sqrt{p^{-1}\big(p^{-2(2\beta_1-1)(1+o(1))}+p^{-\beta-\frac{1}{2}(2\beta_1-1)(1+o(1))}\big)} \\
\nonumber
&= \sqrt{p^{-1-2(2\beta_1-1)(1+o(1))}+p^{-1-\beta-\frac{1}{2}(2\beta_1-1)(1+o(1))}} \\
\label{var1main3}
&= o(\E_\theta \bar{S}_{p,\tilde{k}}(q)-q).
\end{align}
Next, observe that
\begin{align}
\label{asymp:chebyshev-cond}
\frac{\var_\theta(\bar{S}_{p,\tilde{k}}(q))}{(\E_\theta \bar{S}_{p,\tilde{k}}(q))^2} \leq \frac{1}{p \E_\theta \bar{S}_{p,\tilde{k}}(q)} = \frac{1}{p(p^{-2(2\beta_1-1)(1+o(1))}+p^{-\beta-\frac{1}{2}(2\beta_1-1)(1+o(1))})} \to 0,
\end{align}
since $1-\beta-\frac{1}{2}(2\beta_1-1) > 0$ is implied by the assumptions $\beta < \beta_1$ and $\beta_1 \leq \frac{3}{4}$. 

Similar to the proof of the upper bound in Theorem \ref{main2} (see Appendix \ref{UBmain2}, Case 2), we split the remaining analysis into two further cases, depending on which of the two terms, either $\bar{\Phi}(2\mu)$ or $\e \bar{\Phi}(\mu)$, dominates in expression \eqref{expmain3} for the mean $\E_\theta \bar{S}_{p,\tilde{k}}(q)$. 

\textbf{(Subcase 1).} Suppose that the first term dominates, that is,
\begin{align*}
    \E_\theta \bar{S}_{p,\tilde{k}}(q) = (1-\e) \bar{\Phi}(2\mu) + \e \bar{\Phi}(\mu) \leq 2 \bar{\Phi}(2\mu).
\end{align*}
Together with Chebyshev's inequality and condition \eqref{asymp:chebyshev-cond}, this implies
\begin{align*}
    \frac{\bar{S}_{p,\tilde{k}}(q)}{q} &= \frac{\E_\theta \bar{S}_{p,\tilde{k}}(q)}{q}(1+o_{\P_\theta}(1)) \\
    &\leq \frac{2 \bar{\Phi}(2\mu)}{q}(1+o_{\P_\theta}(1)) = 2(1+o_{\P_\theta}(1)).
\end{align*}
By Part 2 of Lemma \ref{kbehave}, on a set with probability tending to 1, we have
\begin{align*}
    K(\bar{S}_{p,\tilde{k}}(q),q) &\geq \frac{(\bar{S}_{p,\tilde{k}}(q)-q)^2}{9q} \\
    &= \frac{(\bar{S}_{p,\tilde{k}}(q)-\E_\theta \bar{S}_{p,\tilde{k}}(q)+\E_\theta \bar{S}_{p,\tilde{k}}(q)-q)^2}{9q},
\end{align*}
which, by another application of Chebyshev's inequality and \eqref{var1main3}, is equal to
\begin{align*}
    &= \frac{(\E_\theta \bar{S}_{p,\tilde{k}}(q) - q)^2}{9q}(1+o_{\P_\theta}(1)).
\end{align*}
To show (\ref{type2crit2}) holds for $\tilde{k}$ and $q$, by (\ref{expqdiffmain3}) it suffices to check that
\begin{align*}
p \cdot \frac{\big( p^{-\beta-\frac{1}{2}(2\beta_1-1)(1+o(1))} - p^{-\beta-2(2\beta_1-1)(1+o(1))} \big)^2}{p^{-2(2\beta_1-1)(1+o(1))}} \to \infty.
\end{align*}
The above divergence is implied by $1-2\beta-(2\beta_1-1)+2(2\beta_1-1) > 0$, which in turn is implied by the assumption $\beta < \beta_1$.

\textbf{(Subcase 2).} Now suppose that the second term in \eqref{expmain3} dominates, in the sense that, once $p$ exceeds some universal constant, we have
\begin{align*}
    \frac{\E_\theta \bar{S}_{p,\tilde{k}}(q)}{q}=\frac{\E_\theta \bar{S}_{p,\tilde{k}}(q)}{\bar{\Phi}(2\mu)} \geq e^2 + 1. 
\end{align*}
Chebyshev's inequality and \eqref{asymp:chebyshev-cond} now imply that $\P_\theta(e^2 \leq \bar{S}_{p,\tilde{k}}(q)/q) \to 1$. Thus, Part 3 of Lemma \ref{kbehave} can be applied together with Chebyshev's inequality to claim that, on a set with probability tending to 1,
\begin{align*}
    K(\bar{S}_{p,\tilde{k}}(q),q) &\geq \frac{1}{2}  \bar{S}_{p,\tilde{k}}(q) \log \frac{\bar{S}_{p,\tilde{k}}(q)}{q}\tag{Part 3 of Lemma \ref{kbehave}}\\
    &\geq \bar{S}_{p,\tilde{k}}(q) \tag{$\log (\bar{S}_{p,\tilde{k}}(q)/q) \geq 2$}\\
    &= (\E_\theta \bar{S}_{p,\tilde{k}}(q))(1+o_{\P_\theta}(1)) \tag{Chebyshev} \\
    &\geq \e \bar{\Phi}(\mu) (1+o_{\P_\theta}(1)) = p^{-\beta + \frac{1}{2}-\beta_1} (1+o_{\P_\theta}(1)).
\end{align*}
Finally, since $\beta_1 < 3/4 \Rightarrow 2\beta_1 < 3/2$, and since $\beta<\beta_1$ implies $\beta+\beta_1 < 2\beta_1$, we must have that $\beta+\beta_1 < 3/2$ and thus $-\beta+1/2-\beta_1 > -1$, from which \eqref{type2crit2} now follows.

\subsubsection*{Case 2: $\frac{1}{4}<r<1$} 
In this case, $\rho = \sqrt{2r\log p} = (1-\sqrt{1-\beta_1})\sqrt{2\log p}$, and $\beta_1 \in (\frac{3}{4},1)$. Put $q \coloneqq \bar{\Phi}(\sqrt{2\log p}) \to 0$. Then,
\begin{align*}
\E_\theta \bar{S}_{p,\tilde{k}}(q) &= (1-\e) \P(Z > \sqrt{2\log p}) + \e \P(\mu + Z > \sqrt{2\log p}) \\
&= (1-\e) \bar{\Phi}(\sqrt{2\log p}) + \e \bar{\Phi}\big((1+o(1))\sqrt{1-\beta_1}\sqrt{2\log p}\big) \\
&= p^{-(1+o(1))} + p^{-\beta - (1-\beta_1)(1+o(1))}, \\
\frac{\var_\theta(\bar{S}_{p,\tilde{k}}(q))}{(\E_\theta \bar{S}_{p,\tilde{k}}(q))^2} &\leq \frac{1}{p \E_\theta \bar{S}_{p,\tilde{k}}(q)} = \frac{1}{p(p^{-(1+o(1))} + p^{-\beta-(1-\beta_1)(1+o(1))})} \to 0,
\end{align*}
since $\beta < \beta_1$. It then follows by Chebyshev's inequality that
\begin{align*}
\frac{\bar{S}_{p,\tilde{k}}(q)}{q} &= \frac{\E_\theta \bar{S}_{p,\tilde{k}}(q)}{q}(1+o_{\P_\theta}(1)) \\
&= \frac{p^{-(1+o(1))}+p^{-\beta-(1-\beta_1)(1+o(1))}}{p^{-(1+o(1))}}(1+o_{\P_\theta}(1)) \\
&\to \infty,
\end{align*}
since $\beta < \beta_1$. The same application of Chebyshev's inequality gives
\begin{align*}
\bar{S}_{p,\tilde{k}}(q) =\E_\theta \bar{S}_{p,\tilde{k}}(q) (1+o_{\P_\theta}(1)) \to 0 \mathrm{ in } \P_\theta.
\end{align*}
It follows that on a set with probability tending to 1, we have that $e^2\leq \frac{\bar{S}_{p,\tilde{k}}(q)}{q}$ and $\frac{1}{2} \E_\theta \bar{S}_{p,\tilde{k}}(q) \leq \bar{S}_{p,\tilde{k}}(q) \leq \frac{1}{2}$. On this set, by Part 3 of Lemma \ref{kbehave}, 
\begin{align*}
K(\bar{S}_{p,\tilde{k}}(q),q) &\geq \frac{1}{2} \bar{S}_{p,\tilde{k}}(q) \log \frac{\bar{S}_{p,\tilde{k}}(q)}{q} \\
&\geq \frac{1}{2} \E_\theta \bar{S}_{p,\tilde{k}}(q) \tag{$\log e^2=2$}\\
&\asymp p^{-\beta-(1-\beta_1)(1+o(1))} .
\end{align*}
(\ref{type2crit2}) now follows since $\beta < \beta_1$.

\subsubsection*{Case 3: $\beta = 1$}

In this case, we show that if $r>1+a$ and $\rho \coloneqq \sqrt{2r\log p}$, then $\mathcal{R}_{\mathrm{1-side}}(p,n,\rho,1) \to 0$, as $p \to \infty$. 
Consider the test,
\begin{align}
\label{eq:max-test}
\psi_{\max}(X) \coloneqq \textbf{1}\left\{\max_{k\leq |\T|,j\leq p} Y_{jk} > \sqrt{(2+\gamma)\log (p |\T|)} \right\},
\end{align}
where the $\{Y_{jk}\}$ are defined in (\ref{xtoy}), and $\gamma > 0$ is a small constant satisfying,
\begin{align}
\label{gammamain3}
\sqrt{(2+\gamma)(1+a)} < \sqrt{2r}.
\end{align}
Since each $Y_{jk} \sim N(0,1)$ under any null parameter $\theta \in \Theta_0(p,n)$, a union bound yields the Type I error control,
\begin{align*}
\P_\theta\left(\max_{k\leq |\T|,j\leq p} Y_{jk} > \sqrt{(2+\gamma)\log(p|\T|)} \right) &\leq \sum_{k=1}^{|\T|} \sum_{j=1}^p \bar{\Phi}\big(\sqrt{(2+\gamma)\log (p|\T|)} \big) \\
&= p^{1+a(1+o(1))} p^{-\frac{2+\gamma}{2}(1+a(1+o(1)))} \\
&\to 0,
\end{align*}
since $|\T| = p^{a(1+o(1))}$. Now suppose $\theta \in \Theta_1^{\mathrm{1-side}}(p,n,\rho,1)$, so that $\theta$ has at least one non-null row. The Type II error is
\begin{align*}
\P_\theta(\psi_{\max} = 0) = \P_\theta\left(\max_{k\leq |\T|,j\leq p}Y_{jk} \leq \sqrt{(2+\gamma)\log (p|\T|)} \right) \leq \P_\theta\big(Y_{jk} \leq \sqrt{(2+\gamma)\log (p|\T|)}\big),
\end{align*}
for any $j\leq p$ and $k\leq |\T|$. Since the max test only becomes more powerful if the means of $Y_{jk}$ are increased, we may assume that the lower bound $\rho$ on the signal size in the definition of $\Theta_1^{\mathrm{1-side}}(p,n,\rho,s=1)$ is achieved. By definition of $\theta \in \Theta_1^{\mathrm{1-side}}(p,n,\rho,s=1)$, there is some non-null row in $\theta$. Denote the index of this non-null row as $i \in [p]$. Without loss of generality, suppose that the true changepoint location $t^*$ in $\theta$ satisfies $t^* \leq n/2$; by symmetry of the grid, an analogous argument can be made for $t^* > n/2$. By the second part of Lemma \ref{closegrid}, there exists some $\tilde{t} \coloneqq \lfloor (1+\d)^{\tilde{k}} \rfloor \in \T$ for which
\begin{align*}
Y_{i\tilde{k}} \sim N(\mu,1) \text{ where } \mu = (1+o(1)) \rho,
\end{align*}
under $\P_\theta$. Then the Type II error is bounded,
\begin{align*}
\P_\theta(\psi_{\max} = 0) &\leq \P_\theta\big(Y_{i\tilde{k}} \leq \sqrt{(2+\gamma)\log(p|\T|)} \big) \\
&= \P\big(\mu + Z \leq \sqrt{(2+\gamma)(1+a)\log p} \big) \tag{$Z\sim N(0,1)$}\\
&= \Phi\left(\sqrt{(2+\gamma)(1+a)\log p} - (1+o(1))\sqrt{2r \log p} \right)\\
&\to 0,
\end{align*}
by the assumption that $\gamma$ satisfies (\ref{gammamain3}).

\subsection{Lower bound for Theorem \ref{thm:two-changepoints}}
\label{sec-two-changepoints-LB}

The cases in which $1-2\beta_1 \leq a$ are covered by Theorem 1 in \cite{chan2015optimal}, so we will only show the first case here. To this end, suppose $a \leq 1-2\beta_1$, let $\beta_1 < \beta$ and put $\rho^2= p^{\frac{a-(1-2\beta_1)}{2}} \to 0$. Consider testing between
\begin{align*}
    &H_0: (X_{j:})_{j\leq p} \sim \prod_{j=1}^p N(0,I_n) \\
    &H_1: (X_{j:})_{j\leq p} \sim \frac{1}{n-1} \sum_{k=1}^{n-1} \prod_{j=1}^p \left[(1-\e)N(0,I_n) + \e \left(\frac{1}{2}  N(\rho \theta^{(k)},I_n) + \frac{1}{2} N(-\rho\theta^{(k)},I_n) \right)\right],
\end{align*}
where $\theta^{(k)}_i=1\{i=k+1\}$ for $i=1,\dots,n$. We will show
\begin{align*}
    \limsup_{p\to \infty}\E_1 \left( \frac{f_1}{f_0}(X) \right) \leq 1.
\end{align*}
To this end, we have
\begin{align*}
    \frac{f_1}{f_0}(X) = \frac{1}{n-1} \sum_{k=1}^{n-1} \prod_{j=1}^p \left[1-\e+\e e^{-\rho^2/2}\left(\frac{e^{\rho\langle X_{j:},\theta^{(k)}\rangle}+e^{-\rho \langle X_{j:},\theta^{(k)}\rangle}}{2}\right) \right].
\end{align*}
Distributing the expectation, we have
\begin{align*}
    \E_1 \frac{f_1}{f_0}(X) &= \frac{1}{n-1} \sum_{k=1}^{n-1} \E_1 \prod_{j=1}^p \left[1-\e+\e e^{-\rho^2/2}\left(\frac{e^{\rho\langle X_{j:},\theta^{(k)}\rangle}+e^{-\rho \langle X_{j:},\theta^{(k)}\rangle}}{2}\right) \right] .
\end{align*}
Conditional on the location $k^*=\ell$ of the changepoint, the noise variables across $j=1,\dots,p$ are independent, so the above is equal to
\begin{align*}
    = \frac{1}{(n-1)^2} \sum_{k,\ell\leq n-1}  \prod_{j=1}^p \left[1-\e+\e e^{-\rho^2/2}\E_1\left(\frac{e^{\rho\langle X_{j:},\theta^{(k)}\rangle}+e^{-\rho \langle X_{j:},\theta^{(k)}\rangle}}{2} \mid k^*=\ell\right) \right] .
\end{align*}
Next, note that
\begin{align*}
    \E_1 \left( e^{\rho \langle X_{j:},\theta^{(k)} \rangle} \mid k^*=\ell \right) &= (1-\e) e^{\rho^2/2} + \e\left( \frac{1}{2} \E e^{\rho\langle \rho \theta^{(\ell)}+Z_{j:},\theta^{(k)}\rangle} + \frac{1}{2}\E e^{\rho \langle -\rho \theta^{(\ell)}+Z_{j:},\theta^{(k)}\rangle}\right) \\
    &= \begin{cases}
        e^{\rho^2/2} \hspace{1em} &\text{if }k\neq \ell \\
        (1-\e)e^{\rho^2/2} + \e e^{\rho^2/2}\left( \frac{1}{2} e^{\rho^2} + \frac{1}{2} e^{-\rho^2} \right) &\text{if }k=\ell.
    \end{cases}
\end{align*}
Arguing by symmetry,
\begin{align*}
    \E_1 \left(e^{-\rho\langle X_{j:},\theta^{(k)}\rangle }\mid k^*=\ell\right)=\E_1 \left( e^{\rho \langle X_{j:},\theta^{(k)} \rangle} \mid k^*=\ell \right) &\leq \begin{cases}
        e^{\rho^2/2} \hspace{1em} &\text{if } k\neq \ell \\
        (1-\e) e^{\rho^2/2} + \e e^{\rho^2/2+\rho^4/2} &\text{if } k= \ell.
    \end{cases}
\end{align*}
Plugging this into our expression for $\E_1 \frac{f_1}{f_0}(X)$, we have
\begin{align*}
    \E_1 \frac{f_1}{f_0}(X) &\leq \frac{1}{(n-1)^2} \left( \sum_{k\neq \ell}  \prod_{j=1}^p \left[ 1 -\e + \e  \right] + \sum_{k= \ell} \prod_{j=1}^p \left[1-\e + \e(1-\e+\e e^{\rho^4/2}) \right]\right) \\
    &= 1 + o(1) + \frac{1}{(n-1)^2} \cdot (n-1) \left[1 + \e^2 (e^{\rho^4/2}-1) \right]^p \\
    &\sim 1 + \frac{1}{n}\cdot \exp(p^{1-2\beta}(e^{\rho^4/2}-1)) 
\end{align*}
since $\beta > \beta_1$ implies
\begin{align*}
    p^{1-2\beta}(e^{\rho^4/2}-1) \sim p^{1-2\beta} \rho^4/2 \sim p^{1-2\beta+(a-(1-2\beta_1))}/2 \asymp p^{a+2(\beta_1-\beta)} \ll p^{a(1+o(1))}=\log n.
\end{align*}
Thus we have
\begin{align*}
    \E_1 \frac{f_1}{f_0}(X) \leq 1 + \exp\left(p^{a+2(\beta_1-\beta)+o(1)}-\log n \right) \to 1.
\end{align*}

\subsection{Upper bound for Theorem \ref{thm:two-changepoints}}
\label{sec-two-changepoints-UB}

\noindent\textbf{Test statistic 1.} For the case $a \leq 1-2\beta_1$, we may use the following test:
\begin{align*}
    \text{Reject $H_0$ if }\max_{t=(t_1,t_2)\in [n-1]^2} \; \sum_{j=1}^p Y_{j,t} > 2p+2\sqrt{2xp}+2x,
\end{align*}
where $x=2.1\log n$. Lemma \ref{lem:birge} implies type 1 error control, since
\begin{align*}
    \sum_{j=1}^p Y_{j,t} \sim \chi^2_{2p} \Rightarrow \P_0(\text{Reject }H_0) &\leq \sum_{t_1,t_2 \in [n-1]} \P_0(\chi^2_{2p} > 2p+2\sqrt{2xp}+2x) \\
    &\leq n^{2} e^{-x} \to 0.
\end{align*}
The type 2 error is
\begin{align*}
    \P_{\theta}\bigg(\max_{t=(t_1,t_2)\in [n-1]^2} \; \sum_{j=1}^p Y_{j,t} \leq 2p+2\sqrt{2xp}+2x \bigg) &\leq \P_\theta\bigg(\sum_{j=1}^p Y_{j,k} \leq 2p+2\sqrt{2xp}+2x \bigg) ,
\end{align*}
where $k \in [n-1]^2$ is the vector of true changepoint locations. Since $\sum_{j=1}^p Y_{j,k} \sim \chi^2_{2p,s\rho^2}$, we have
\begin{align*}
    \P_\theta\bigg(\sum_{j=1}^p Y_{j,k} \leq 2p+2\sqrt{2xp}+2x\bigg) \leq \P_\theta\bigg(\sum_{j=1}^p Y_{j,k} \leq s\rho^2+2p-2\sqrt{x(2p+2s\rho^2)}\bigg)
\end{align*}
as long as $s\rho^2 \geq 2x+2\sqrt{2xp}+2\sqrt{2x(p+s\rho^2)}$. Since $x\asymp \log n \sim p^{a(1+o(1))}$, this condition is equivalent to
\begin{align*}
    p^{1-\beta+ \frac{1}{2}(a-(1-2\beta_1))} \gtrsim p^{a(1+o(1))} + p^{\frac{1}{2}(a+1)} \asymp p^{\frac{1}{2}(a+1)},
\end{align*}
where the last equivalence follows because $a\leq 1-2\beta_1 < 1$. Inspecting the exponents, we see that the above condition is satisfied since $\beta<\beta_1$. The type 2 error probability goes to zero by the second part of Lemma \ref{lem:birge}.

\vspace{1em}

\noindent \textbf{Test statistic 2.} For the cases in which $a > 1-2\beta_1$, we use the following test. For any $t:=(t_1,t_2) \in [n-1]^2$ with $t_1 < t_2$, define the statistic
\begin{align}
\label{eq:count-test-statistic}
    Y_{j,t} \coloneqq \left\| \Sigma_{t_1,t_2}^{-1/2} \left( \bar{X}_{j,[t_1+1:t_2]}-\bar{X}_{j,[1:t_1]}, \; \bar{X}_{j,[t_2+1:n]} - \bar{X}_{j,[t_1+1:t_2]} \right)^\top \right\|^2,
\end{align}
where $\bar{X}_{j,[i:k]}$ represents the sample mean of the $i^{\th}$ through $k^{\th}$ entries of $X_{j:}$, so that, under $H_0$, we have $Y_{j,t} \sim \chi^2_2$ independently for $j=1,\dots,p$. Define the count proportion
\begin{align*}
    \bar{S}_{p,t} (q) \coloneqq \frac{1}{p}\sum_{j=1}^p 1\{Y_{j,t} > \bar{F}_2^{-1}(q)\},
\end{align*}
where $\bar{F}_{2}$ is the tail probability function for the $\chi^2_2$ distribution. Let $\gamma>0$. The penalized Berk-Jones test is defined
\begin{align*}
    \text{Reject $H_0$ if }\max_{t_1,t_2\in [n-1]:t_1<t_2} \; \max_{q \in (0,1)} \; pK(\bar{S}_{p,t}(q),q) - (4+\gamma)\log n > 2(2+\gamma) \log p
\end{align*}
where $K(x,y)\coloneqq x \log \frac{x}{y} + (1-x) \log \frac{1-x}{1-y}$.

\vspace{1em}

\noindent For the type 1 error of this test, a union bound yields
\begin{align*}
    \P_0(\text{Reject $H_0$}) &\leq \sum_{t_1=1}^{n-1} \sum_{t_2=t_1+1}^{n-1} \sum_{j=1}^p \P_0(pK(j/p,P_{t,(j)}) > 2(2+\gamma)\log p + (4+\gamma) \log n)
\end{align*}
where $P_{t,(j)}\coloneqq \bar{F}_2(Y_{j,t})$. By Lemma \ref{chernoff}, the above is bounded by
\begin{align*}
    &\leq \sum_{t_1=1}^{n-1}\sum_{t_2=t_1+1}^{n-1} \left[5 e^{-2(2+\gamma)\log p - (2+\gamma) \log n} + \sum_{j=2}^p e\sqrt{2} j e^{-\frac{1}{2}(2(2+\gamma)\log p + (4+\gamma) \log n)}\right] \\
    &\leq 5 n^2 e^{-2(2+\gamma)\log p - (4+\gamma) \log n} + n^2 p^2 e\sqrt{2} e^{-\frac{1}{2}(2(2+\gamma)\log p + (4+\gamma)\log n)} \\
    &\lesssim p^{-2(2+\gamma)} n^{-2} + p^{2-(2+\gamma)}  \to 0.
\end{align*}
For the type 2 error, as in the proof of Theorem \ref{UBmain2}, we assume that the signal requirement is tight, i.e. for each non-null $j$ the lower bound on the signal is achieved,
\begin{align*}
    \|\Sigma_{t_1,t_2}^{-1/2} (\Delta^{(1)}_j, \Delta^{(2)}_j)^\top \|^2 = \rho^2.
\end{align*}
We must now show for any $\theta$ satisfying the conditions of $H_1$,
\begin{align*}
    \P_\theta \left( \max_{t_1,t_2\in [n-1]:t_1<t_2} \; \max_{q \in (0,1)} K(\bar{S}_{p,t}(q),q) \leq \frac{2(2+\gamma)\log p + (2+\gamma)\log n}{p} \right) \to 0.
\end{align*}
Now since $(2+\gamma) \log n \asymp p^{a(1+o(1))}$, it is sufficient to pick some $t \in [n-1]^2$ and $q \in (0,1)$ such that
\begin{align}
\label{eq:type2-condition}
    \frac{K(\bar{S}_{p,k}(q),q)}{p^{a-1}} = \omega_{\P_\theta}(p^\zeta)
\end{align}
for some $\zeta>0$. Now let $\beta_1 \in (0,1)$ satisfying $\beta < \beta_1$ and put $\rho^2 \equiv \rho^*_{2-\text{side}}(a,\beta_1)^2$,
and let $t^* \in [n-1]^2$ denote the true changepoint locations, for which $Y_{j,t^*} \sim \chi^2_{2,\rho^2}$.

\subsubsection*{Case 2: $\beta = 1-2\beta_1< a<1-4\beta_1/3$}

In this case, $\rho^2 = (a-(1-2\beta_1))\log p$. Choose $q = \bar{F}_2(4\rho^2)$.
\begin{align*}
    \E_\theta \bar{S}_{p,k}(q) = (1-\e) q + \e \P(\chi^2_{2,\rho^2} > 4\rho^2)
\end{align*}
As expression \eqref{lem:chi-sq-tightness} indicates, the first part of Lemma \ref{lem:birge} is tight in the following sense:
\begin{align*}
    \P(\chi^2_{2,\rho^2} \geq \rho^2 + 2\sqrt{2x\rho^2} + 2x) \geq e^{-x(1+o(1))},
\end{align*}
as $x,\rho^2 \to \infty$. This implies
\begin{align}
\label{eq:case2-chebyshev}
    \E_\theta \bar{S}_{p,k}(q) \geq \e \P(\chi^2_{2,\rho^2} \geq 4\rho^2) \geq p^{-\beta}e^{-\frac{1}{2}\rho^2 (1+o(1))} = p^{-\beta - \frac{1}{2}(a-(1-2\beta_1))(1+o(1))},
\end{align}
since $4\rho^2 \sim \rho^2 + 2\sqrt{2\rho^2 x}+2x \Rightarrow \sqrt{x}\sim \rho /\sqrt{2}$. It follows that
\begin{align*}
    \frac{\var(\bar{S}_{p,k}(q))}{(\E_\theta \bar{S}_{p,k}(q))^2} \leq \frac{1}{p\E_\theta \bar{S}_{p,k}(q)} \to 0
\end{align*}
since $1-\beta-\frac{1}{2}(a-(1-2\beta_1)) > 0$ in this case.
Expression \eqref{lem:chi-sq-tightness} of Lemma \ref{lem:birge} also implies $q=\bar{F}_2(4\rho^2) \geq e^{-2\rho^2(1+o(1))}$, which yields
\begin{align*}
    \E_\theta \bar{S}_{p,k}(q) \geq \max \left\{ p^{-\beta - \frac{1}{2}(a-(1-2\beta_1))(1+o(1))}, p^{-2(a-(1-2\beta_1))(1+o(1))} \right\}.
\end{align*}
\textbf{First suppose:} $-\beta-\frac{1}{2}(a-(1-2\beta_1)) < -2(a-(1-2\beta_1))$. Then
\begin{align*}
    \E_\theta \bar{S}_{p,k}(q) = q + \e( p^{-\frac{1}{2}(a-(1-2\beta_1))(1+o(1))}-q) \sim q,
\end{align*}
so eventually $\frac{\E_{\theta}\bar{S}_{p,k}(q)}{q} \leq 4$ is satisfied, which implies by Part 2 of Lemma \ref{kbehave} that
\begin{align*}
    K(\bar{S}_{p,k}(q),q) \geq \frac{(\bar{S}_{p,k}(q)-q)^2}{9q} = \frac{(\E_\theta \bar{S}_{p,k}(q)-q)^2(1+o_{\P_\theta}(1))}{9q},
\end{align*}
where the equality follows from $-1-2(a-(1-2\beta_1)) < -2\beta-(a-(1-2\beta_1))$ and
\begin{align*}
    \var(\bar{S}_{p,k}(q)) &\leq p^{-1} \E_{\theta}\bar{S}_{p,k}(q) \sim p^{-1-2(a-(1-2\beta_1))(1+o(1))} \\
    \E_\theta \bar{S}_{p,k}(q) - q &\geq p^{-\beta-\frac{1}{2}(a-(1-2\beta_1))(1+o(1))}.
\end{align*}
Then it follows that
\begin{align*}
    \frac{(\E_\theta \bar{S}_{p,k}(q)-q)^2}{q} &\gtrsim p^{-2\beta-(a-(1-2\beta_1))(1+o(1))} / p^{-2(a-(1-2\beta_1))(1+o(1))} \\
    &= p^{(a-(1-2\beta_1))(1+o(1)) - 2\beta} \gg p^{a-1}
\end{align*}
since $\beta < \beta_1$, which yields \eqref{eq:type2-condition} with, e.g. $\zeta=\beta_1-\beta$. \\

\noindent \textbf{Now suppose:} $-\beta-\frac{1}{2}(a-(1-2\beta_1)) > -2(a-(1-2\beta_1))$. Then \eqref{eq:case2-chebyshev} implies
\begin{align*}
    \bar{S}_{p,k}(q) = (\E_\theta \bar{S}_{p,k}(q)) (1+o_{\P_\theta}(1)).
\end{align*}
We also have
\begin{align*}
    \frac{\E_\theta \bar{S}_{p,k}(q)}{q} \geq \frac{p^{-\beta-\frac{1}{2}(a-(1-2\beta_1))(1+o(1))}}{p^{-2(a-(1-2\beta_1))(1+o(1))}} \to \infty.
\end{align*}
Part 3 of Lemma \ref{kbehave} implies
\begin{align*}
    K(\bar{S}_{p,k}(q),q) \geq (\E_\theta \bar{S}_{p,k}(q))(1+o_{\P_\theta}(1)) \sim p^{-\beta-\frac{1}{2}(a-(1-2\beta_1))} \gg p^{a-1}
\end{align*}
since $-\beta-\frac{1}{2}(a-(1-2\beta_1)) > a-1$ happens iff
\begin{align*}
     -\beta > \frac{3}{2}(a-1) + \beta_1 \Leftarrow -\beta_1 \geq \frac{3}{2}(a-1)+\beta_1 \iff 1-\frac{4}{3} \beta_1 \geq a,
\end{align*}
which was assumed at the start of Case 2.

\subsubsection*{Case 3: $1-4\beta_1/3 < a \leq 1-\beta_1$}

In this case, $\rho^2 = 2(x-y)^2\log p$ where $x=\sqrt{1-a}$ and $y=\sqrt{1-a-\beta_1}$. Choose $q = \bar{F}_{2}(2x^2 \log p)$. Lemma \ref{lem:birge} implies
\begin{align*}
    \E_\theta \bar{S}_{p,k}(q) &= (1-\e) q + \e \P(\chi^2_{2,\rho^2}> 2x^2\log p) \\
    &\geq p^{-x^2(1+o(1))} + p^{-\beta - y^2(1+o(1))},
\end{align*}
since $q \geq e^{-\frac{1}{2} \cdot 2x^2(1+o(1)) \log p} = p^{-x^2(1+o(1))}$ and
\begin{align*}
    \P(\chi^2_{2,\rho^2}>2x^2 \log p) \geq \exp(-y^2(1+o(1))\log p),
\end{align*}
since $2x^2 \log p \sim \rho^2 + 2\sqrt{2z\rho^2}+2z$ implies $\sqrt{z}\sim y\sqrt{\log p}$. It follows now that
\begin{align*}
    \frac{\var(\bar{S}_{p,k}(q))}{(\E_\theta \bar{S}_{p,k}(q))^2} &\leq \frac{1}{p\E_{\theta}\bar{S}_{p,k}(q)} \\
    &\leq \frac{1}{p^{1-x^2(1+o(1))}+ p^{1-\beta-y^2(1+o(1))}} \to 0,
\end{align*}
since $y^2=1-a-\beta_1 < 1-\beta$. Chebyshev's inequality implies
\begin{align*}
    \frac{\bar{S}_{p,k}(q)}{q} =\frac{(\E_\theta \bar{S}_{p,k}(q))(1+o_{\P_\theta}(1))}{q} \geq \frac{p^{-x^2(1+o(1))}+p^{-\beta-y^2(1+o(1))}}{p^{-x^2(1+o(1))}} \to \infty
\end{align*}
since $\beta<\beta_1$ implies $-\beta-y^2=-\beta-1+a+\beta_1 > a-1= -x^2$. Part 3 of Lemma \ref{kbehave} implies
\begin{align*}
    K(\bar{S}_{p,k}(q),q) \geq (\E_{\theta}\bar{S}_{p,k}(q))(1+o_{\P_\theta}(1)) \geq p^{-\beta-y^2 + o(1)} \gg p^{a-1}
\end{align*}
since $-\beta-y^2 = -\beta-(1-a-\beta_1)$ and $\beta<\beta_1$.

\subsubsection*{Case 4: $a > 1-\beta_1$}

In this case, $\rho^2 = p^{a-(1-\beta_1)} \to \infty$. Choose $q = \bar{F}_2(\rho^2/2)\leq e^{-\frac{1}{4}p^{a-(1-\beta_1)}} \to 0$. Then
\begin{align*}
    \E_\theta \bar{S}_{p,k}(q) &= (1-\e) q + \e \P(\chi^2_{2,\rho^2} > \rho^2/2) \\
    &\geq \e (1-\P(\chi^2_{2,\rho^2(1+o(1))} \leq \rho^2/2) ) \\
    &\geq \e (1-e^{-\frac{\rho^2}{32}(1+o(1))}) \sim \e,
\end{align*}
by the second part of Lemma \ref{lem:birge}, since $\rho^2/2 \sim \rho^2 - 2\sqrt{2x\rho^2}$ implies $\sqrt{x} \sim \frac{\rho}{4\sqrt{2}}$ and
\begin{align*}
    \P(\chi^2_{2,\rho^2} \leq \rho^2/2) \leq e^{-\left(\frac{\rho^2}{4\sqrt{2}}\right)^2}.
\end{align*}
Thus, we have the condition
\begin{align*}
    \frac{\var(\bar{S}_{p,k}(q))}{(\E_\theta \bar{S}_{p,k}(q))^2} \leq \frac{1}{p\E_{\theta}\bar{S}_{p,k}(q)} \leq \frac{1}{p\e} \to 0
\end{align*}
since $\beta < 1$. Now since
\begin{align*}
    \frac{\bar{S}_{p,k}(q)}{q} = \frac{(\E_\theta \bar{S}_{p,k}(q))(1+o_{\P_\theta}(1))}{q} \geq p^{\beta-1} \cdot \exp\left( \frac{p^{a-(1-\beta_1)}}{4} \right) \to \infty.
\end{align*}
Part 3 of Lemma \ref{kbehave} implies
\begin{align*}
    K(\bar{S}_{p,k}(q),q) &\geq \frac{1}{2} (\E_\theta \bar{S}_{p,k}(q))(1+o_{\P_\theta}(1)) \log \frac{\E_\theta \bar{S}_{p,k}(q)}{q} \\
    &\geq \frac{1}{2}\e \left((\beta-1)\log p + \frac{p^{a-(1-\beta_1)}}{4}\right) \\
    &\sim \frac{p^{a-1+\beta_1-\beta}}{8} \gg p^{a-1}
\end{align*}
since $\beta_1>\beta$.


\section{Technical Lemmas}
\label{technical}
This supplement file contains the statements and proofs of several technical lemmas that were referenced in the proofs of the main results.

\begin{lemma}
\label{thetas}
Let $\T$ be the grid of geometrically growing changepoint locations with base $b$,
\begin{align*}
\T \coloneqq \{\lfloor b^1\rfloor,\lfloor b^2 \rfloor, \dots, \lfloor b^{\log_{\log n}n}\rfloor \}.
\end{align*}
We abuse notation and treat $b^k$ as $\lfloor b^k \rfloor$ in the statements and calculations that follow. For each $k =1,\dots,|\T|$, recall the definition of the vector $\theta^{(k)} \in \R^n$,
\begin{align*}
\theta^{(k)}_t \coloneqq \begin{cases}
\left(\frac{n-b^k}{nb^k} \right)^{1/2} \hspace{1em} &\text{if } t \leq b^k \\
-\left(\frac{b^k}{n(n-b^k)} \right)^{1/2} &\text{if } t > b^k.
\end{cases}
\end{align*}
Then the following are true for each $k,l \leq |\T|$, and $Z \in \R^n$.
\begin{enumerate}
\item $\langle Z,\theta^{(k)}\rangle = \sqrt{\frac{b^k(n-b^k)}{n}} \left(\overline{Z}_{1:b^k}-\overline{Z}_{b^k+1:n} \right)$.
\item $\|\theta^{(k)}\|= 1$.
\item $\langle \theta^{(k)},\theta^{(l)} \rangle \leq b^{-\frac{|k-l|}{2}}$.
\end{enumerate}
\end{lemma}

\begin{proof}
For the first identity, apply the definition of $\theta^{(k)}$ to obtain,
\begin{align*}
\langle Z,\theta^{(k)} \rangle &= \left(\left(\frac{n-b^k}{nb^k} \right)^{1/2} \sum_{i\leq b^k} Z_i - \left(\frac{b^k}{n(n-b^k)} \right)^{1/2} \sum_{i> b^k} X_i \right) \\
&= \left(\frac{b^k(n-b^k)}{n} \right)^{1/2}\cdot \left(b^{-k} \sum_{i\leq b^{k}} Z_i - (n-b^k)^{-1} \sum_{i > b^k} Z_i \right) \\
&= \sqrt{\frac{b^k(n-b^k)}{n}} \left(\overline{Z}_{1:b^k}-\overline{Z}_{b^k+1:n} \right).
\end{align*}
For the second and third results, we compute the inner product for arbitrary $k,l \leq |\T|$. Suppose $k\leq l$. Then the inner product $\langle \theta^{(k)},\theta^{(l)}\rangle$ is equal to,
\begin{align*}
&= b^k \left(\frac{n-b^k}{nb^k}\cdot \frac{n-b^l}{nb^l}\right)^{1/2}-(b^l-b^k)\left(\frac{b^k}{n(n-b^k)}\cdot \frac{n-b^l}{nb^l}\right)^{1/2}\\
&+(n-b^l)\left( \frac{b^k}{n(n-b^k)}\cdot \frac{b^l}{n(n-b^l)}\right)^{1/2} \\
&= \frac{b^{\frac{k+l}{2}}}{n}((n-b^k)(n-b^l))^{1/2}\left[b^k \cdot \frac{1}{b^{k+l}} - (b^l-b^k)b^{\frac{k-l}{2}-\frac{k+l}{2}}\frac{1}{n-b^k}+\frac{n-b^l}{(n-b^k)(n-b^l)} \right] \\
&= \frac{b^{\frac{k+l}{2}}}{n}((n-b^k)(n-b^l))^{1/2}\left[\frac{1}{b^{l}} - (b^l-b^k)b^{-l} \frac{1}{n-b^k}+\frac{1}{n-b^k} \right] \\
&= \frac{b^{\frac{k+l}{2}}}{n}((n-b^k)(n-b^l))^{1/2}\left[\frac{n-b^k}{b^{l}(n-b^k)} -  \frac{b^l-b^k}{b^{l}(n-b^k)}+\frac{b^l}{b^l(n-b^k)} \right] \\
&= \frac{b^{\frac{k+l}{2}}}{n}((n-b^k)(n-b^l))^{1/2}\left[\frac{n}{b^l(n-b^k)} \right] \\
&= b^{\frac{k-l}{2}} \cdot \left(\frac{n-b^l}{n-b^k}\right)^{1/2}.
\end{align*}
An identical calculation gives a symmetric expression when $l \leq k$. Putting them together gives,
\begin{align*}
\langle \theta^{(k)},\theta^{(l)} \rangle = b^{-\frac{|k-l|}{2}}\left(\frac{n-b^{k\vee l}}{n-b^{k\wedge l}} \right)^{1/2} \leq b^{-\frac{|k-l|}{2}},
\end{align*}
which yields the third claim. Taking $k=l$ in the first equality gives $\|\theta^{(k)}\|=1$.

\end{proof}

\begin{lemma}
\label{lem-LRT-error-lower-bd}
Consider the simple vs simple testing problem,
\begin{align*}
H_0 : X \sim f_0 \hspace{1em}\mathrm{ versus } \hspace{1em}H_1: X \sim f_1.
\end{align*}
The likelihood ratio test is defined $\psi(x) = 1_{\left\{\frac{f_1}{f_0}(x) > 1\right\}}$. The sum of type 1 and 2 errors is bounded below,
\begin{align*}
\P_0(\psi=1) + \P_1(\psi=0) \geq \frac{1}{k} \P_1(A_k),
\end{align*}
for any $k\geq 1$, where $A_k \coloneqq \left\{ \frac{f_1}{f_0}(X) \leq k \right\}$, and $\P_0,\P_1$ are the probability measures associated with $f_0,f_1$.
\end{lemma}

\begin{proof}
The type 1 error can be lower bounded,
\begin{align*}
\P_0(\psi=1) \geq \P_0(\{\psi=1\}\cap A_k) \geq \E_0 \left( \frac{1}{k} \cdot \frac{f_1}{f_0}(X) 1_{\{\psi=1\}\cap A_k} \right) = \frac{\P_1(\{\psi=1\}\cap A_k)}{k}.
\end{align*}
The type 2 error is trivially lower bounded,
\begin{align*}
\P_1(\psi=0) \geq \frac{\P_1(\{\psi=0\}\cap A_k) }{k} .
\end{align*}
Summing these two inequalities gives the result.
\end{proof}

\begin{lemma}
\label{lem-Lk}
In the setting of Theorem \ref{main1}, for $k\neq k^*$, we have
\begin{align*}
\bar{\E}_1 (L_k \mid k^*) = 1+O(p \e^2 \rho^2e^{-\frac{1}{2}p^{a(1+o(1))}}). 
\end{align*}
\end{lemma}

\begin{proof}
Recall that
\begin{align*}
L_k = \prod_{j=1}^p \left(1 - \e + \e e^{\rho \langle X_{j:},\theta^{(k)}\rangle - \rho^2/2} \right), \hspace{1em} X_{j:} \mid k^* = \begin{cases}
\rho \theta^{(k^*)}+ Z \hspace{1em} &\text{if }Q_j=1 \\
Z &\text{if } Q_j=0
\end{cases}
\end{align*}
where $Z \sim N(0,I_n)$. Also by part 3 of Lemma \ref{thetas}, we have $\langle \theta^{(k)},\theta^{(l)}\rangle \leq~b^{-\frac{|k-l|}{2}}$, which is $\leq b^{-\frac{1}{2}}$ if $k\neq l$, where $b$ is related to the grid size, and in this setting is $b \sim \log n \sim e^{p^{a(1+o(1))}}$. Then for $k\neq k^*$, we have
\begin{align}
\label{dot-bound}
\langle \theta^{(k)},\theta^{(k^*)} \rangle \lesssim e^{-\frac{1}{2}p^{a(1+o(1))}}.
\end{align}
Since conditional on $k^*$, the rows of $X\in \R^{p \times n}$ are independent, we have
\begin{align*}
\bar{\E}_1 (L_k \mid k^*) = \prod_{j=1}^p \left(1-\e + \e \left[(1-\e) \E e^{\rho Z_1-\rho^2/2} + \e \E e^{\rho \langle \rho\theta^{(k^*)}+Z,\theta^{(k)}\rangle - \rho^2/2} \right] \right),
\end{align*}
where the expectation is now taken with respect to $Z \sim N(0,I_n)$, and $Z_1 \sim N(0,1)$. The above is equal to
\begin{align*}
&= \prod_{j=1}^p \left(1-\e + \e \left[(1-\e)\cdot 1 + \e \E e^{\rho^2 \langle \theta^{(k)},\theta^{(k^*)} \rangle +\rho Z_1 - \rho^2/2} \right] \right),
\end{align*}
where we have used that $\langle \theta^{(k)},Z \rangle \sim N(0,1)$ by part 2 of Lemma \ref{thetas}. Using (\ref{dot-bound}), the above becomes
\begin{align*}
&= \prod_{j=1}^p \left(1-\e + \e \left[(1-\e)\cdot 1 + \e \cdot e^{\rho^2 \langle \theta^{(k)},\theta^{(k^*)} \rangle} \right] \right) \\
&= \left(1-\e+\e -\e^2 + \e^2 e^{\rho^2 \langle \theta^{(k)},\theta^{(k^*)}\rangle } \right)^p \sim \exp\left(p \e^2 \rho^2 e^{-\frac{1}{2}p^{a(1+o(1))}} \right) = 1+ O(p \e^2 \rho^2e^{-\frac{1}{2}p^{a(1+o(1))}})
\end{align*}

\end{proof}

\begin{lemma}
\label{beta}
Let $p_{(1)},\dots,p_{(n)}$ be the order statistics of $n$ \text{Uniform$(0,1)$} random variables. Then
\begin{align*}
p_{(k)} \sim \mathrm{Beta}(k,n-k+1).
\end{align*}
\end{lemma}

\begin{proof}
This is a standard result, but we include the proof here for completeness. If $X_1,\dots,X_n \sim F$ are iid from a continuous CDF $F$, then the density of the $k^{th}$ order statistic is
\begin{align*}
f_{(k)}(x) &= nf(x) \binom{n-1}{k-1} F(x)^{k-1}(1-F(x))^{n-k}.
\end{align*}
Plugging in $f(x) = \textbf{1}_{\{x \in [0,1]\}}$ and $F(x) = x \textbf{1}_{\{x\in [0,1]\}}$ yields the Beta$(k,n-k+1)$ density.
\end{proof}

\begin{lemma}
\label{chernoff}
Let $p_{(1)},\dots,p_{(n)}$ be the order statistics of $n>2$ many \text{Uniform$(0,1)$} random variables, and put $K(x,t) \coloneqq x\log \frac{x}{t}+(1-x)\log\frac{1-x}{1-t}$. Then for any $s > 0$, and $j \in \{2,\dots,n\}$,
\begin{align*}
\P(n K(j/n,p_{(j)}) > s) \leq e\sqrt{2}je^{-(1-1/j)s}.
\end{align*}
When $j=1$, we have
\begin{align*}
\P(nK(1/n,p_{(1)})>s) \leq (1+9/e)e^{-s}.
\end{align*}
\end{lemma}

\begin{proof}
First suppose $j \geq 2$. For any $\lambda \in (0,1-1/j]$, Markov's inequality implies
\begin{align}
\nonumber
\P(n K(j/n,p_{(j)}) > s) &\leq e^{-\lambda s} \E \exp(\lambda n K(j/n,p_{(j)})) \\
\nonumber
&= e^{-\lambda s} \E \exp\left(\lambda j \log \frac{j/n}{p_{(j)}} + \lambda (n-j) \log \frac{(n-j)/n}{1-p_{(j)}} \right) \\
\label{KL-beta-expectation}
&= e^{-\lambda s} (j/n)^{\lambda j} (1-j/n)^{\lambda(n-j)} \E\left( p_{(j)}^{-\lambda j} (1-p_{(j)})^{-\lambda (n-j)} \right).
\end{align}
Now since $p_{(j)} \sim \mathrm{Beta}(j,n-j+1)$ (see, e.g. Lemma \ref{beta}), we have
\begin{align*}
    \E \left(p_{(j)}^{-\lambda j} (1-p_{(j)})^{-\lambda(n-j)}\right) &= \frac{\Gamma(n+1)}{\Gamma(j)\Gamma(n-j+1)} \int_0^1 u^{(j-\lambda j)-1}(1-u)^{(n-j+1-\lambda(n-j))-1} du \\
    &= \frac{\Gamma(n+1)}{\Gamma(j)\Gamma(n-j+1)} \frac{\Gamma(j(1-\lambda)) \Gamma((n-j)(1-\lambda)+1)}{\Gamma(n(1-\lambda)+1)}.
\end{align*}
Expression \eqref{KL-beta-expectation} is thus equal to
\begin{align*}
&= e^{-\lambda s} (j/n)^{\lambda j} (1-j/n)^{\lambda (n-j)} \frac{\Gamma(n+1)}{\Gamma(j)\Gamma(n-j+1)} \cdot \frac{\Gamma(j(1-\lambda))\Gamma((n-j)(1-\lambda)+1)}{\Gamma(n(1-\lambda)+1)}.
\end{align*}
Note that $j(1-\lambda) \geq 1$ is equivalent to $\lambda \leq 1-1/j$. Recalling Stirling's approximation, 
\begin{align*}
\sqrt{2\pi}\left(\frac{m}{e} \right)^m \sqrt{m}\leq \Gamma(m+1) \leq e\left(\frac{m}{e} \right)^m \sqrt{m}, \hspace{1em} \text{ for all } m\geq 1,
\end{align*}
we obtain the following upper bound on the tail probability 
\begin{align*}
\P(n K(j/n,p_{(j)}) > s) &\leq e^{-\lambda s} (j/n)^{\lambda j} (1-j/n)^{\lambda (n-j)} \frac{(n/e)^n \sqrt{n}}{\left(\frac{j-1}{e}\right)^{j-1}\sqrt{j-1} \left(\frac{n-j}{e}\right)^{n-j}\sqrt{n-j}} \times \\
&\frac{\left(\frac{j(1-\lambda)-1}{e}\right)^{j(1-\lambda)-1} \sqrt{j(1-\lambda)}\left(\frac{(n-j)(1-\lambda)}{e}\right)^{(n-j)(1-\lambda)} \sqrt{(n-j)(1-\lambda)}}{\left(\frac{n(1-\lambda)}{e}\right)^{n(1-\lambda)}\sqrt{n(1-\lambda)}} .
\end{align*}
Since $j\geq 2$, we have $\sqrt{\frac{j}{j-1}} \leq \sqrt{2}$, so the above is bounded by,
\begin{align*}
&\leq \sqrt{2} \cdot e^{-\lambda s} (j/n)^{\lambda j} (1-j/n)^{\lambda (n-j)} \frac{n^n(j(1-\lambda)-1)^{j(1-\lambda)-1}((n-j)(1-\lambda))^{(n-j)(1-\lambda)}}{(j-1)^{j-1}(n-j)^{n-j}(n(1-\lambda))^{n(1-\lambda)}}  \\
&=\sqrt{2} \cdot e^{-\lambda s} (j/n)^{\lambda j} (1-j/n)^{\lambda (n-j)} \frac{n^n}{(j-1)^{j-1}(n-j)^{n-j}} \times \\
&\frac{(j(1-\lambda)-1)^{j(1-\lambda)}((n-j)(1-\lambda))^{(n-j)(1-\lambda)}}{(n(1-\lambda))^{n(1-\lambda)}} \cdot \frac{1}{j(1-\lambda)-1}\\ 
&= \sqrt{2} \cdot e^{-\lambda s} \frac{j^{\lambda j} (n-j)^{\lambda(n-j)-(n-j)}}{n^{\lambda n-n}}\cdot \frac{1}{(j-1)^{j-1}} \times\\
& \frac{(j-1/(1-\lambda))^{j(1-\lambda)} (n-j)^{(n-j)(1-\lambda)}}{n^{n(1-\lambda)}} \cdot \frac{1}{j(1-\lambda)-1}\\
&= \sqrt{2} \cdot e^{-\lambda s} \frac{j^{\lambda j}}{(j-1)^{j-1}} \cdot (j-1/(1-\lambda))^{j(1-\lambda)}  \cdot \frac{1}{j(1-\lambda)-1} \\
&=\sqrt{2}\cdot e^{-\lambda s} j^{-j(1-\lambda)+1} \cdot \left(\frac{j}{j-1}\right)^{j-1}(j(1-\lambda)-1)^{j(1-\lambda)-1} \frac{1}{(1-\lambda)^{j(1-\lambda)}} \\
&\leq \sqrt{2} e \cdot e^{-\lambda s} j^{-j(1-\lambda)+1} j^{j(1-\lambda)-1} \frac{1}{(1-\lambda)^{j(1-\lambda)}} \tag{$\left(\frac{j}{j-1}\right)^{j-1}<e$}\\
&= \sqrt{2} e \cdot e^{-\lambda s} \cdot \frac{1}{(1-\lambda)^{j(1-\lambda)}}.
\end{align*}
Since the above holds for $\lambda \in (0,1-j^{-1}]$, which is non-empty for $j\geq 2$, let $\lambda = 1-j^{-1}$ so that the above becomes
\begin{align*}
\sqrt{2} e \cdot e^{-\lambda s} \cdot \frac{1}{(1-\lambda)^{j(1-\lambda)}} = \sqrt{2}e\cdot j e^{-(1-j^{-1})s} ,
\end{align*}
as desired. 

Now suppose $j=1$. we bound the tail probability by directly integrating the Beta density,
\begin{align}
\nonumber
\P(nK(1/n,p_{(1)}) > s) &= \P\left(\log \frac{1/n}{p_{(1)}}+(n-1)\log \frac{(n-1)/n}{1-p_{(1)}} > s \right) \\
\nonumber
&= \P\left(\frac{1/n}{p_{(1)}} \left(\frac{(n-1)/n}{1-p_{(1)}} \right)^{n-1} > e^s \right) \\
\label{betapc1}
&= \P\left(\frac{1/n}{p_{(1)}} \left(\frac{(n-1)/n}{1-p_{(1)}} \right)^{n-1} > e^s, (1-p_{(1)})^{n-1} > \frac{1}{2} \right)\\
\label{betapc2}
&+ \P\left(\frac{1/n}{p_{(1)}} \left(\frac{(n-1)/n}{1-p_{(1)}} \right)^{n-1} > e^s, (1-p_{(1)})^{n-1} \leq \frac{1}{2} \right).
\end{align}
(\ref{betapc1}) can be bounded as follows. By monotonicity,
\begin{align*}
\P\left(\frac{1/n}{p_{(1)}} \left(\frac{(n-1)/n}{1-p_{(1)}} \right)^{n-1} > e^s, (1-p_{(1)})^{n-1} > \frac{1}{2} \right) \leq \P\left(1/n\cdot e^{-s} \frac{(1-1/n)^{n-1}}{1/2} > p_{(1)}  \right).
\end{align*}
Since $p_{(1)} \sim \mathrm{Beta}(1,n)$, and since $(1-1/n)^{n-1}<1/2$ for $n > 2$, this probability is bounded by,
\begin{align*}
\P\left(1/n\cdot e^{-s} \frac{(1-1/n)^{n-1}}{1/2} > p_{(1)}  \right) &\leq \int_{0}^{\frac{1}{n}e^{-s}} n (1-x)^{n-1} dx \leq e^{-s},
\end{align*}
since $(1-x)^{n-1}<1$ on the domain of integration. To bound (\ref{betapc2}), notice $(1-p_{(1)})^{n-1} \leq \frac{1}{2}$ is equivalent to $1-\left(\frac{1}{2}\right)^{\frac{1}{n-1}} \leq p_{(1)}$. Thus (\ref{betapc2}) is equal to
\begin{align*}
&= \P\left(\frac{1/n}{p_{(1)}} \left(\frac{(n-1)/n}{1-p_{(1)}} \right)^{n-1} > e^s, 1-2^{-\frac{1}{n-1}} \leq p_{(1)} \right) \\
&\leq \P\left(\frac{1/n}{1-2^{-\frac{1}{n-1}}} \cdot e^{-s} ((n-1)/n)^{n-1} > (1-p_{(1)})^{n-1} \right) \\
&= \P\left(p_{(1)} > 1-\left(\frac{1/n}{1-2^{-\frac{1}{n-1}}} \cdot e^{-s}\right)^{\frac{1}{n-1}}\cdot  \frac{n-1}{n} \right) \\
&= \int_{L}^1 n(1-x)^{n-1}dx,
\end{align*}
where $L\coloneqq 1-\left(\frac{1/n}{1-2^{-\frac{1}{n-1}}} \cdot e^{-s}\right)^{\frac{1}{n-1}}\cdot  \frac{n-1}{n}$. Evaluating the integral gives
\begin{align*}
\int_{L}^1 n(1-x)^{n-1}dx &= (1-L)^{n} \\
&= \left(\frac{1/n}{1-2^{-\frac{1}{n-1}}} \right)^{\frac{n}{n-1}}\cdot  \left(\frac{n-1}{n}\right)^{n} \cdot e^{- \frac{sn}{n-1}} \\
&< \left(\frac{1/n}{1-2^{-\frac{1}{n-1}}} \right)^{\frac{n}{n-1}}\cdot \left(\frac{n-1}{n}\right)^{n}\cdot  e^{- s} .
\end{align*}
It is straightforward to show that $\left(\frac{1/n}{1-2^{-\frac{1}{n-1}}} \right)^{\frac{n}{n-1}}\cdot \left(\frac{n-1}{n}\right)^{n} < 9/e$ for $n > 2$. Combining the bounds on (\ref{betapc1}) and (\ref{betapc2}), we obtain the desired estimate.

\end{proof}

\begin{lemma}
\label{closegrid}
The grid $\T$ is defined
\begin{align*}
\T &\coloneqq \left\{\lfloor (1+\d)^0\rfloor,\lfloor (1+\d)^1\rfloor,\dots,\lfloor (1+\d)^{\log_{1+\d}\frac{n}{2}}\rfloor,\lfloor n-(1+\d)^{\log_{1+\d}\frac{n}{2}}\rfloor,\dots,\lfloor n-(1+\d)^0\rfloor \right\},
\end{align*}
with $\d = \frac{1}{\log\log n} \to 0$. For $\theta \in \Theta_1^{\mathrm{2-side}}(p,n,\rho,s)$ with true changepoint location $t^* \leq n/2$, and changepoint size $\rho > 0$ defined,
\begin{align*}
\rho \coloneqq \sqrt{\frac{t^*(n-t^*)}{n}}|\mu_{j1}-\mu_{j2}|, \tag{see definition of $\Theta_1^{\mathrm{2-side}}$}
\end{align*}
there exists an element $\tilde{t} \in \T$ with $\frac{t^*}{1+\d} < \tilde{t} \leq t^*$, for which 
\begin{align*}
\sqrt{\frac{\tilde{t}(n-\tilde{t})}{n}}\left| \left(\bar{X}_{j,1:\tilde{t}}-\bar{X}_{j,\tilde{t}+1:n} \right) \right| \stackrel{(d)}{=} |\mu + Z|, \tag{$Z \sim N(0,1)$}
\end{align*}
where $\mu = (1+o(1)) \rho$, for every non-null row $j$. For $\theta \in \Theta_1^{\mathrm{1-side}}(p,n,\rho,s)$ with true changepoint location $t^*\leq n/2$ and changepoint size $\rho > 0$, there exists $\tilde{t} \in \T$ with $\frac{t^*}{1+\d}<\tilde{t}\leq t^*$, for which
\begin{align*}
\sqrt{\frac{\tilde{t}(n-\tilde{t})}{n}} \left(\bar{X}_{j,1:\tilde{t}}-\bar{X}_{j,\tilde{t}+1:n} \right) \sim N(\mu,1),
\end{align*}
where $\mu = (1+o(1))\rho$, for every non-null row $j$. 
\end{lemma}

\begin{proof}
First let $\theta \in \Theta_1^{\mathrm{2-side}}(p,n,\rho,s)$ with true changepoint location $t^* \leq n/2$. The existence of $\tilde{t}$ satisfying $\frac{t^*}{1+\d} < \tilde{t} \leq t^*$ follows from the definition of the grid $\T$ and the assumption $t^* \leq n/2$. Since the normalized contrast $\sqrt{\frac{\tilde{t}(n-\tilde{t})}{n}} \big(\bar{X}_{j,1:\tilde{t}}-\bar{X}_{j,\tilde{t}+1:n} \big)$ is a linear combination of normal variables each with variance 1, it also has variance 1. Hence, it suffices to compute its mean under the alternative parameter $\theta$. Let $j \in [p]$ denote the index of a non-null row in $\theta$. Then since $\tilde{t} \leq t^*$, we have
\begin{align*}
\E_\theta \sqrt{\frac{\tilde{t}(n-\tilde{t})}{n}} \left(\bar{X}_{j,1:\tilde{t}}-\bar{X}_{j,\tilde{t}+1:n} \right) &= \sqrt{\frac{\tilde{t}(n-\tilde{t})}{n}} \left( \mu_{j1}-\frac{1}{n-\tilde{t}}(\mu_{j1}(t^*-\tilde{t})+\mu_{j2}(n-t^*))\right) \\
&= \sqrt{\frac{\tilde{t}(n-\tilde{t})}{n}} \cdot \frac{n-t^*}{n-\tilde{t}} \cdot (\mu_{j1}-\mu_{j2}) \\
&= \sqrt{\frac{\tilde{t}(n-\tilde{t})}{n}} \cdot \frac{n-t^*}{n-\tilde{t}} \cdot \sqrt{\frac{n}{t^*(n-t^*)}}\cdot \rho\cdot \mathrm{sign}(\mu_{j1}-\mu_{j2}),
\end{align*}
from which it follows that 
\begin{align*}
 \sqrt{\frac{\tilde{t}(n-\tilde{t})}{n}} \left|\bar{X}_{j,1:\tilde{t}}-\bar{X}_{j,\tilde{t}+1:n} \right| \stackrel{(d)}{=} \left|\sqrt{\frac{\tilde{t}(n-\tilde{t})}{n}} \cdot \frac{n-t^*}{n-\tilde{t}} \cdot \sqrt{\frac{n}{t^*(n-t^*)}}\cdot \rho +Z \right|,
\end{align*}
since the distribution of $Z\sim N(0,1)$ is symmetric. Then it suffices to show that
\begin{align}
\label{ntnt}
\lim_{n\to\infty} \sqrt{\frac{\tilde{t}(n-\tilde{t})}{n}} \cdot \frac{n-t^*}{n-\tilde{t}} \cdot \sqrt{\frac{n}{t^*(n-t^*)}} = 1.
\end{align}
Observe that $\tilde{t}$ and $t^*$ satisfy $\frac{1}{1+\d} < \frac{\tilde{t}}{t^*} \leq 1$, which implies the above limit, by the choice of $\d = \frac{1}{\log\log n} \to 0$ and the assumption $t^* \leq n/2$.

For $\theta \in \Theta_1^{\mathrm{1-side}}(p,n,\rho,s)$, the same calculation as above gives
\begin{align*}
\E_\theta \sqrt{\frac{\tilde{t}(n-\tilde{t})}{n}} \left(\bar{X}_{j,1:\tilde{t}}-\bar{X}_{j,\tilde{t}+1:n} \right) &= \sqrt{\frac{\tilde{t}(n-\tilde{t})}{n}} \cdot \frac{n-t^*}{n-\tilde{t}} \cdot \sqrt{\frac{n}{t^*(n-t^*)}}\cdot \rho,
\end{align*}
since $\rho = \sqrt{\frac{t^*(n-t^*)}{n}} (\mu_{j1}-\mu_{j2})$ for the non-null row with index $j$. Then the conclusion follows from (\ref{ntnt}).
\end{proof}

\begin{lemma}
\label{kbehave}
Let $x,t \in (0,1)$ and denote $K(x,t) \coloneqq x\log \frac{x}{t} + (1-x)\log\frac{1-x}{1-t}$. The following inequalities hold,
\begin{enumerate}
\item $K(x,t) \geq 2(x-t)^2$ for all $x,t \in (0,1)$
\item $K(x,t) \geq \frac{(x-t)^2}{9t}$ when $\frac{x}{t} \leq 4$.
\item $K(x,t) \geq \frac{1}{2} x \log \frac{x}{t}$ when $e^2\leq \frac{x}{t}$ and $x\leq \frac{1}{2}$.
\end{enumerate}
\end{lemma}

\begin{proof}
The first inequality follows from Pinsker's inequality,
\begin{align*}
\frac{1}{2} D(P\| Q) \geq \TV(P,Q)^2,
\end{align*}
applied to $P = \mathrm{Bern}(x)$ and $Q= \mathrm{Bern}(t)$. The second inequality follows from the inequality,
\begin{align*}
D(P \| Q) \geq H^2(P,Q),
\end{align*}
together with the identity $(\sqrt{x}-\sqrt{t})(\sqrt{x}+\sqrt{t}) = x-t$, since
\begin{align*}
H^2(P,Q) &\coloneqq \E_Q \big(\sqrt{P/Q}-1 \big)^2 \\
&= t(\sqrt{x/t}-1)^2+(1-t)(\sqrt{(1-x)/(1-t)}-1)^2 \\
&\geq  (\sqrt{x}-\sqrt{t})^2 \\
&= \frac{(x-t)^2}{(\sqrt{x}+\sqrt{t})^2} \\
&\geq \frac{(x-t)^2}{9 t} \tag{$x\leq 4t$}.
\end{align*}
For the third inequality, using $\log(1-x)\geq -2x$ when $x<1/2$, and the inequality $\log(1-t) \leq -t$ for all $t \in \R$, we have
\begin{align*}
(1-x) \log\frac{1-x}{1-t} &= (1-x)(\log(1-x)-\log(1-t)) \\
&\geq \frac{1}{2}(t-2x) \\
&= -\frac{1}{2}x\left(2-\frac{t}{x}\right).
\end{align*}
Using $e^2 \leq \frac{x}{t}$, we have $0\leq 2-\frac{t}{x} \leq 2 \leq \log \frac{x}{t}$, so that $-\frac{1}{2}x\left(2-\frac{t}{x} \right) \geq -\frac{1}{2}x\log\frac{x}{t}$, which implies the result.
\end{proof}

\begin{lemma}[Lemma 8.1 of \cite{birge2001alternative}]
\label{lem:birge}
    Let $Y \sim \chi^2_{d,\lambda}$ be a non-central chi-squared random variable with d degrees of freedom and non-centrality parameter $\lambda \geq 0$. Then, for any $x > 0$, we have
    \begin{align*}
        \P(Y \geq (d+\lambda) + 2\sqrt{x(d+2\lambda)} + 2x) \leq e^{-x} \\
        \P(Y \leq (d+\lambda)-2\sqrt{x(d+2\lambda)}) \leq e^{-x}.
    \end{align*}
    Additionally, we have
    \begin{align}
    \label{lem:chi-sq-tightness}
        \P(Y \geq \lambda + 2\sqrt{2x\lambda}+2x) \geq e^{-x(1+o(1))}
    \end{align}
    as $x \to \infty$.
\end{lemma}

\begin{proof}[Proof of \eqref{lem:chi-sq-tightness}]

Let $U \sim N(\sqrt{\lambda},1)$ and $Z_1,\dots,Z_{d-1} \sim N(0,1)$ independently. Then 
\begin{align*}
    Y \stackrel{(d)}{=} U^2 + \sum_{i=2}^d Z_i^2.
\end{align*}
It follows that
\begin{align*}
    \P(Y \geq \lambda + 2\sqrt{2x \lambda} +2x) &\geq \P(U^2 \geq \lambda + 2\sqrt{2x \lambda} +2x).
\end{align*}
Applying Mill's ratio gives
\begin{align*}
    \P(|\sqrt{\lambda}+Z_1| \geq \sqrt{\lambda}+\sqrt{2x}) \geq \P(Z_1 \geq \sqrt{2x}) \sim \frac{1}{\sqrt{2x}} \; e^{-x} = e^{-x(1+o(1))}
\end{align*}
as $x \to \infty$.
\end{proof}


\end{document}